\documentclass
[
    fontsize = 11 pt,        
    american,                
    captions = tableheading, 
    numbers = noenddot,      
    footheight = 35 pt,      
    abstracton,
    letterpaper,             
]
{scrartcl}

\usepackage[T1]{fontenc}
\usepackage[utf8]{inputenc}
\usepackage
[
    babel = true, 
]
{microtype}           
\usepackage{csquotes} 

\usepackage[dvipsnames]{xcolor} 

\definecolor{stroke1}{HTML}{2574A9} 

\usepackage{amsmath}
\usepackage{amssymb}
\usepackage{amsthm}
\usepackage{thmtools}
\usepackage{mathtools}
\usepackage{thm-restate}
\usepackage{dsfont}        
\usepackage{braceMnSymbol} 

\usepackage
[
    ttscale = 0.85, 
]
{libertine} 
\usepackage
[
    libertine,    
    slantedGreek, 
    vvarbb,       
    libaltvw,     
]
{newtxmath} 
\usepackage{url} 
\usepackage{bm}  

\usepackage{graphicx} 
\usepackage
[
    subrefformat = simple, 
    labelformat = simple,  
]
{subcaption}         

\usepackage{array}     
\usepackage{booktabs}  
\usepackage{longtable} 
\usepackage{pdflscape} 

\usepackage[shortlabels]{enumitem} 

\usepackage
[
    ruled,         
    vlined,        
    linesnumbered, 
]
{algorithm2e} 

\usepackage{xspace}   
\usepackage
[
    shortcuts, 
]
{extdash}             
\usepackage{setspace} 

\usepackage{xparse}    
\usepackage{footnote}  
\usepackage{afterpage} 
\usepackage
[
    textsize = scriptsize, 
]
{todonotes}            

\usepackage
[
    sortcites,              
    style = numeric,        
    safeinputenc,           
    backref = true,         
    backrefstyle = three,   
    hyperref = true,        
    maxbibnames = 99,       
    maxcitenames = 2,       
]
{biblatex} 

\DefineBibliographyStrings{english}%
{%
    backrefpage  = {\lowercase{s}ee page}, 
    backrefpages = {\lowercase{s}ee pages} 
}

\usepackage{mfirstuc}
\usepackage{algorithmicx}
\usepackage{amsfonts}
\usepackage
[
    labelfont = bf,
    format = plain,
]
{caption}

\usepackage[american]{babel}

\usepackage
[
    bookmarks = true,                 
    bookmarksopen = false,            
    bookmarksnumbered = true,         
    pdfstartpage = 1,                 
    colorlinks = true,                
    allcolors = stroke1,              
]
{hyperref} 

\usepackage
[
    noabbrev,   
    nameinlink, 
]
{cleveref} 

\date{}

\makeatletter
    \def\IfEmptyTF#1%
    {%
        \if\relax\detokenize{#1}\relax%
            \expandafter\@firstoftwo%
        \else%
            \expandafter\@secondoftwo%
        \fi%
    }
\makeatother

\NewDocumentCommand{\mathOrText}{m}
{%
    \ensuremath{#1}\xspace%
}

\let\originalleft\left
\let\originalright\right
\renewcommand{\left}{\mathopen{}\mathclose\bgroup\originalleft}
\renewcommand{\right}{\aftergroup\egroup\originalright}

\makeatletter
    \DeclareRobustCommand{\bfseries}%
    {%
        \not@math@alphabet\bfseries\mathbf%
        \fontseries\bfdefault\selectfont%
        \boldmath%
    }

\xspaceaddexceptions{]\}}

\allowdisplaybreaks

\crefname{ineq}{inequality}{inequalities}
\creflabelformat{ineq}{#2{\upshape(#1)}#3}

\crefname{term}{term}{terms}
\creflabelformat{term}{#2{\upshape(#1)}#3}

\crefname{cond}{condition}{conditions}
\creflabelformat{term}{#2{\upshape(#1)}#3}

\crefname{assume}{assumption}{assumptions}
\creflabelformat{term}{#2{\upshape(#1)}#3}

\let\oldfootnote\footnote

\newlength{\spaceBeforeFootnote} 
\newlength{\spaceAfterFootnote}  

\RenewDocumentCommand{\footnote}{o o o m}%
{%
    \IfNoValueTF{#1}%
    {%
        \oldfootnote{#4}%
    }%
    {%
        \setlength{\spaceBeforeFootnote}{\IfEmptyTF{#1}{0}{#1} em}%
        \IfNoValueTF{#2}%
        {%
            \hspace*{\spaceBeforeFootnote}\oldfootnote{#4}%
        }%
        {%
            \setlength{\spaceAfterFootnote}{\IfEmptyTF{#2}{0}{#2} em}%
            \hspace*{\spaceBeforeFootnote}\IfNoValueTF{#3}{\oldfootnote{#4}}{\oldfootnote[#3]{#4}}\hspace*{\spaceAfterFootnote}%
        }%
    }%
}

\makesavenoteenv{figure}
\makesavenoteenv{table}
\makesavenoteenv{tabular}

\declaretheoremstyle
[
   	spaceabove = \topsep,
   	spacebelow = \topsep,
   	headfont = \bfseries,
   	headformat = \textcolor{stroke1}{$\blacktriangleright$} \NAME~\NUMBER \NOTE,
   	notefont = \bfseries,
   	notebraces = {(}{)},
   	bodyfont = \normalfont,
   	postheadspace = 0.5 em,
   	qed = \textcolor{stroke1}{\bfseries$\blacktriangleleft$},
]
{myTheoremStyle}

\declaretheorem
[
   	style = myTheoremStyle,
   	name = Lemma,
    sharenumber = conjecture,
]
{lemma}
\declaretheorem
[
   	style = myTheoremStyle,
   	name = Corollary,
    sharenumber = conjecture,
]
{corollary}
\declaretheorem
[
   	style = myTheoremStyle,
   	name = Theorem,
    sharenumber = conjecture,
]
{theorem}
\declaretheorem
[
   	style = myTheoremStyle,
   	name = Definition,
    sharenumber = conjecture,
]
{definition}

\declaretheorem
[
    style = myTheoremStyle,
    name = Observation,
    sharenumber = conjecture,
]
{observation}
\declaretheorem
[
    style = myTheoremStyle,
    name = Condition,
    sharenumber = conjecture,
]
{condition}

\NewDocumentCommand{\functionTemplate}{m m m m o}%
{%
    \IfNoValueTF{#5}%
    {%
        \mathOrText{#1\left#2{#4}\right#3}%
    }%
    {%
        \mathOrText{#1#5#2{#4}#5#3}%
    }%
}

\newcommand*{\leftBracketType}{(}
\newcommand*{\rightBracketType}{)}

\NewDocumentCommand{\createFunction}{m m o o}%
{%
    \renewcommand*{\leftBracketType}{\IfNoValueTF{#3}{(}{#3}}%
    \renewcommand*{\rightBracketType}{\IfNoValueTF{#4}{)}{#4}}%
    \NewDocumentCommand{#1}{o o}%
    {%
        \IfNoValueTF{##1}%
        {%
            \mathOrText{#2}%
        }%
        {%
            \functionTemplate{#2}{\leftBracketType}{\rightBracketType}{##1}[##2]%
        }%
    }%
}

\DeclareDocumentCommand{\probabilisticFunctionTemplate}{m m O{} o}
{%
    \functionTemplate{#1}%
    {\lbrack}%
    {\rbrack}%
    {#2\IfEmptyTF{#3}{}{\ \IfNoValueTF{#4}{\left}{#4}\vert\ \vphantom{#2}#3\IfNoValueTF{#4}{\right.}{}}}%
    [#4]%
}

\newcommand*{\N}{\mathOrText{\mathds{N}}}
\newcommand*{\Z}{\mathOrText{\mathds{Z}}}

\newcommand*{\R}{\mathOrText{\mathds{R}}}
\newcommand*{\C}{\mathOrText{\mathds{C}}}
\newcommand*{\indicatorFunctionSymbol}{\mathds{1}}

\RenewDocumentCommand{\Pr}{m O{} o}%
{%
    \probabilisticFunctionTemplate{\mathrm{Pr}}{#1}[#2][#3]%
}

\NewDocumentCommand{\E}{m O{} o}%
{%
    \probabilisticFunctionTemplate{\mathrm{E}}{#1}[#2][#3]%
}

\NewDocumentCommand{\Var}{m O{} o}%
{%
    \probabilisticFunctionTemplate{\mathrm{Var}}{#1}[#2][#3]%
}

\DeclareDocumentCommand{\bigO}{m o}%
{%
    \functionTemplate{\mathrm{O}}{(}{)}{#1}[#2]%
}

\DeclareDocumentCommand{\smallO}{m o}%
{%
    \functionTemplate{\mathrm{o}}{(}{)}{#1}[#2]%
}

\DeclareDocumentCommand{\bigTheta}{m o}%
{%
    \functionTemplate{\upTheta}{(}{)}{#1}[#2]%
}

\DeclareDocumentCommand{\bigOmega}{m o}%
{%
    \functionTemplate{\upOmega}{(}{)}{#1}[#2]%
}

\DeclareDocumentCommand{\smallOmega}{m o}%
{%
    \functionTemplate{\upomega}{(}{)}{#1}[#2]%
}

\DeclareDocumentCommand{\eulerE}{o}%
{%
    \mathOrText{\mathrm{e}\IfNoValueTF{#1}{}{^{#1}}}%
}

\DeclareDocumentCommand{\poly}{m o}%
{%
    \functionTemplate{\mathrm{poly}}{(}{)}{#1}[#2]%
}

\createFunction{\id}{\mathrm{id}}

\NewDocumentCommand{\ind}{m o o}%
{%
    \IfNoValueTF{#2}%
    {%
        \mathOrText{\indicatorFunctionSymbol_{#1}}%
    }%
    {%
        \functionTemplate{\indicatorFunctionSymbol_{#1}}{(}{)}{#2}[#3]%
    }%
}

\DeclareDocumentCommand{\dom}{m o}%
{%
    \functionTemplate{\mathrm{dom}}{(}{)}{#1}[#2]%
}

\DeclareDocumentCommand{\rng}{m o}%
{%
    \functionTemplate{\mathrm{rng}}{(}{)}{#1}[#2]%
}

\DeclareDocumentCommand{\d}{o}%
{%
    \mathrm{d}\IfNoValueTF{#1}{}{^{#1}}%
}

\DeclareDocumentCommand{\set}{m m o}%
{%
    \mathOrText{\IfNoValueTF{#3}{\left}{#3}\{#1\ \IfNoValueTF{#3}{\left}{#3}\vert\ \vphantom{#1}#2\IfNoValueTF{#3}{\right.}{}\IfNoValueTF{#3}{\right}{#3}\}}%
}

\newcommand\sbullet[1][.6]{\mathbin{\vcenter{\hbox{\scalebox{#1}{$\bullet$}}}}}
\newcommand*{\separator}{$\sbullet$\xspace}

\newcommand*{\mathbold}[1]{\textbf{\mathOrText{#1}}}

\newcommand*{\indicator}[1]{\mathOrText{\indicatorFunctionSymbol{\left\{#1\right\}}}}

\newcommand*{\defeq}{\mathOrText{\coloneqq}}

\newcommand*{\symDiff}{\mathOrText{\oplus}}

\newcommand*{\subgraph}[2]{\mathOrText{#1 \left[ #2 \right]}}

\newcommand*{\graph}{\mathOrText{G}}
\newcommand*{\edges}{\mathOrText{E}}
\newcommand*{\vertices}{\mathOrText{V}}
\newcommand*{\numberOfVertices}{\mathOrText{n}}
\newcommand*{\degree}{\mathOrText{\Delta}}

\DeclareDocumentCommand{\neighbors}{m o}%
{%
	\mathOrText{N\IfNoValueTF{#2}{}{_{#2}}\left(#1\right)}%
}
\DeclareDocumentCommand{\neighborsClosed}{m o}%
{%
	\mathOrText{N\IfNoValueTF{#2}{}{_{#2}}\left[#1\right]}%
}

\newcommand*{\singleWeight}{\mathOrText{\lambda}}
\newcommand*{\weights}{\mathbold{\singleWeight}}
\newcommand*{\criticalWeight}[1]{\mathOrText{\singleWeight_{\text{c}}\left(#1\right)}}
\newcommand*{\gibbs}{\mathOrText{\mu}}
\newcommand*{\prt}{\mathOrText{Z}}
\newcommand*{\maxPartition}{\mathOrText{\prt_{\max}}}
\newcommand*{\independentSet}{\mathOrText{I}}

\newcommand*{\numberOfCliques}{\mathOrText{m}}
\newcommand*{\numberOfBlocks}{\mathOrText{\numberOfCliques}}

\newcommand*{\powerset}[1]{\mathOrText{2^{#1}}}
\newcommand*{\dtv}[2]{\mathOrText{d_{\text{TV}}\big(#1, #2\big)}}
\newcommand*{\size}[1]{\mathOrText{\left| #1 \right|}}
\newcommand*{\absolute}[1]{\mathOrText{\left| #1 \right|}}
\DeclareDocumentCommand{\norm}{m o}%
{%
	\mathOrText{\left\lVert #1 \right\rVert\IfNoValueTF{#2}{}{_{#2}}}%
}

\newcommand*{\err}{\mathOrText{\varepsilon}}

\DeclareDocumentCommand{\mix}{m m o}%
{%
	\mathOrText{\tau_{#1}\IfNoValueTF{#3}{}{^{(#3)}}\left(#2\right)}
}

\newcommand*{\pmc}{clique dynamics condition\xspace}

\newcommand*{\spmc}{strict clique dynamics condition\xspace}
\newcommand*{\Spmc}{\xmakefirstuc{\spmc}}
\newcommand*{\spmcConstant}{\mathOrText{\alpha}}
\newcommand*{\spmcFunctionSymbol}{\mathOrText{f}}
\DeclareDocumentCommand{\spmcFunction}{o}%
{%
	\mathOrText{\spmcFunctionSymbol\IfNoValueTF{#1}{}{\left(#1\right)}}%
}

\DeclareDocumentCommand{\independentSets}{o o}%
{%
    \mathOrText{\mathcal{I}\IfNoValueF{#1}%
        {%
            \IfNoValueF{#2}{_{\vert#2}}\IfEmptyTF{#1}{}{\left(#1\right)}%
        }%
    }%
}

\DeclareDocumentCommand{\spinConfig}{o o o}%
{%
	\mathOrText{\sigma\IfNoValueF{#1}%
		{%
			\IfEmptyTF{#1}{}{_{\vert#1}}
			\IfNoValueF{#3}{^{(#3)}} 
			\IfNoValueF{#2}{\IfEmptyTF{#2}{}{\left(#2\right)}}%
		}%
	}%
}

\DeclareDocumentCommand{\zeroSpinConfig}{o o}%
{%
	\mathOrText{\mathbold{0}\IfNoValueF{#1}%
		{%
			\IfEmptyTF{#1}{}{_{\vert#1}}
			\IfNoValueF{#2}{\left(#2\right)}%
		}%
	}%
}

\DeclareDocumentCommand{\gibbsPr}{m m O{} o}%
{%
    \probabilisticFunctionTemplate{\mathds{P}_{#1}}{#2}[#3][#4]%
}

\newcommand*{\inSet}[1]{\mathOrText{#1}}
\newcommand*{\notInSet}[1]{\mathOrText{\overline{#1}}}

\DeclareDocumentCommand{\partitionFunction}{o}%
{%
    \mathOrText{\prt\IfNoValueF{#1}%
    	{%
            \IfEmptyTF{#1}{}{\left(#1\right)}%
    	}%
    }%
}

\DeclareDocumentCommand{\gibbsDistribution}{o o}%
{%
\mathOrText{\gibbs\IfNoValueF{#1}%
    {%
        \IfNoValueF{#2}{\IfEmptyTF{#2}{}{_{\vert#2}}}\IfEmptyTF{#1}{}{^{(#1)}}%
    }%
}%
}

\DeclareDocumentCommand{\gibbsDistributionFunction}{m o o o}%
{%
    \IfNoValueTF{#2}%
    {%
        \gibbsDistribution%
    }%
    {%
        \IfNoValueTF{#3}{\gibbsDistribution[#2]}{\gibbsDistribution[#2][#3]}%
    }%
    \left(#1\IfNoValueF{#4}{\middle\vert\ #4}\right)%
}

\DeclareDocumentCommand{\weight}{o}%
{%
	\mathOrText{\lambda\IfNoValueTF{#1}{}{_{#1}}}%
}

\DeclareDocumentCommand{\block}{o}%
{%
	\mathOrText{\clique\IfNoValueTF{#1}{}{_{#1}}}%
}

\DeclareDocumentCommand{\clique}{o}%
{%
    \mathOrText{\Lambda\IfNoValueTF{#1}{}{_{#1}}}%
}

\DeclareDocumentCommand{\abstractDynamics}{m o}%
{%
	\mathOrText{#1\IfNoValueTF{#2}{}{\left(#2\right)}}%
}

\newcommand*{\markovSymbol}{\mathOrText{\mathcal{M}}}
\DeclareDocumentCommand{\markov}{o}%
{%
	\mathOrText{\IfNoValueTF{#1}{\abstractDynamics{\markovSymbol}}{\abstractDynamics{\markovSymbol}[#1]}}%
}

\newcommand*{\cliqueDynamicsSymbol}{\mathOrText{\mathcal{C}}}
\DeclareDocumentCommand{\cliqueDynamics}{o}%
{%
	\mathOrText{\IfNoValueTF{#1}{\abstractDynamics{\cliqueDynamicsSymbol}}{\abstractDynamics{\cliqueDynamicsSymbol}[#1]}}%
}

\newcommand*{\blockDynamicsSymbol}{\mathOrText{\mathcal{B}}}
\DeclareDocumentCommand{\blockDynamics}{o}%
{%
	\mathOrText{\IfNoValueTF{#1}{\abstractDynamics{\blockDynamicsSymbol}}{\abstractDynamics{\blockDynamicsSymbol}[#1]}}%
}

\newcommand*{\transitionSymbol}{\mathOrText{P}}
\DeclareDocumentCommand{\transitionMatrix}{o o}%
{%
	\mathOrText{\transitionSymbol\IfNoValueF{#1}%
		{%
			\IfEmptyTF{#1}{}{_{#1}}%
			\IfNoValueF{#2}{\left( #2 \right)}
		}%
	}
}

\DeclareDocumentCommand{\stateSpace}{o}%
{%
	\mathOrText{\Omega\IfNoValueTF{#1}{}{_{#1}}}%
}

\DeclareDocumentCommand{\stationary}{o o}%
{%
	\mathOrText{\pi\IfNoValueF{#1}%
		{%
			\IfEmptyTF{#1}{}{_{#1}}%
			\IfNoValueF{#2}{\left( #2 \right)}
		}%
	}
}

\newcommand*{\complexGroundset}{\mathOrText{U}}
\newcommand*{\complexPartition}[1]{\mathOrText{\complexGroundset_{#1}}}

\newcommand*{\complexFace}{\mathOrText{\tau}}

\newcommand*{\faceFromIndep}[1]{\mathOrText{\complexFace_{#1}}}

\newcommand*{\indepFromFace}[1]{\mathOrText{\independentSet_{#1}}}

\newcommand*{\complexElement}[1]{\mathOrText{x_{#1}}}

\newcommand*{\complexEmpty}[1]{\mathOrText{\emptyset_{#1}}}

\DeclareDocumentCommand{\complex}{o o}%
{%
	\mathOrText{X\IfNoValueF{#1}%
		{%
			\IfNoValueF{#2}{_{#2}}
			\IfEmptyTF{#1}{}{\left( #1 \right)}%
		}%
	}%
}

\DeclareDocumentCommand{\complexWeight}{o o}%
{%
	\mathOrText{w\IfNoValueF{#1}%
		{%
			\IfNoValueF{#2}{_{#2}}
			\IfEmptyTF{#1}{}{\left( #1 \right)}%
		}%
	}%
}

\newcommand*{\twoStepSymbol}{\mathOrText{\mathcal{V}}}
\DeclareDocumentCommand{\twoStep}{o}%
{
	\mathOrText{\IfNoValueTF{#1}{\abstractDynamics{\twoStepSymbol}}{\abstractDynamics{\twoStepSymbol}[#1]}}%
}

\newcommand*{\skeletonWalkSymbol}{\mathOrText{\mathcal{S}}}
\DeclareDocumentCommand{\skeletonWalk}{o}%
{%
	\mathOrText{\IfNoValueTF{#1}{\abstractDynamics{\skeletonWalkSymbol}}{\abstractDynamics{\skeletonWalkSymbol}[#1]}}%
}

\newcommand*{\sawRoot}{\mathOrText{r}}

\DeclareDocumentCommand{\sawTree}{o o}%
{%
	\mathOrText{\mathds{T}\IfNoValueF{#1}%
		{%
			\IfNoValueF{#2}{^{#2}}
			\IfEmptyTF{#1}{}{\left( #1 \right)}%
		}%
	}%
}

\DeclareDocumentCommand{\sawCopies}{m o}%
{%
	\mathOrText{C\IfNoValueTF{#2}{}{_{#2}}\left(#1\right)}%
}

\newcommand*{\sawCopy}[1]{\mathOrText{\widehat{#1}}}

\newcommand*{\sawLayer}[2]{\mathOrText{L_{#1}\left(#2\right)}}

\DeclareDocumentCommand{\sawFunction}{o}%
{%
	\mathOrText{\sawCopy{\spmcFunctionSymbol}\IfNoValueTF{#1}{}{\left(#1\right)}}%
}

\newcommand*{\pairwiseInfluenceSymbol}{\mathOrText{\Psi}}
\DeclareDocumentCommand{\pairwiseInfluence}{m o o}%
{%
	\mathOrText{\pairwiseInfluenceSymbol_{#1}\IfNoValueF{#2}%
		{%
			\IfEmptyTF{#2}{}{^{#2}}
			\IfNoValueF{#3}{\left(#3\right)}%
		}%
	}
}

\newcommand*{\cliqueInfluenceSymbol}{\mathOrText{\Phi}}
\DeclareDocumentCommand{\cliqueInfluence}{m m o}%
{%
	\mathOrText{\cliqueInfluenceSymbol_{#1,#2}\IfNoValueTF{#3}{}{\left(#3\right)}
	}
}

\newcommand*{\eigenvalueSymbol}{\mathOrText{\beta}}

\DeclareDocumentCommand{\eigenvalue}{o}%
{%
	\mathOrText{\eigenvalueSymbol\IfNoValueTF{#1}{}{_{#1}}}%
}

\DeclareDocumentCommand{\eigenvalueOf}{m o}%
{%
	\mathOrText{\eigenvalueSymbol\IfNoValueTF{#2}{}{_{#2}}\left(#1\right)}%
}

\newcommand*{\transposed}[1]{\mathOrText{#1^{\text{T}}}}

\newcommand*{\maxAbsEigenvalue}[1]{\mathOrText{\eigenvalueSymbol^{*}\left(#1\right)}}

\newcommand*{\spectralRadius}[1]{\mathOrText{\rho\left(#1\right)}}

\newcommand*{\mcEdges}[1]{\mathOrText{E\left(#1\right)}}

\newcommand*{\mcAllEdges}[1]{\mathOrText{E^{*}\left(#1\right)}}

\DeclareDocumentCommand{\transitionWeight}{m o}%
{%
	\mathOrText{Q_{#1}\IfNoValueTF{#2}{}{\left(#2\right)}}%
}

\DeclareDocumentCommand{\congestion}{o}%
{%
	\mathOrText{\rho\IfNoValueTF{#1}{}{\left(#1\right)}}%
}

\newcommand*{\canonicalPaths}{\mathOrText{\Gamma}}

\DeclareDocumentCommand{\canonicalPath}{o}%
{%
	\mathOrText{\gamma\IfNoValueTF{#1}{}{_{#1}}}%
}

\newcommand*{\statRatio}{\mathOrText{a}}

\newcommand*{\flowRatio}[1]{\mathOrText{A\left(#1\right)}}

\newcommand*{\length}[1]{\mathOrText{\left\vert#1\right\vert}}

\newcommand*{\sidelength}{\mathOrText{\ell}}

\newcommand*{\volume}{\mathOrText{\mathds{V}}}

\newcommand*{\numParticles}{\mathOrText{k}}

\newcommand*{\hsFugacity}{\mathOrText{\lambda}}

\newcommand*{\hsGraph}[1]{\mathOrText{\graph_{#1}}}

\newcommand*{\hsVertices}[1]{\mathOrText{\vertices_{#1}}}

\newcommand*{\hsEdges}[1]{\mathOrText{\edges_{#1}}}

\newcommand*{\hsWeight}[1]{\mathOrText{\weight_{#1}}}

\newcommand*{\hsVertex}[1]{\mathOrText{v_{#1}}}

\newcommand*{\hsDegree}[1]{\mathOrText{\degree_{#1}}}

\newcommand*{\gaussCircleError}{\mathOrText{\gamma}}

\newcommand*{\radius}{\mathOrText{r}}

\DeclareDocumentCommand{\grid}{o}%
{%
	\mathOrText{\mathds{G}\IfNoValueTF{#1}{}{\left(#1\right)}}%
}

\DeclareDocumentCommand{\subgrid}{o}%
{%
	\mathOrText{\mathds{H}\IfNoValueTF{#1}{}{_{#1}}}%
}

\newcommand*{\resolution}{\mathOrText{\rho}}

\DeclareDocumentCommand{\hardSpherePrt}{o o}%
{%
	\mathOrText{\prt\IfNoValueF{#1}%
		{%
			\IfNoValueF{#2}{_{#2}}\IfEmptyTF{#1}{}{\left(#1\right)}%
		}%
	}%
} 

\DeclareDocumentCommand{\valid}{o o}%
{%
	\mathOrText{D\IfNoValueF{#1}%
		{%
			\IfNoValueF{#2}{_{#2}}\IfEmptyTF{#1}{}{\left(#1\right)}%
		}%
	}%
} 

\newcommand*{\normSphere}[1]{\mathOrText{v_{#1}}}

\newcommand*{\dist}[2]{\mathOrText{d \left(#1, #2 \right)}}

\DeclareDocumentCommand{\lebesgue}{m o}%
{%
	\mathOrText{\nu^{#1} \IfNoValueTF{#2}{}{\left(#2\right)}}%
}

\newcommand*{\vol}[1]{\mathOrText{\left\vert#1\right\vert}}

\DeclareDocumentCommand{\integerSphere}{m o}%
{%
	\mathOrText{b_{#1}\IfNoValueTF{#2}{}{\left(#2\right)}}%
}

\DeclareDocumentCommand{\rescale}{o o}%
{%
	\mathOrText{\varphi\IfNoValueF{#1}%
		{%
			\IfEmptyTF{#1}{}{^{(#1)}}\IfNoValueF{#2}{\left(#2\right)}%
		}%
	}%
} 

\DeclareDocumentCommand{\remap}{o o}%
{%
	\mathOrText{\Phi\IfNoValueF{#1}%
		{%
			\IfEmptyTF{#1}{}{^{(#1)}}\IfNoValueF{#2}{\left(#2\right)}%
		}%
	}%
} 

\newcommand*{\invremap}[2]{\mathOrText{\left( \Phi^{(#1)} \right)^{-1} \left( #2 \right)}}

\title{A spectral independence view on hard spheres via block dynamics}

\author{Tobias Friedrich$^{*}$ \and Andreas Göbel$^{*}$ \and Martin~S. Krejca$^{\dagger}$ \and Marcus Pappik$^{*}$}

\publishers
{%
    $^{*}$ \normalsize Hasso Plattner Institute, University of Potsdam, Potsdam, Germany\\
    \{tobias.friedrich, andreas.goebel, marcus.pappik\}@hpi.de \\[4 ex]
    
    $^{\dagger}$ Sorbonne Université, CNRS, LIP6, Paris, France \\
    martin.krejca@lip6.fr \\ 
    This work was supported by the Paris Île-de-France Region. 
}

\begin{document}

\maketitle
\thispagestyle{empty}
\vspace*{-4 ex}

\begin{abstract}
The hard-sphere model is one of the most extensively studied models in statistical physics.
It describes the continuous distribution of spherical particles, governed by hard-core interactions.
An important quantity of this model is the normalizing factor of this distribution, called the \emph{partition function}.
We propose a Markov chain Monte Carlo algorithm for approximating the grand-canonical partition function of the hard-sphere model in $d$ dimensions.
Up to a fugacity of $\hsFugacity < \eulerE/2^d$, the runtime of our algorithm is polynomial in the volume of the system.
This covers the entire known real-valued regime for the uniqueness of the Gibbs measure.

Key to our approach is to define a discretization that closely approximates the partition function of the continuous model.
This results in a discrete hard-core instance that is exponential in the size of the initial hard-sphere model.
Our approximation bound follows directly from the correlation decay threshold of an infinite regular tree with degree equal to the maximum degree of our discretization.
To cope with the exponential blow-up of the discrete instance we use clique dynamics, a Markov chain that was recently introduced in the setting of abstract polymer models.
We prove rapid mixing of clique dynamics up to the tree threshold of the univariate hard-core model.
This is achieved by relating clique dynamics to block dynamics and adapting the spectral expansion method, which was recently used to bound the mixing time of Glauber dynamics within the same parameter regime.
\end{abstract}

\vspace*{3 ex}
\hspace*{1.3 em}
\begin{minipage}{0.85\textwidth}
\end{minipage}

\newpage

\section{Introduction}
Statistical physics models particle systems as probability distributions.
One of the most fundamental and mathematically challenging models in this area is the hard-sphere model, which plays a central role in understanding the thermodynamic properties of monoatomic gases and liquids \cite{boublik1980statistical,LiquidsBook}.
It is a continuous model that studies the distribution and macroscopic behavior of indistinguishable spherical particles, assuming only hard-core interactions, i.e., no two particles can occupy the same space. 

We focus on computational properties of the \emph{grand-canonical ensemble} of the hard-sphere model in a finite $d$-dimensional cubic region $\volume = [0, \sidelength)^d$ in the Euclidean space.
In the grand-canonical ensemble, the system can exchange particles with its surrounding based on a fugacity parameter $\hsFugacity$, which is inverse to the temperature of the system.
For the rest of the paper, we make the common assumption that the system is normalized such that the particles have unit volume.
That means we fix their radii to $\radius = (1/ \normSphere{d})^{1/d}$, where $\normSphere{d}$ is the volume of a unit sphere in $d$ dimensions.

A simple probabilistic interpretation of the distribution of particles in the grand-canonical ensemble is that centers of points that are drawn from a Poisson point process on $\volume$ with intensity~$\hsFugacity$, conditioned on the event that no two particles overlap (i.e., every pair of centers has distance at least $2 \radius$).
The resulting distribution over particle configurations in $\volume$ is called the \emph{Gibbs distribution} of the model.
An important quantity in such models is the so called partition function  $\partitionFunction[\volume, \hsFugacity]$, which can be seen as the normalizing constant of the Gibbs distribution.
Formally, it is defined as
\[
	\hardSpherePrt[\volume, \hsFugacity] = 
	1 + \sum_{\numParticles \in \N_{> 0}}  
	\frac{\hsFugacity^\numParticles}{\numParticles !}
	\int_{\volume^\numParticles} \valid[x^{(1)}, \dots, x^{(\numParticles)}] \,\d \lebesgue{d \times \numParticles} ,
\]
where 
\[
	\valid[x^{(1)}, \dots, x^{(\numParticles)}] = \begin{cases}
		1 \emph{ if $\dist{x^{(i)}}{x^{(j)}} \ge 2r$ for all $i, j \in [\numParticles]$ with $i \neq j$} \\
		0 \emph{ otherwise} 
	\end{cases} 
\]
and $\lebesgue{d \times \numParticles}$ is the Lebesgue measure on $\R^{d \times \numParticles}$.
Commonly, two computational task are associated with the grand-canonical hard-sphere model:
(1) approximating its partition function $\partitionFunction[\volume, \hsFugacity]$, and (2) approximately sampling from the Gibbs distribution.

Studying computational aspects of the hard-sphere model carries a historical weight, as in the seminal work of Metropolis \cite{Metropolis}, the Monte Carlo method was introduced to investigate a two-dimensional hard-sphere model. 
Approximate-sampling Markov chain approaches have been mainly focused on the canonical ensemble of the model, that is, the system does not exchange particles with its surrounding and thus, the distribution is defined over a fixed number of spheres~\cite{HayesM14,Shperes2,JJP19}.
Considering the grand canonical ensemble, exact sampling algorithms have appeared in the literature for the two-dimensional model without asymptotic runtime guarantees~\cite{Kendall1998, kendall_2000, MJM17}. 
A result that is more aligned with theoretical computer science was given in \cite{GJ18}, where the authors introduced an exact sampling algorithm for the grand-canonical hard-sphere model in $d$-dimensions. 
Their algorithm is based on rejection sampling with a runtime linear in the volume of the system $\vol{\volume}$ when assuming a continuous computational model and access to a sampler from a continuous Poisson point process. 
Their approach is guaranteed to apply for $\hsFugacity<2^{-(d+1/2)}$.

Besides such sampling results, there is an ongoing effort to improve the known fugacity regime where the Gibbs measure is unique and correlations decay exponentially fast \cite{fernandez2007analyticity,christoph2019disagreement,HPP20,perkinsHardSpheres2020}.
Note that for many discrete spin systems, such as the hard-core model, correlation decay is closely related to the applicability of different methods for efficient approximation of the partition function \cite{2010:Sly:computational_transition,2014:Galanis:inapproximability_independent_hard_core,Weitz2006Counting}.
Recently, the correlation decay bounds for the hard-sphere model were improved in \cite{HPP20} to $\hsFugacity<2^{-(d-1)}$, using probabilistic arguments, and in \cite{perkinsHardSpheres2020} to $\hsFugacity < \eulerE/2^d$, based on an analytic approach.
A common feature of \cite{HPP20} and \cite{perkinsHardSpheres2020} is that they translated tools originally developed in theoretical computer science for investigating the discrete hard-core model to the continuous domain.

Our work is in line with the computational view on the hard-sphere model but more algorithmic in nature.
We investigate the range of the fugacity~$\hsFugacity$ for which an approximation of $\partitionFunction[\volume, \hsFugacity]$ can be obtained efficiently in terms of the volume of the system $\vol{\volume}$,  assuming a discrete computational model.
Our main result is that for all $\hsFugacity < \eulerE/2^d$ there is a randomized algorithm for $\err$-approximating the partition function in time polynomial in $\vol{\volume}$ and $1/\err$.
\begin{restatable}{theorem}{hardsphereApproximation}
	\label{thm:hard_sphere_approx}
	Let $(\volume, \hsFugacity)$ be an instance of the continuous hard-sphere model with $\volume = [0, \sidelength)^d$.
	If there is a $\delta \in (0, 1]$ such that
	\[
	\hsFugacity \le (1 - \delta) \frac{\eulerE}{2^d} ,
	\]
	then for each $\err \in (0, 1]$ there is a randomized $\err$-approximation of $\hardSpherePrt[\volume, \hsFugacity]$ computable in time polynomial in $\vol{\volume}^{1/\delta^2}$ and $\frac{1}{\err}$.	
\end{restatable}
Note that this bound on $\hsFugacity$ precisely coincides with the best known bound for the uniqueness of the Gibbs measure in the thermodynamic limit, recently established in \cite{perkinsHardSpheres2020}.
For many discrete spin systems, such as the hard-core model or general anti-ferromagnetic 2-state spin systems, the region of efficient approximation of the partition function is closely related to uniqueness of the Gibbs measure. 
More precisely, it can be shown that the partition function of every graph of maximum degree $\degree$ can be approximated efficiently if the corresponding Gibbs distribution on an infinite $\degree$ regular tree is unique \cite{li2013correlation,weitz2005combinatorial}.
A detailed discussion for the discrete hard-core model can be found in the next subsection.
In a sense, \Cref{thm:hard_sphere_approx} can be seen as the algorithmic counterpart of the recent uniqueness result for the continuous hard-sphere model. 
This answers an open question, asked in \cite{perkinsHardSpheres2020}.

The way we prove our result is quite contrary to \cite{HPP20} and \cite{perkinsHardSpheres2020}.
Instead of translating discrete tools from computer science into the continuous domain, we rather discretize the hard-sphere model. 
By this, existing algorithmic and probabilistic techniques for discrete models become available, and we avoid continuous analysis.
 
Our applied discretization scheme is fairly intuitive and results in an instance of the discrete hard-core model. This model has been extensively studied in the computer science community.
However, as this hard-core instance is exponential in the size of the continuous system $\vol{\volume}$, existing approaches for approximating its partition function, such a Markov chain Monte Carlo methods based on Glauber dynamics, are not feasible.
We overcome this problem by applying a Markov chain Monte Carlo approach based on clique dynamics, which were introduced in \cite{friedrich2020polymer} in the setting of abstract polymer models.
Previously known conditions for the rapid mixing of clique dynamics were developed for the multivariate version of the hard-core model.
Due to this generality, those conditions do not result in the desired bound in our univariate setting. Instead we relate those clique dynamics to another Markov chain, called \emph{block dynamics.}
We then prove the desired mixing time for the block dynamics by adapting a recently introduced technique for bounding the mixing time of Markov chains, based on local spectral expansion \cite{ALOG20}. 
Together with a known self-reducibility scheme for clique dynamics, this results in the desired approximation algorithm.

Note that we aim for a rigorous algorithmic result for approximating the partition function of the continuous hard-sphere model.
To be in line with commonly used discrete computational models, our Markov chain Monte Carlo algorithm does not assume access to a continuous sampler but instead samples approximately from a discretized version of the Gibbs distribution. 
Note that sampling from the continuous hard-sphere partition function cannot be done using a discrete computation model as this would involve infinite float pointer precision.
For practical matters, our discretization of the Gibbs distribution can be seen as an approximation of the original continuous Gibbs measure.
However, a rigorous comparison between both distributions based on total variation distance is not applicable, due to the fact that one is discrete whereas the other is continuous in nature.

Assuming access to a continuous sampler, we believe that our approach can be used to obtain an approximation of the Gibbs distribution of the continuous model within the same fugacity regime, by adding small perturbations to the drawn sphere centers. 
This would be in line with the relation between the mixing time of continuous heat-bath dynamics and strong spatial mixing, pointed out in \cite{HPP20}, combined with the uniqueness bound from \cite{perkinsHardSpheres2020}.

In \Cref{sec:intro_discrete,sec:intro_block,sec:intro_spectral} we discuss our technical contributions in more detail and explain how they relate to the existing literature. Finally in \Cref{sec:intro_outlook} we discuss possible extensions and future work.

\subsection{Discretization and hard-core model}\label{sec:intro_discrete}
Our discretization scheme expresses the hard-sphere partition function as the partition function of an instance of the (univariate) hard-core model.
An instance of the hard-core model is a tuple $(\graph, \weight)$ where $\graph = (\vertices, \edges)$ is an undirected graph and $\weight \in \R_{>0}$.
Its partition function is defined as
\[
	\partitionFunction[\graph, \weight] \defeq \sum_{\independentSet \in \independentSets[\graph]} \weight^{\size{\independentSet}} ,
\]
where $\independentSets[\graph]$ denotes the independent sets of $\graph$. 
A common way to obtain an approximation for the partition function is by applying a Markov chain Monte Carlo algorithm.
This involves sampling from the Gibbs distribution $\gibbsDistribution[\graph, \weight]$ of $(\graph, \weight)$, which is a probability distribution on $\independentSets[\graph]$ that assigns each independent set $\independentSet \in \independentSets[\graph]$ the probability
\[
	\gibbsDistributionFunction{\independentSet}[\graph, \weight] = \frac{\weight^{\size{\independentSet}}}{\partitionFunction[\graph, \weight]} .
\]

Conditions for efficient approximation of the hard-core partition function have been studied extensively in the theoretical computer science community.
Due to hardness results in \cite{2010:Sly:computational_transition} and \cite{2014:Galanis:inapproximability_independent_hard_core}, it is known that for general graphs of maximum degree $\degree \in \{3\} \cup \N_{>5}$ there is a critical parameter value $\criticalWeight{\degree} = (\degree - 1)^{\degree - 1}/(\degree - 2)^{\degree}$, such that there is no FPRAS for the partition function of $(\graph, \weight)$ for $\weight > \criticalWeight{\degree}$, unless $\mathrm{RP} = \mathrm{NP}$.
On the other hand, in \cite{Weitz2006Counting} it was proven that there is a deterministic algorithm for approximating the partition function of $(\graph, \weight)$ for $\weight < \criticalWeight{\degree}$ that runs in time $\size{\vertices}^{\bigO{\degree}}$.
The critical value $\criticalWeight{\degree}$ is especially interesting, as it precisely coincides with the upper bound on $\weight$ for which the hard-core model on an infinite $\degree$-regular tree exhibits strong spatial mixing and a unique Gibbs distribution \cite{Weitz2006Counting}.
For this reason, it is also referred to as the \emph{tree threshold}. 
This relation between computational hardness and phase transition in statistical physics is one of the most celebrated results in the area.
Both, the hardness results \cite{galanis2016inapproximability,bezkova2018inapprox} and the approximation algorithms \cite{patel2017deterministic,harvey2018computing} were later generalized for complex $\weight$.

Note that the computational hardness above the tree threshold $\criticalWeight{\degree}$ for general graphs of maximum degree $\degree$ applies to both, randomized and deterministic algorithms.
However, in the randomized setting, Markov chain Monte Carlo methods are known to improve the runtime of the algorithm introduced in \cite{Weitz2006Counting}.
Those approaches use the vertex-wise self-reducibility of the hard-core model to construct a randomized approximation of the partition function based on an approximate sampler for the Gibbs distribution.
Commonly, a Markov chain on the state space $\independentSets[\graph]$, called \emph{Glauber dynamics}, is used to construct the sampling scheme.
At each step, a vertex $v \in \vertices$ is chosen uniformly at random.
With probability~$\weight/(1 + \weight)$ the chain tries to add $v$ to the current independent set and otherwise it tries to remove it.
The resulting Markov chain is ergodic and reversible with respect to the Gibbs distribution, meaning that it eventually converges to $\gibbsDistribution[\graph, \weight]$.
A sequence of results has shown that for all $\degree \ge 3$ there is a family of graphs with maximum degree $\Delta$, such that the Glauber dynamics are torpidly mixing for $\weight > \criticalWeight{\degree}$, even without additional complexity-theoretical assumptions \cite{dyer2002counting,galvin2006slow,mossel2009hardness}.
Whether the Glauber dynamics are rapidly mixing for the entire regime $\weight < \criticalWeight{\degree}$ remained a long-standing open problem, until recently the picture was completed \cite{ALOG20}.
By relating spectral expansion properties of certain random walks on simplicial complexes to the Glauber dynamics, it was shown that the mixing time is polynomial in $\size{\vertices}$ below the tree threshold. 
The mixing time was recently further improved in \cite{ChenOptimal} for a broader class of spin systems by combining simplicial complexes with entropy factorization and using the modified log-Sobolev inequality.  

By mapping the hard-sphere model to an instance of the hard-core model we can make use of the existing results about approximation and sampling below the tree threshold.
Roughly, our discretization scheme restricts the positions of sphere centers to an integer grid, while scaling the radii of spheres and the fugacity appropriately.
For a hard-sphere instance $(\volume, \hsFugacity)$ with $\volume = [0, \sidelength)^d$ the hard-core representation for resolution $\resolution \in \R_{\ge 1}$ is a hard-core instance $(\hsGraph{\resolution}, \hsWeight{\resolution})$ with $\hsGraph{\resolution} = (\hsVertices{\resolution}, \hsEdges{\resolution})$.
Each vertex $v \in \hsVertices{\resolution}$ represents a grid point in the finite integer lattice of side length $\resolution \sidelength$.
Two distinct vertices in $\hsVertices{\resolution}$ are connected by an edge in $\hsEdges{\resolution}$ if the Euclidean distance of the corresponding grid points is less than $2 \resolution \radius$. 
Furthermore, we set $\hsWeight{\resolution} = \hsFugacity/\resolution^d$.
We provide the following result on the rate of convergence of $\partitionFunction[\hsGraph{\resolution}, \hsWeight{\resolution}]$ to the hard-sphere partition function $\hardSpherePrt[\volume, \hsFugacity]$ in terms of $\resolution$.
\begin{restatable}{lemma}{hardsphereConvergence}
	\label{lemma:hard_sphere_convergence}
	Let $(\volume, \hsFugacity)$ be an instance of the continuous hard-sphere model in $d$ dimensions.
	For each resolution $\resolution \ge 2 \sqrt{d}$ it holds that
	\[
		1 - \resolution^{-1} \eulerE^{\bigTheta{\vol{\volume} \ln\vol{\volume}}} 
		\le \frac{\hardSpherePrt[\volume, \hsFugacity]}{\partitionFunction[\hsGraph{\resolution}, \hsWeight{\resolution}]}
		\le 1 + \resolution^{-1} \eulerE^{\bigTheta{\vol{\volume} \ln\vol{\volume}}} .
		\qedhere
	\]
\end{restatable}
Although proving this rate of convergence involves some detailed geometric arguments, there is an intuitive reason why the partition functions converge eventually as $\resolution \to \infty$.
Increasing the resolution $\resolution$ also linearly increases the side length of the grid and the minimum distance that sphere centers can have.
This is equivalent to putting a grid into $\volume$ with increasing granularity but fixing the radii of spheres instead.
However, only scaling the granularity of this grid increases the number of possible configurations by roughly $\resolution^d$, which would cause the partition function of the hard-core model to diverge as $\resolution \to \infty$. 
To compensate for this, we scale the weight of each vertex in the hard-core model by the inverse of this factor.

Using this discretization approach, the fugacity bound from \Cref{thm:hard_sphere_approx} results from simply considering $\hsDegree{\resolution}$, the maximum degree of $\hsGraph{\resolution}$ and comparing $\hsWeight{\resolution}$ with the corresponding tree threshold $\criticalWeight{\hsDegree{\resolution}}$. 
Recall that we assume $\radius = (1/\normSphere{d})^{1/d}$.
A simple geometric argument shows that $\hsDegree{\resolution}$ is roughly upper bounded by $2^d \resolution^d$ for sufficiently large $\resolution$.
Now, observe that
\[
	\hsWeight{\resolution} = \frac{\hsFugacity}{\resolution^d} < \criticalWeight{2^d \resolution^d} ,
\] 
for $\hsFugacity < \resolution^d \criticalWeight{2^d \resolution^d}$.
This follows from the fact that $\resolution^d \criticalWeight{2^d \resolution^d}$ converges to $\eulerE/2^d$ from above for $\resolution \to \infty$.
Thus, the approximation bound from \Cref{thm:hard_sphere_approx} and the uniqueness bound in \cite{perkinsHardSpheres2020} coincide with the regime of $\hsFugacity$, for which $\hsWeight{\resolution}$ is below the tree threshold $\criticalWeight{\hsDegree{\resolution}}$ in the limit $\resolution \to \infty$. 

The arguments above show that for a sufficiently high resolution $\resolution$ the partition function of the hard-sphere model $\partitionFunction[\volume, \hsFugacity]$ is well approximated by the partition function of our discretization $(\hsGraph{\resolution}, \hsWeight{\resolution})$ and that $(\hsGraph{\resolution}, \hsWeight{\resolution})$ is below the tree threshold for $\hsFugacity < \eulerE/2^d$.
However, this does not immediately imply an approximation algorithm within the desired runtime bounds.
Based on \Cref{lemma:hard_sphere_convergence}, we still need to choose $\resolution$ exponentially large in the volume $\vol{\volume}$.
Note that the number of vertices in $\hsGraph{\resolution}$ is roughly $\size{\hsVertices{\resolution}} \in \bigTheta{\resolution^d \vol{\volume}}$. 
Even without explicitly constructing the graph, this causes problems, as the best bound for the mixing time of the Glauber dynamics is polynomial in $\size{\hsVertices{\resolution}}$ and thus exponential in $\vol{\volume}$.
Intuitively, the reason for this mixing time is that the Glauber dynamics only change one vertex at each step.
Assuming that each vertex should be updated at least once to remove correlations with the initial state, any mixing time that is sublinear in the number of vertices is unlikely.
We circumvent this problem by applying dynamics that update multiple vertices at each step but still allow each step to be computed efficiently without constructing the graph explicitly.

\subsection{Block and clique dynamics}\label{sec:intro_block}
Most of the results that we discuss from now on apply to the multivariate version of the hard-core model, that is, each vertex $v \in \vertices$ has its own weight $\weight[v]$.
For a given graph $\graph = (\vertices, \edges)$ we denote the set of such vertex weights by $\weights = \{\weight[v]\}_{v \in \vertices}$ and write $(\graph, \weights)$ for the resulting multivariate hard-core instance.
In the multivariate setting, the contribution of an independent set $\independentSet \in \independentSets[\graph]$ to the partition function is defined as the product of its vertex weights (i.e., $\prod_{v \in \independentSet} \weight[v]$), where the contribution of the empty set is fixed to $1$.
Similar to the univariate hard-core model, the Gibbs distribution assigns a probability to each independent set proportionally to its contribution to the partition function.
For a formal definition, see \Cref{subsec:hard_core}.  

As we discussed before, the main problem with approximating the partition function of our discretization $(\hsGraph{\resolution}, \hsWeight{\resolution})$ is that the required graph $\hsGraph{\resolution}$ is exponential in the volume of the original continuous system $\vol{\volume}$.
As the Glauber dynamics Markov chain only updates a single vertex at each step, the resulting mixing time is usually polynomial in the size of the graph, which is not feasible in our case.
Various extensions to Glauber dynamics for updating multiple vertices in each step have been proposed in the literature, two of which we discuss in the following.

\subsubsection*{Clique dynamics}
Recently, in \cite{friedrich2020polymer} a Markov chain, called \emph{clique dynamics}, was introduced in order to efficiently sample from the Gibbs distribution of abstract polymer models.
Note that this is similar to our algorithmic problem, as abstract polymer models resemble multivariate hard-core instances. 
For a given graph $\graph = (\vertices, \edges)$, we call a set $\clique = \{\clique[i]\}_{i \in [\numberOfCliques]} \subseteq \powerset{\vertices}$ a clique cover of size $\numberOfBlocks$ if and only if its union covers all vertices $\vertices$ and each $\clique[i] \in \clique$ induces a clique in $\graph$.
For an instance of the multivariate hard-core model $(\graph, \weights)$ and a given clique cover $\clique$ of $\graph$ with size $\numberOfCliques$ the clique dynamics Markov chain $\cliqueDynamics[\graph, \weights, \clique]$ is defined as follows.
First, a clique $\clique[i] \in \clique$ for $i \in [\numberOfCliques]$ is chosen uniformly at random.
Let us write $\subgraph{\graph}{\clique[i]}$ for the subgraph, induced by $\clique[i]$, and $\subgraph{\weights}{\clique[i]} = \{\weight[v]\}_{v \in \clique[i]}$ for the corresponding set of vertex weights.
Next, an independent set from $\independentSets[\subgraph{\graph}{\clique[i]}]$ is chosen according to the Gibbs distribution $\gibbsDistribution[\subgraph{\graph}{\clique[i]}, \subgraph{\weights}{\clique[i]}]$.
Note that, as the vertices $\clique[i]$ form a clique, such an independent set is either the empty set or contains a single vertex from $v \in \clique[i]$.
If the empty set is drawn, all vertices from $\clique[i]$ are removed from the current independent set.
Otherwise, if a single vertex $v \in \clique[i]$ is drawn, the chain tries to add $v$ to the current independent set.

Using a coupling argument, it was proven in \cite{friedrich2020polymer} that the so-called \emph{\pmc} implies that for any clique cover of size $\numberOfCliques$ the clique dynamics are mixing in time polynomial in $\numberOfCliques$ and $\maxPartition$, where $\maxPartition = \max_{i \in [\numberOfCliques]} \{\partitionFunction[\subgraph{\graph}{\clique[i]}, \subgraph{\weights}{\clique[i]}]\}$ denotes the maximum partition function of a clique in $\clique$.
This is important for the application to polymer models, as they are usually used to model partition functions of other spin systems, which often results in a multivariate hard-core model of exponential size \cite{HPR19,CFFGL19,JKP19,CGGPSV19,BCHPT20,CP20,GGS20}. 
As discussed in \cite{friedrich2020polymer}, those instances tend to have polynomial size clique covers that arise naturally from the original spin system.
In such cases, the mixing time of clique dynamics is still polynomial in the size of original spin system, as long as the \pmc is satisfied.   

This is very similar to our discretization $(\hsGraph{\resolution}, \hsWeight{\resolution})$.
To see this, set $a = \frac{2 \resolution}{\sqrt{d}} \radius$ and divide the $d$-dimensional integer lattice of side length $\resolution \sidelength$ into cubic regions of side length $a$.
Every pair of integer points within such a cubic region has Euclidean distance less than $2 \resolution \radius$, meaning that the corresponding vertices in $\hsGraph{\resolution}$ are adjacent.
Thus, each such cubic region forms a clique, resulting in a clique cover of size $(\resolution \sidelength/a)^d \in \bigO{\vol{\volume}}$.
This means, there is always a clique cover with size linear in $\vol{\volume}$ and independent of the resolution $\resolution$.
By showing that, for the univariate hard-core model, the mixing time of clique dynamics is polynomial in the size of the clique cover for all $\hsWeight{\resolution} < \criticalWeight{\hsDegree{\resolution}}$, we obtain a Markov chain with mixing time polynomial in $\vol{\volume}$ independent of the resolution $\resolution$.
Unfortunately, the \pmc does not hold for the entire regime up to the tree threshold in the univariate hard-core model.
We overcome this problem by proving a new condition for rapid mixing of clique dynamics based on a comparison with block dynamics.

\subsubsection*{Block dynamics}
Block dynamics are a very natural generalization of Glauber dynamics to arbitrary sets of vertices.
For a given graph $\graph = (\vertices, \edges)$, we call a set $\block = \{\block[i]\}_{i \in [\numberOfBlocks]} \subseteq \powerset{\vertices}$ a block cover of size $\numberOfBlocks$ if and only if its union covers all vertices $\vertices$.
We refer to the elements of $\block$ as blocks.
Note that the clique cover discussed before is a special case of a block cover, where all blocks are restricted to be cliques.
At each step, the block dynamics Markov chain $\blockDynamics[\graph, \weight, \block]$ chooses a block $\block[i] \in \block$ uniformly at random.
Then, the current independent set is updated on $\block[i]$ based on the projection of the Gibbs distribution onto $\block[i]$ and conditioned on the current independent set outside $\block[i]$.
For a formal definition, see \Cref{subsec:block_dynamics}.

In fact, block dynamics are defined for a much more general class of spin systems than the hard-core model.
However, due to the fact that each step of the Markov chain involves sampling from a conditional Gibbs distribution, block dynamics are rarely used as an algorithmic tool on its own.
Instead, they are usually used to deduce rapid mixing of other dynamics.

For spin systems on lattice graphs, close connections between the mixing time of block dynamics and Glauber dynamics are known \cite{martinelli1999lectures}.
Such connections were for example applied to improve the mixing time of Glauber dynamics of the Monomer Dimer model on torus graphs \cite{monomerDimerTorus}.
Moreover, block dynamics were used to improve conditions for rapid mixing of Glauber dynamics on specific graph classes, such as proper colorings \cite{coloring2006,coloring+indep2018,coloring2018,mossel2010gibbs} or the hard-core model \cite{coloring+indep2018,mossel2010gibbs} in sparse random graphs. 
A very general result for the mixing time of block dynamics was achieved in \cite{nonlocalMC}, who proved that for all spin systems on a finite subgraph of the $d$-dimensional integer lattice the mixing time of block dynamics is polynomial in the number of blocks if the spin system exhibits strong spatial mixing.
This result was later generalized in \cite{blanceIsing} for the Ising model on arbitrary graphs.
Very recently, block dynamics based random equally-sized blocks where used in \cite{ChenOptimal} to prove entropy factorization and improve the mixing time of Glauber dynamics for a variety of discrete spin systems up to the tree threshold.
 
Although our discretization works by restricting sphere positions to the integer lattice, the resulting graph is rather different from the lattice.
Thus, results like those in \cite{nonlocalMC} do not apply to our setting.
However, on the other hand, we do not need to prove rapid mixing for arbitrary block covers.
Instead, in order to obtain rapid  mixing for clique dynamics, it is sufficient to establish this result for cases where all blocks are cliques. 

Applying block dynamics directly would involve sampling from a conditional Gibbs distribution within each clique.
Due to the exponential size of the cliques in our discretization, this would impose additional algorithmic challenges.
Instead, similar to the previous literature, we rather use block dynamics as a tool for proving rapid mixing of another Markov chain, namely clique dynamics.

\subsubsection*{Improved mixing condition for clique dynamics via block dynamics}
We analyze the mixing time of clique dynamics for a given clique cover by relating it to the mixing time of block dynamics, using the cliques as blocks.
This is done by investigating a notion of pairwise influence between vertices that has also been used to establish rapid mixing of Glauber dynamics up to the tree threshold \cite{ALOG20}.
Let $\gibbsPr{\graph}{\inSet{w}}$ denote the probability of the event that a vertex $w \in \vertices$ is in an independent set drawn from $\gibbsDistribution[\graph, \weights]$.
Further, let $\gibbsPr{\graph}{\notInSet{w}}$ denote the probability that $w$ is not in an independent set. 
We extend this abuse of notation to conditional probabilities, so $\gibbsPr{\graph}{\cdot}[\notInSet{w}]$ for example denotes the probability of some event conditioned on $w$ not being in an independent set.
For a pair of vertices $v, w \in \vertices$ the influence $\pairwiseInfluence{\graph}[][v, w]$ of $v$ on $w$ is defined as 
\[
	\pairwiseInfluence{\graph}[][v, w] = 
	\begin{cases}
		0 & \text{ if } v = w, \\
		\gibbsPr{\graph}{\inSet{w}}[\inSet{v}] - \gibbsPr{\graph}{\inSet{w}}[\notInSet{v}] & \text{ otherwise. } 
	\end{cases}
\] 
The following condition in terms of pairwise influence is central to our analysis.
\begin{condition}
	\label{condition:influence_bound}
	Let $(\graph, \weights)$ be an instance of the multivariate hard-core model.
	There is a constant $C \in \R_{>0}$ and a function $q\colon V \to \R_{>0}$ such that for all $S \subseteq \vertices$ and $r \in S$ it holds that
	\[
		\sum_{v \in S} \absolute{\pairwiseInfluence{\graph}[][r, v]} q(v) \le C q(r) .
		\qedhere
	\]
\end{condition}
Note that this condition appeared before in \cite{Chen2020rapid}, where it was used for bounding the mixing time of Glauber dynamics for anti-ferromagnetic spin systems.
Given \Cref{condition:influence_bound}, we obtain the following result for the mixing time of block dynamics based on a clique cover.
\begin{restatable}{theorem}{mixingBlockDynamics}
	\label{thm:main}
	Let $(\graph, \weights)$ be an instance of the multivariate hard-core model such that \Cref{condition:influence_bound} is satisfied.
	Let $\clique$ be a clique cover for $\graph$ of size $\numberOfCliques$, and let $\maxPartition = \max_{i \in [\numberOfCliques]} \{\partitionFunction[\subgraph{\graph}{\clique[i]}, \subgraph{\weights}{\clique[i]}]\}$.
	The mixing time of the block dynamics $\blockDynamics[\graph, \weights, \clique]$, starting from $\emptyset \in \independentSets[\graph]$, is bounded by
	\[
		\mix{\blockDynamics}{\err}[\emptyset] \le \numberOfCliques^{\bigO{(2 + C)C}} \maxPartition^{\bigO{(2 + C)C}} \ln\left( \frac{1}{\err} \right) .
		\qedhere
	\]
\end{restatable}

Using a bound for the sum of absolute pairwise influences that was recently established in \cite{Chen2020rapid}, it follows that the univariate hard-core model satisfies \Cref{condition:influence_bound} up to the tree threshold.
As a result, we know that the mixing time of block dynamics is polynomial in $\numberOfCliques$ and $\maxPartition$ for any clique cover of size $\numberOfCliques$. 
To the best of our knowledge, this is the first result for the mixing time of block dynamics for the univariate hard-core model on general graphs that holds in this parameter range.

As we aim to apply clique dynamics to avoid sampling from the conditional Gibbs distribution in each step, we still need to prove that \Cref{thm:main} also holds in terms of clique dynamics.
To this end, we apply a Markov chain comparison argument from \cite{diaconis1993comparison} to prove that using clique dynamics instead of block dynamics for the same clique cover $\clique$ increases the mixing time by at most a factor $2 \maxPartition$.
The following corollary, which is central for proving \Cref{thm:hard_sphere_approx}, follows immediately.
\begin{restatable}{corollary}{univariateCliqueDynamics}
	\label{cor:univarite_clique_dynamics}
	Let $(\graph, \weight)$ be an instance of the univariate hard-core model such that the degree of $\graph$ is bounded by $\degree$.
	Let $\clique$ be a given clique cover of size $\numberOfCliques$ with $\maxPartition = \max_{i \in [\numberOfCliques]} \{\partitionFunction[\subgraph{\graph}{\clique[i]}, \weight]\}$.
	Denote by $\cliqueDynamics = \cliqueDynamics[\graph, \weight, \clique]$ the corresponding clique dynamics.
	If there is some $\delta \in \R_{>0}$ such that $\weight \le (1-\delta) \criticalWeight{\degree}$ then the mixing time of the clique dynamics $\cliqueDynamics$, starting from $\emptyset \in \independentSets[\graph]$, is bounded by
	\[
		\mix{\cliqueDynamics}{\err}[\emptyset] \le \numberOfCliques^{\bigO{1/\delta^2}} \maxPartition^{\bigO{1/\delta^2}} \ln\left( \frac{1}{\err} \right) .
		\qedhere
	\]
\end{restatable}

\subsubsection*{A side journey: comparison to multivariate conditions}
In fact, \Cref{cor:univarite_clique_dynamics} is sufficient for our application to the hard-sphere model.
However, we also aim to set \Cref{condition:influence_bound} in the context of other conditions for rapid mixing of clique dynamics for the multivariate hard-core model.
Note that such a rapid mixing result for clique dynamics caries over to Glauber dynamics by taking each vertex as a separate clique of size 1.

To this end, we compare \Cref{condition:influence_bound} to a strict version of the \pmc, originally introduced in \cite{friedrich2020polymer} in the setting of clique dynamics for abstract polymer models.
It turns out that this strict version of the \pmc directly implies \Cref{condition:influence_bound}.
This is especially interesting, as the \pmc was initially introduced as a local condition (only considering the neighborhood of each vertex) and is based on a coupling argument.
However, we show that it can as well be understood as a sufficient condition for the global decay of pairwise influence with increasing distance between vertices. 

Formally, we say that the \spmc is satisfied for an instance of the multivariate hard-core model $(\graph, \weights)$ if there is a function $\spmcFunction\colon \vertices \to \R_{>0}$ and a constant $\spmcConstant \in (0, 1)$ such that for all $v \in \vertices$ it holds that
\[
	\sum_{w \in \neighbors{v}} \frac{\weight[w]}{1 + \weight[w]} \spmcFunction[w] \le (1 - \spmcConstant) \spmcFunction[v], 
\]
where $\neighbors{v}$ is the neighborhood of $v$ in $\graph$.
This is a strict version of the \pmc in that the original \pmc would correspond to the case $\spmcConstant=0$ (i.e., the \spmc requires some strictly positive slack $\spmcConstant$).

The result of our comparison is summarized in the following statement.
\begin{restatable}{lemma}{pairwiseInfluenceSpmc}
	\label{lemma:pairwise_influence_spmc}
	Let $(\graph, \weights)$ be an instance of the multivariate hard-core model.
	If $(\graph, \weights)$ satisfies the \spmc for a function $\spmcFunction$ and a constant $\spmcConstant$, then it also satisfies \Cref{condition:influence_bound} for $q=\spmcFunction$ and $C = \frac{1}{\spmcConstant}$.
\end{restatable}
\Cref{lemma:pairwise_influence_spmc} is proven by translating the calculation of pairwise influences to the self-avoiding walk tree of the graph, based on a result in \cite{Chen2020rapid}, and applying a recursive argument on this tree.
The technical details are given in  \Cref{sec:multivariate}.
   
Despite being an interesting relationship between local coupling arguments and global pairwise influence, \Cref{lemma:pairwise_influence_spmc} also implies that, from an algorithmic perspective, \Cref{thm:main} can be used to produce similar results as those obtained in \cite{friedrich2020polymer} for abstract polymer Further, note that for the univariate model, using pairwise influence yields strictly better results than any coupling approach in the literature.
This raises the question if a refined argument based on pairwise influences can be used in the multivariate setting to improve on the clique dynamics condition, leading to better approximation results on abstract polymer models.

\subsection{Analyzing spectral expansion}\label{sec:intro_spectral}
As core technique for obtaining \Cref{thm:main}, we adapt an approach for bounding the mixing time that was recently used to prove rapid mixing of Glauber dynamics for the entire regime blow the tree threshold for several applications, such as the hard-core model~\cite{ALOG20}, general two-state spin systems~\cite{Chen2020rapid}, and proper colorings~\cite{Chen2020coloring,feng2020rapid}.
The idea is to map the desired distribution to a weighted simplicial complex.

A simplicial complex $\complex$ over a groundset $\complexGroundset$ is a set family $\complex \subseteq \powerset{\complexGroundset}$ such that for each $\complexFace \in \complex$ every subset of $\complexFace$ is also in $\complex$.
We call the elements $\complexFace \in \complex$ the faces of $\complex$ and refer to its cardinality~$\size{\complexFace}$ as dimensionality.

For a hard-core instance $(\graph, \weight)$, the authors of \cite{ALOG20} construct a simplicial complex over a ground set~$\complexGroundset$ that contains two elements $\complexElement{\inSet{v}}, \complexElement{\notInSet{v}} \in \complexGroundset$ for each vertex $v \in \vertices$.
For every independent set $\independentSet \in \independentSets[\graph]$, a face $\faceFromIndep{\independentSet} \in \complex$ is introduced such that $\complexElement{\inSet{v}} \in \faceFromIndep{\independentSet}$ if $v \in \independentSet$ and $\complexElement{\notInSet{v}} \in \faceFromIndep{\independentSet}$ otherwise.
The simplicial complex is completed by taking the downward closure of these faces.
Note that by construction all maximum faces of the resulting complex are $\size{\vertices}$-dimensional and there is a one-to-one correspondence between the maximum faces and the independent sets in $\independentSets[\graph]$.
By assigning each maximum face $\faceFromIndep{\independentSet} \in \complex$ an appropriate weight, the Glauber dynamics can be represented as a random walk on those maximum faces, which is sometimes referred to as the two-step walk or down-up walk.
Using a local-to-global theorem \cite{alev2020improved}, the mixing time of this two-step walk can then be bounded based on certain local expansion properties of the simplicial complex $\complex$ (see \Cref{sec:prelim} for the technical details).
It is then proved that such local expansion properties are well captured by the largest eigenvalue of the pairwise influence matrix $\pairwiseInfluence{\graph}$, which is a $\size{\vertices} \times \size{\vertices}$ matrix that contains the pairwise influence $\pairwiseInfluence{\graph}[][v, w]$ for all $v, w \in \vertices$.
Finally, by bounding those influences a bound on this largest eigenvalue of $\pairwiseInfluence{\graph}$ is obtained.
This analysis was later refined and generalized in \cite{Chen2020rapid} to general two-state spin systems.

This method was independently extended in \cite{Chen2020coloring} and \cite{feng2020rapid} to the non-Boolean domain by applying it to the Glauber dynamics for proper colorings.
The main differences to the Boolean domain are that elements of the simplicial complex now represent combinations of a vertex and a color.
Furthermore, the bound on the local spectral expansion was obtained by using a different influence matrix, which captures the effect of selecting a certain color for one vertex on the distribution of colors for another vertex.

Although we are dealing with the hard-core model, which is Boolean in nature, the way that we model block dynamics is mainly inspired by the existing work on proper colorings \cite{Chen2020coloring}.
Assume we have an instance of the multivariate hard-core model $(\graph, \weights)$ and let $\clique$ be a clique cover for $\graph$ of size $\numberOfCliques$ such that every pair of distinct cliques is vertex-disjoint (i.e., $\clique$ is a partition of~$\graph$ into cliques).
We construct a simplicial complex $\complex$ based on a ground set $\complexGroundset$ that contains one element $\complexElement{v} \in \complexGroundset$ for each vertex $v \in \vertices$ and one additional element $\complexEmpty{i}$ for each clique $\clique[i] \in \clique$.
We introduce a face $\faceFromIndep{\independentSet} \in \complex$ for each independent set $\independentSet \in \independentSets[\graph]$ such that for every $\clique[i] \in \clique$ we have $\complexEmpty{i} \in \faceFromIndep{\independentSet}$ if $\clique[i] \cap \independentSet = \emptyset$ and $\complexElement{v} \in \faceFromIndep{\independentSet}$ if $\clique[i] \cap \independentSet = \{v\}$ for some $v \in \clique[i]$.
The simplicial complex is completed by taking the downward closure of these faces.
As we discuss in \Cref{subsec:complex_representation}, all maximum faces of the resulting complex are $\numberOfCliques$-dimensional and there is a bijection between the maximum faces and the independent sets of $\graph$.
Furthermore, there is a natural partitioning $\{\complexPartition{i}\}_{i \in [\numberOfCliques]}$ of the ground set~$\complexGroundset$, each partition $\complexPartition{i}$ corresponding to a clique $\clique[i]$, such that every maximum face in $\complex$ contains exactly one element from each partition $\complexPartition{i}$.

By weighting each maximum face of $\complex$ by the contribution of the corresponding independent set to the partition function, the block dynamics based on $\clique$ are equivalent to the two-step walk on~$\complex$.
Thus, in order to bound the mixing time of the block dynamics, it is sufficient to study the local expansion properties of $\complex$.
To this end, we adapt the influence matrix used for proper colorings in \cite{Chen2020coloring}.
For $x \in \complexGroundset$, let $\gibbsPr{\graph}{x}$ denote the probability that $x \in \faceFromIndep{\independentSet}$ for an independent set $\independentSet \in \independentSets[\graph]$ drawn from $\gibbsDistribution[\graph, \weights]$ and corresponding maximum face $\faceFromIndep{\independentSet} \in \complex$.
Similarly as for defining pairwise influences, we extend this notation to conditional probabilities.
The clique influence matrix $\cliqueInfluence{\graph}{\clique}$ contains an entry $\cliqueInfluence{\graph}{\clique}[x, y]$ for each $x, y \in \complexGroundset$ with
\[
	\cliqueInfluence{\graph}{\clique}[x, y] = \begin{cases}
		0 & \text{ if $x,y \in \complexPartition{i}$ for some $i \in [\numberOfCliques]$, } \\
		\gibbsPr{\graph}{y}[x] - \gibbsPr{\graph}{y} & \text{ otherwise. }
	\end{cases}
\]
By using similar linear-algebraic arguments as in \cite{Chen2020coloring} we prove that the maximum eigenvalue of $\cliqueInfluence{\graph}{\clique}$ can be used to upper bound the local spectral expansion of $\complex$.
To obtain \Cref{thm:main} it is then sufficient to relate \Cref{condition:influence_bound} to the maximum eigenvalue of $\cliqueInfluence{\graph}{\clique}$.
The following lemma establishes this connection.
\begin{restatable}{lemma}{cliqueInfluenceBound}
	\label{lemma:clique_influence_bound}
	Let $(\graph, \weights)$ be an instance of the multivariate hard-core model that satisfies \Cref{condition:influence_bound} for a function $q$ and a constant $C$.
	For every $S \subseteq \vertices$ and every disjoint clique cover $\clique$ of $\subgraph{\graph}{S}$ it holds that
	the largest eigenvalue of $\cliqueInfluence{\subgraph{\graph}{S}}{\clique}$ is at most $(2 + C) C$.
\end{restatable}

Note that our simplicial-complex representation is only given under the assumption that the cliques in the clique cover $\clique$ are pairwise disjoint.
Indeed, this is a necessary requirement to map the block dynamics to the two-step walk such that the local-global-theorem from \cite{alev2020improved} can be applied.
Thus, \Cref{lemma:clique_influence_bound} only helps to prove \Cref{thm:main} for disjoint clique covers.
However, we relax this requirement by proving that for every clique cover $\clique$ a disjoint clique cover~$K$ can be constructed such that the block dynamics $\blockDynamics[\graph, \weights, \clique]$ and $\blockDynamics[\graph, \weights, K]$ have asymptotically the same mixing time.
By this comparison argument, we extend \Cref{thm:main} to arbitrary clique covers. 

We are aware that, in the case of Glauber dynamics, more recent techniques for combining simplical complex representations with entropy factorization as proposed in \cite{ChenOptimal} yield superior mixing time results. 
However, in case of the hard-core model, this approach comes with an additional multiplicative factor of $\degree^{\bigO{\degree^2}}$ in the mixing time (see section 8 of \cite{ChenOptimal}).
Although negligible for bounded degree graphs, this would be too much for our application, as the degree of our discretization gets exponentially large in the continuous volume $\vol{\volume}$ of the system.
Thus, directly relating local spectral expansion with the spectral gap of block dynamics is more suitable in our case.
We leave as an open question, whether a modification of the approach in \cite{ChenOptimal} can be applied to further improve our mixing time result.

\subsection{Outlook}\label{sec:intro_outlook}
We obtain the fugacity bound from \Cref{thm:hard_sphere_approx} based on the tree threshold $\criticalWeight{\degree}$ of the hard-core model.
An obvious question is whether there are any structural properties of our discretization that can be used to improve this bound. 
Similar results are known for specific graph classes, such as the $2$-dimensional square lattice \cite{restrepo2013improved,vera2013improved,Weitz2006Counting}.
In \cite{perkinsHardSpheres2020} the authors discuss that a generalization of the connective constant to the continuous Euclidean space might be applicable to improve their uniqueness result for the hard-sphere model.
A comparable algorithmic result was already established for the discrete hard-core model in \cite{sinclair2017spatial}.
However, any such improvement for our discretization would require the connective constant of~$\hsGraph{\resolution}$ to be at least by a constant factor small than its maximum degree $\hsDegree{\resolution}$.
Unfortunately, due to a result in~\cite{penrose1994self}, this is not the case.
Although this does not necessarily imply that a similar concept does not work in continuous space, it gives a strong evidence that a more specialized tool instead of the connective constant might be required.

A different direction for future work is to see which other quantities and properties of the model are preserved under discretization.
This would especially include the thermodynamic pressure and its analyticity.
As a matter of fact, non-analytic points of the pressure along the positive real axis of fugacity in the thermodynamic limit are known to mark phase transitions in infinite volume systems (see for example \cite{perkinsHardSpheres2020}).
One way to approach this could be to prove a relation between zero-freeness of the continuous and the discretized partition function in a complex neighborhood of the real axis by extending our convergence result to the complex domain. 
Along this line, insights could be gained in how far properties like correlation decay and phase transitions (or their absence) are preserved under sufficiently fine discretization. 

From a purely technical point of view, it is interesting to see if our result on the mixing time of block dynamics in \Cref{thm:main} also holds without the requirement of using cliques as blocks.
Especially: is the mixing time for block dynamics for the univariate hard-core model polynomial in the number of blocks for any block cover? 
Most of our techniques that we use for clique covers, such as modeling the distribution as a simplicial complex and relating its local spectral expansion to the clique influence matrix, can be generalized in a straightforward way for different choices of blocks.
However, the main difficulty is to relate generalized versions of the clique influence matrix to pairwise influences, as we do in \Cref{lemma:clique_influence_bound}.
One way to circumvent this might be to not rely on pairwise influences at all but to rather investigate the influence matrix directly, for example, via different computational-tree methods. 

Finally, it would be interesting to see if approaches like ours can be extended to other continuous models from statistical physics (see for example coarse-graining \cite{espanol2004statistical}).
We believe that the variety of tools that are already established for discrete spin systems are useful in this setting to establish rigorous computational results for different continuous models.
We emphasize that clique and block dynamics are a useful computational tool to handle the exponential blow-ups caused by discretization.


\section{Preliminaries}\label{sec:prelim}
We denote the set of all natural numbers (including~$0$) by~\N and the set of all real numbers by~\R.
For each $n \in \N$, let $[n]$ denote the interval $[1, n] \cap \N$.
Further, for a graph $\graph = (\vertices, \edges)$, we write $\neighbors{v}[\graph]$ for the open neighborhood of a vertex $v \in \vertices$(i.e., all $w \in \vertices$ with $(v, w) \in \edges$) and $\neighborsClosed{v} = \neighbors{v}[\graph] \cup \{v\}$ for the closed neighborhood.
We might omit the graph if it is clear from the context.

\subsection{Markov chains and spectral properties}
For any (time-homogeneous) Markov chain $\markov$, we denote its state space by $\stateSpace[\markov]$ and its transition probabilities by $\transitionMatrix[\markov]$.
If $\markov$ has a unique stationary distribution, we denote it by $\stationary[\markov]$.
Assume $\size{\stateSpace[\markov]} = N \in \N_{>0}$.
It is well known that, if $\markov$ is reversible with respect to $\stationary[\markov]$, this implies that $\transitionMatrix[\markov]$ has $N$ real eigenvalues
\[
	1 = \eigenvalueOf{\transitionMatrix[\markov]}[1] \ge \eigenvalueOf{\transitionMatrix[\markov]}[2] \ge \dots \ge \eigenvalueOf{\transitionMatrix[\markov]}[N] \ge -1 . 
\] 
We write $\maxAbsEigenvalue{\transitionMatrix[\markov]} = \max \{\eigenvalueOf{\transitionMatrix}[1], \absolute{\eigenvalueOf{\transitionMatrix}[N]}\}$ for the largest absolute eigenvalue and call $1 - \maxAbsEigenvalue{\transitionMatrix[\markov]}$ the \emph{spectral gap} of $\transitionMatrix[\markov]$.
We extend these notations to general matrices $A$ with real eigenvalues, e.g., we denote the largest eigenvalue by $\eigenvalueOf{A}[1]$.

If $\markov$ is ergodic, we define its mixing time starting from some state $x \in \stateSpace[\markov]$ as 
\[
	\mix{\markov}{\err}[x] = \inf \{ t \in \N \mid \dtv{\transitionSymbol^t_{\markov}(x, \cdot)}{\stationary[\markov]} \le \err \},
\]
where $\transitionSymbol^t_{\markov}(x, \cdot)$ is the distribution of $\markov$ on $\stateSpace[\markov]$ after $t$ steps, starting from $x$, and where $\dtv{\cdot}{\cdot}$ denotes the total variation distance.
Recall that for any ergodic, reversible Markov chain $\markov$ and every state $x \in \stateSpace[\markov]$, it holds that 
\begin{equation} \label{eq:mixing_spectral_gap}
	\mix{\markov}{\err}[x] \le \frac{1}{1 - \maxAbsEigenvalue{\transitionMatrix[\markov]}} \ln\left( \frac{1}{\stationary[\markov][x] \cdot \err} \right) .
\end{equation}
For further details on Markov chains please refer to \cite{levin2017markov}.

\subsection{The multivariate hard-core model} \label{subsec:hard_core}
Let $\graph = (\vertices, \edges)$ be an undirected graph, and let $\independentSets[\graph]$ denote the set of independent sets in $\graph$; if the graph is clear from the context, we only write $\independentSets$. 
The \emph{multivariate hard-core model} is a tuple $(\graph, \weights)$, where $\weights = \{\weight[v]\}_{v \in \vertices}$ is a set of weights, containing one weight $\weight[v] \in \R_{>0}$ for each vertex $v \in \vertices$.
The \emph{partition function} of $(\graph, \weights)$ is defined as
\[
	\partitionFunction[\graph, \weights] \defeq \sum_{\independentSet \in \independentSets} \prod_{v \in \independentSet} \weight[v] .
\]
The \emph{Gibbs distribution} $\gibbsDistribution[\graph, \weights]$ is a probability distribution on $\independentSets$, assigning each independent set $\independentSet \in \independentSets$ the probability
\[
	\gibbsDistributionFunction{\independentSet}[\graph, \weights] = \frac{\prod_{v \in \independentSet} \weight[v]}{\partitionFunction[\graph, \weights]} . 
\]
If the model $(\graph, \weights)$ is clear, we only write~$\partitionFunction$ and~$\gibbsDistribution$.

Large parts of our analysis consider the Gibbs distributions and the partition functions of induced subgraphs $\subgraph{\graph}{S}$ for $S \subseteq \vertices$ while keeping the weights of the respective vertices in $S$.
In this case, we might omit the set of weights and write $\partitionFunction[\subgraph{\graph}{S}]$ for $\partitionFunction[\subgraph{\graph}{S}, \subgraph{\weights}{S}]$ or $\gibbsDistribution[\subgraph{\graph}{S}]$ for $\gibbsDistribution[\subgraph{\graph}{S}, \subgraph{\weights}{S}]$.
Further, for any non-empty set of vertices $S \subseteq \vertices$, we define $\gibbsDistribution[\graph][S]$ to be the distribution of the independent sets in $\independentSets[\subgraph{\graph}{S}]$ induced by $\gibbsDistribution[\graph]$.
Formally, this means $\gibbsDistribution[\graph][S]$ assigns every independent set $\independentSet \in \independentSets[\subgraph{\graph}{S}]$ the probability
\[
	\gibbsDistributionFunction{\independentSet}[\graph][S] = \sum_{\independentSet' \in \independentSets[\graph]} \ind{\independentSet \subseteq \independentSet'} \gibbsDistributionFunction{\independentSet'}[\graph][] .	
\]

We associate every independent set $\independentSet \in \independentSets$ with a spin assignment $\spinConfig[][][\independentSet]\colon \vertices \to \{0, 1\}$ such that $\left(\spinConfig[][][\independentSet]\right)^{-1}(1) = \independentSet$.
We extend this notation to restrictions on subsets $S \subseteq \vertices$.
For any independent set $\independentSet \in \independentSets$, the \emph{partial configuration} on~$S$ corresponding to $\independentSet$ is a spin assignment $\spinConfig[S][][\independentSet]\colon S \to \{0, 1\}$ such that $\big(\spinConfig[S][][\independentSet]\big)^{-1}(1) = \independentSet \cap S$.
By abuse of notation, we use these spin assignments as events (e.g., for conditioning on partial configurations).
Further, for all $S \subseteq V$ let $\zeroSpinConfig[S]\colon S \to \{0\}$ be the partial configuration that fixes all vertices in $S$ not to be in the independent set (i.e., $\zeroSpinConfig[S] = \spinConfig[S][][\emptyset]$).

Finally, for each for $v \in \vertices$, we write $\gibbsPr{\graph}{\inSet{v}}$ to denote the probability of the event that $v \in \independentSet$ for $\independentSet \sim \gibbsDistribution$, and $\gibbsPr{\graph}{\notInSet{v}}$ to denote the probability of the event $v \notin \independentSet$ for $\independentSet \sim \gibbsDistribution$.
Formally,
\[
	\gibbsPr{\graph}{\inSet{v}} = \gibbsDistributionFunction{\{v\}}[\graph][\{v\}] 
	\text{ and } 
	\gibbsPr{\graph}{\notInSet{v}} = \gibbsDistributionFunction{\emptyset}[\graph][\{v\}] .
\]

\subsection{Clique dynamics and clique covers} \label{subsec:clique_dynamics}
For any graph $\graph = (\vertices, \edges)$, a set $\clique = \{\clique[i]\}_{i \in [\numberOfCliques]} \subseteq \powerset{\vertices}$ is called a \emph{clique cover} of $\graph$ with size $\numberOfCliques \in \N_{> 0}$ if $\bigcup_{i \in [\numberOfCliques]} \clique[i] = \vertices$ and each $\clique[i] \in \clique$ induces a clique in $\graph$.
Further, we call $\clique$ a \emph{disjoint clique cover} if every pair of distinct cliques in $\clique$ is vertex-disjoint.

Note that for every $S \subseteq \vertices$ it holds that $\gibbsDistributionFunction{~\cdot~}[\graph][S][\zeroSpinConfig[\vertices \setminus S]] = \gibbsDistributionFunction{\cdot}[\subgraph{\graph}{S}][]$. 
Thus, the following definition of clique dynamics is equivalent to our description in the introduction and easier to compare with our formalization of block dynamics in \Cref{def:block_dynamics}.
\begin{definition}[clique dynamics {\cite{friedrich2020polymer}}]
	\label{def:markov_chain}
	Let~$(\graph, \weights)$ be a multivariate hard-core model, and let~$\clique$ be a clique cover of $\graph$ with size $\numberOfCliques$. 
	The \emph{clique dynamics}~\cliqueDynamics[\graph, \weights, \clique] are the following Markov chain with state space~\independentSets[\graph]. 
	Let $(X_t)_{t \in \N}$ denote a (random) sequence of states of~\cliqueDynamics[\graph, \weights, \clique], where~$X_0$ is arbitrary. Then, for all $t \in \N$ and all $X_t = \independentSet$ with $\independentSet \in \independentSets[\graph]$, the transitions of \cliqueDynamics[\graph, \weights, \clique] are as follows:
	\begin{algorithmic}[1]
		\State\label{markov_chain:line:chooseClique} choose $i \in [\numberOfCliques]$ uniformly at random\,;
		\State\label{markov_chain:line:sampleVertex} choose $\independentSet_{+} \in \independentSets[\subgraph{\graph}{\clique[i]}]$ according to $\gibbsDistributionFunction{~\cdot~}[][\clique[i]][\zeroSpinConfig[\vertices \setminus \clique[i]]]$\,;
		\State\label{markov_chain:line:removeVertex}\lIf{$\independentSet_{+} = \emptyset$}
		{
			$X_{t+1} = \independentSet \setminus \clique[i]$
		}
		\State\label{markov_chain:line:addVertex}\lElseIf{\emph{$\independentSet \cup \independentSet_{+}$ is an independent set}}
		{
			$X_{t+1} = \independentSet \cup \independentSet_{+}$
		}
		\State\label{markov_chain:line:doNotChange}\lElse
		{
			$X_{t+1} = \independentSet$
		}
	\end{algorithmic}
	\vspace*{-1.2\baselineskip}
	\qedhere
\end{definition}

It was shown in \cite{friedrich2020polymer} that for any clique cover $\clique$ the clique dynamics Markov chain $\cliqueDynamics[\graph, \weights, \clique]$ is ergodic and reversible with stationary distribution $\gibbsDistribution[\graph, \weights]$.
In the case $\clique = \{\{v\} \mid v \in \vertices\}$, the clique dynamics correspond to the Glauber dynamics.

\subsection{Block dynamics and block covers} \label{subsec:block_dynamics}
For any graph $\graph = (\vertices, \edges)$, we call a set $\block = \{\block[i]\}_{i \in [\numberOfBlocks]} \subseteq \powerset{\vertices}$ a \emph{block cover} of $\graph$ with size $\numberOfBlocks \in \N_{> 0}$ if $\bigcup_{i \in [\numberOfBlocks]} \block[i] = \vertices$.
We refer to the elements of $\block$ as blocks.
Note that a clique cover is a special case of a block cover, where all blocks are cliques.
\begin{definition}[block dynamics]
	\label{def:block_dynamics}
	Let~$(\graph, \weights)$ be a multivariate hard-core model, and let $\block$ be a block cover of $\graph$ with size $\numberOfBlocks$. 
	We define the \emph{block dynamics}~\blockDynamics[\graph, \weights, \block] to be the following Markov chain with state space~\independentSets[\graph]. 
	Let $(X_t)_{t \in \N}$ denote a (random) sequence of states of~\blockDynamics[\graph, \weights, \block], where~$X_0$ is arbitrary. 
	Then, for all $t \in \N$ and any $X_t = \independentSet$ with $\independentSet \in \independentSets[\graph]$, the transitions of \blockDynamics[\graph, \weights, \block] are as follows:
	\begin{algorithmic}[1]
		\State\label{block_chain:line:chooseClique} choose $i \in [\numberOfBlocks]$ uniformly at random\,;
		\State\label{block_chain:line:sampleInBlock} choose $\independentSet_{+} \in \independentSets[\subgraph{\graph}{\block[i]}]$ according to $\gibbsDistributionFunction{~\cdot~}[][\block[i]][\spinConfig[\vertices \setminus \block[i]][][\independentSet]]$\,;
		\State\label{block_chain:line:update} $X_{t+1} = (\independentSet \setminus \block[i]) \cup \independentSet_{+}$\,;
	\end{algorithmic}
	\vspace*{-1.2\baselineskip}
	\qedhere
\end{definition}

The block dynamics Markov chain $\blockDynamics[\graph, \weights, \block]$ is ergodic with stationary distribution $\gibbsDistribution$, independent of the chosen block cover $\block$.
If $\block = \{\{v\} \mid v \in \vertices\}$, then the block dynamics correspond to the Glauber dynamics.

\subsection{Pairwise influence}
Let $v, w \in \vertices$ and let $S \subset \vertices$ such that $v, w \notin S$.
Further, let $\spinConfig[S] = \spinConfig[S][][\independentSet]$ be a partial configuration on $S$ corresponding to any independent set $\independentSet \in \independentSets$.
The \emph{pairwise influence} of $v$ on $w$ in $\graph$ under condition $\spinConfig[S]$ is defined as
\[
\pairwiseInfluence{\graph}[\spinConfig[S]][v, w] = 
\begin{cases}
0 & \text{ if } v = w, \\
\gibbsPr{\graph}{\inSet{w}}[\inSet{v}, \spinConfig[S]] - \gibbsPr{\graph}{\inSet{w}}[\notInSet{v}, \spinConfig[S]] & \text{ otherwise } .
\end{cases}
\]
For the case $S = \emptyset$, we also write $\pairwiseInfluence{\graph}[][v, w]$.
Furthermore, we denote by $\pairwiseInfluence{\graph}[\spinConfig[S]]$ and $\pairwiseInfluence{\graph}$ the corresponding $(\numberOfVertices - \size{S}) \times (\numberOfVertices - \size{S})$ matrices.

\subsection{Simplicial complexes and local spectral expansion} \label{subsec:simplicial_complexes}
Let $\complexGroundset$ denote a set. A \emph{simplicial complex} (over~\complexGroundset) is family of subsets $\complex \subseteq \powerset{\complexGroundset}$ such that, for all $\complexFace \in \complex$ and all $\complexFace' \subseteq \complexFace$, it holds that $\complexFace' \in \complex$.
We call the elements $\complexFace \in \complex$ \emph{faces,} and we call $\size{\complexFace}$ the \emph{dimension} of a face $\complexFace$.
We denote the set of all $k$-dimensional faces in $\complex$ by $\complex[k]$.
A simplicial complex is \emph{pure $d$-dimensional} if and only if the set of all inclusion-maximal faces is exactly $\complex[d]$.
Last, we say that a pure $d$-dimensional simplicial complex is \emph{$d$-partite} if and only if there is a partition $\{\complexPartition{i}\}_{i \in [d]}$ such that, for all $i \in [d]$ and all $\complexFace \in \complex[d]$, it holds that $\size{\complexPartition{i} \cap \complexFace} = 1$.

We extend the definition of a pure $d$-dimensional simplicial complex $\complex$ to a \emph{weighted} simplicial complex $(\complex, \complexWeight)$ with a weight function $\complexWeight\colon \complex \to \R_{>0}$ in the following inductive manner.
Each face $\complexFace \in \complex[d]$ is assigned a weight $\complexWeight[\complexFace] \in \R_{>0}$.
Each non-maximal face $\complexFace' \in \complex$ has the weight
\[
	\complexWeight[\complexFace'] = \sum\nolimits_{\complexFace \in \complex[d]: \complexFace' \subset \complexFace} \complexWeight[\complexFace] .
\]

We are interested in two types of Markov chains on a weighted pure $d$-dimensional simplicial complex $(\complex, \complexWeight)$.
(1) The \emph{two-step random walk} $\twoStep[\complex, \complexWeight]$, which is a Markov chain on the state space $\complex[d]$.
Let $\complexFace_t \in \complex[d]$ be the state of $\twoStep[\complex, \complexWeight]$ at time $t \in \N$, then $\complexFace_{t+1}$ is chosen according the following transition rule:
\begin{enumerate}
	\item choose $x \in \complexFace_t$ uniformly at random, let $\complexFace' = \complexFace_t \setminus \{x\}$, and
	\item choose $\complexFace_{t+1} \in \set{\complexFace \in \complex[d]}{\complexFace' \subset \complexFace}$ proportionally to their weights $\complexWeight[\complexFace]$.
\end{enumerate}

(2) The \emph{$1$-skeleton of $(\complex, \complexWeight)$} is an edge-weighted graph with vertices $V_{\complex} = \set{x \in \complexGroundset}{ \{x\} \in \complex}[\big]$, edges $E_{\complex} = \set{(x, y) \in V_{\complex}^2}{ \{x, y\} \in \complex}$, and weights $\complexWeight[\{x, y\}]$.
The \emph{skeleton walk on $(\complex, \complexWeight)$}, denoted by $\skeletonWalk[\complex, \complexWeight]$, is the non-lazy random walk on its $1$-skeleton.

For a face $\complexFace \in \complex$, the \emph{link} of $\complexFace$ is a weighted pure $(d - \size{\complexFace})$-dimensional simplicial complex $(\complex[][\complexFace], \complexWeight[][\complexFace])$, where $\complex[][\complexFace] = \set{\complexFace' \setminus \complexFace}{ \complexFace' \in \complex, \complexFace \subseteq \complexFace'}$ and, for all $\complexFace' \in \complex$, we have $\complexWeight[\complexFace'][\complexFace] = \complexWeight[\complexFace' \cup \complexFace]$.

\begin{definition}[local expander]
	\label{def:local_expander}
	Let $(\complex, \complexWeight)$ be a weighted pure $d$-dimensional simplicial complex, and let $\alpha \in \R_{> 0}$.
	We say that a face $\complexFace \in \complex$ is a \emph{local $\alpha$-expander} if and only if the second largest eigenvalue of its skeleton walk $\skeletonWalk_{\complexFace} = \skeletonWalk[\complex[][\complexFace], \complexWeight[][\complexFace]]$ is at most $\alpha$ (i.e., $\eigenvalueOf{\transitionMatrix[\skeletonWalk_{\complexFace}]}[2] \le \alpha$).
	Further, we say $(\complex, \complexWeight)$ is a \emph{local $(\alpha_0, \dots, \alpha_{d-2})$-expander} if and only if, for all $k \in \{0\} \cup [d-2]$, each face $\complexFace \in \complex[k]$ is a local $\alpha_k$-expander.
\end{definition}

In \cite{alev2020improved} the authors relate local expansion and two-step walks. We use the following formulation of their result.
\begin{theorem}[{\cite[Theorem~$1.3$]{ALOG20}}]
	\label{lemma:twoStepMixing}
	Let $(\complex, \complexWeight)$ be a weighted pure $d$-dimensional simplicial complex.
	If $(\complex, \complexWeight)$ is a local $(\alpha_0, \dots, \alpha_{d-2})$-expander, then for the second-largest eigenvalue of the two-step walk $\twoStep = \twoStep[\complex, \complexWeight]$, it holds that
	\[
		\eigenvalueOf{\transitionMatrix[\twoStep]}[2] \le 1 - \frac{1}{d} \prod\nolimits_{k \in \{0\} \cup [d-2]} (1 - \alpha_k) .
		\qedhere
	\]	
\end{theorem}


\section{Mixing time of block dynamics for clique covers} \label{sec:block_dynamics}

\subsection{Simplicial-complex representation} \label{subsec:complex_representation}
Let $(\graph, \weights)$ be an instance of the multivariate hard-core model and let $\clique$ be a disjoint clique cover of size $\numberOfCliques$.
We construct the \emph{simplicial-complex representation} as follows.
For each clique $\clique[i] \in \clique$, we have a set $\complexPartition{i}$ that consists of an element $\complexEmpty{i} \in \complexPartition{i}$ and one element $\complexElement{v} \in \complexPartition{i}$ for each vertex $v \in \clique[i]$.
The ground set of the simplicial complex is $\complexGroundset = \bigcup_{i \in [\numberOfCliques]} \complexPartition{i}$.
Further, the complex $\complex$ contains a face $\faceFromIndep{\independentSet} \in \complex$ for each independent set $\independentSet \in \independentSets$ where,
\begin{itemize}
	\item for each $i \in [\numberOfCliques]$ and $\complexElement{v} \in \complexPartition{i}$, we have $\complexElement{v} \in \faceFromIndep{\independentSet}$ if and only if $v \in \independentSet$, and,
	\item for each $i \in [\numberOfCliques]$, we have $\complexEmpty{i} \in \faceFromIndep{\independentSet}$ if and only if $\independentSet \cap \clique[i] = \emptyset$ .
\end{itemize}
Note that each independent set contains at most one vertex $v \in \clique[i]$ for each clique in the clique cover $\clique[i] \in \clique$.
As~$\clique$ is a disjoint cover, each of the faces $\faceFromIndep{\independentSet} \in \complex$ contains exactly one element from each $\complexPartition{i}$ for $i \in [\numberOfCliques]$.  
We complete $\complex$ by taking the downward closure of these faces.
We make the following observation.
\begin{observation}
	\label{obs:dimension}
	The simplicial-complex representation $\complex$ for an instance of the multivariate hard-core model $(\graph, \weights)$ with disjoint clique cover $\clique$ of size $\numberOfCliques$ is pure $\numberOfCliques$-dimensional and $\numberOfCliques$-partite with partition $\{\complexPartition{i}\}_{i \in [\numberOfCliques]}$ as constructed above.
	Further, there is a one-to-one correspondence between the independent sets of $\graph$ and the maximum faces $\complex[\numberOfCliques]$.  
\end{observation}

We continue by equipping $\complex[\numberOfCliques]$ with weights, which induces weights for all other faces in $\complex$ as well.
For a face $\faceFromIndep{\independentSet} \in \complex[\numberOfCliques]$, corresponding to the independent set $\independentSet \in \independentSets$, we set $\complexWeight[\faceFromIndep{\independentSet}] = \gibbsDistributionFunction{\independentSet}$.
We now observe the following relation to block dynamics.
\begin{observation}
	\label{observation:twoStep_transitions}
	Let $(\graph, \weights)$ be an instance of the multivariate hard-core model, and let $\clique$ be a disjoint clique cover of $\graph$.
	Further, let $(\complex, \complexWeight)$ be the corresponding simplicial-complex representation.
	It holds that the two-step walk $\twoStep[\complex, \weights]$ is equivalent to the block dynamics $\blockDynamics[\graph, \weights, \clique]$ in the sense that there is a bijection between both state spaces that preserves transition probabilities.
	Consequently, $\twoStep[\complex, \weights]$ is ergodic, reversible and has a unique stationary distribution $\stationary[\twoStep][\faceFromIndep{\independentSet}] = \gibbsDistributionFunction{\independentSet}$ for every maximum face $\faceFromIndep{\independentSet} \in \complex[m]$, corresponding to the independent set $\independentSet \in \independentSets[\graph]$. 
\end{observation}

Based on \Cref{obs:dimension}, applying \Cref{lemma:twoStepMixing}, we obtain a lower bound on the spectral gap of $\twoStep = \twoStep[\complex, \complexWeight]$ in terms of local expansion.
Moreover, for an independent set $\independentSet \sim \gibbsDistribution$, it holds that
\[
	\Pr{\complexElement{v} \in \faceFromIndep{\independentSet}} = \gibbsPr{\graph}{\inSet{v}} \text{ for all $v \in \vertices$ and }
	\Pr{\complexEmpty{i} \in \faceFromIndep{\independentSet}} =  \gibbsPr{\graph}{\bigcap_{v \in \clique[i]} \notInSet{v}} \text{ for all $i \in [\numberOfCliques]$} .
\]
For simplicity, we also write $\gibbsPr{\graph}{\complexElement{v}}$ and $\gibbsPr{\graph}{\complexEmpty{i}}$ for these probabilities.

\subsection{Bounding local expansion by clique influence}
Let $(\graph, \weights)$ be an instance of the multivariate hard-core model with disjoint clique cover $\clique$ of size $\numberOfCliques$.
Further, let $(\complex, \complexWeight)$ be the resulting simplicial-complex representation with ground set $\complexGroundset$ and partition $\{\complexPartition{i}\}_{i \in [\numberOfCliques]}$.
The clique influence matrix $\cliqueInfluence{\graph}{\clique}$ contains an entry $\cliqueInfluence{\graph}{\clique}[x, y]$ for each $x, y \in \complexGroundset$ with
\begin{equation}\label{eq:clique_influence}
	\cliqueInfluence{\graph}{\clique}[x, y] = \begin{cases}
		0 & \text{ if $x,y \in \complexPartition{i}$ for some $i \in [\numberOfCliques]$, } \\
		\gibbsPr{\graph}{y}[x] - \gibbsPr{\graph}{y} & \text{ otherwise. }
	\end{cases}
\end{equation}

Note that this definition includes the cases where $x \in \set{\complexEmpty{i}}{i \in [\numberOfCliques]}$ or $y \in \set{\complexEmpty{i}}{i \in [\numberOfCliques]}$.
The following lemma and its proof are an adapted version of \cite[Theorem~$8$]{Chen2020coloring}.

\begin{lemma}
	\label{lemma:eigenvalues_skeleton}
	Let $(\graph, \weights)$ be an instance of the multivariate hard-core model with a disjoint clique cover $\clique$ of size $\numberOfCliques$.
	Further, let $(\complex, \complexWeight)$ be the resulting simplicial-complex representation.
	Denote by $\skeletonWalk = \skeletonWalk[\complex, \complexWeight]$ be the skeleton walk on $(\complex, \complexWeight)$, and let $\cliqueInfluence{\graph}{\clique}$ be the clique influence matrix as defined in \cref{eq:clique_influence}.
	Then
	\begin{equation}
        \label{eq:eigenvalueBound}
		\eigenvalueOf{\transitionMatrix[\skeletonWalk]}[2] \leq \frac{1}{\numberOfCliques - 1} \eigenvalueOf{\cliqueInfluence{\graph}{\clique}}[1].
		\qedhere 
	\end{equation}
\end{lemma}

\begin{proof}
	We first take a detailed look at the entries of the transitions $\transitionMatrix[\skeletonWalk]$.
	Note that by definition $\stateSpace[\skeletonWalk] = \complexGroundset$ and $\transitionMatrix[\skeletonWalk][x, x] = 0$ for all $x \in \complexGroundset$.
	Further, it holds for each $x \in \complexGroundset$ that
	\[
		\sum_{\substack{z \in \complexGroundset: \\ z \neq x}} \complexWeight[\{x, z\}] 
		= \sum_{\substack{z \in \complexGroundset: \\ z \neq x}} \gibbsPr{\graph}{x, z} 
		= \sum_{i \in [\numberOfCliques]} \sum_{\substack{z \in \complexPartition{i}: \\ z \neq x}} \gibbsPr{\graph}{x, z}
		= \sum_{\substack{i \in [\numberOfCliques]: \\ x \notin \complexPartition{i}}} \sum_{z \in \complexPartition{i}} \gibbsPr{\graph}{x, z}
		= (\numberOfCliques - 1) \gibbsPr{\graph}{x} ,
	\]
	where the third equality comes from the fact that for all $i \in [\numberOfCliques]$ and every pair $x, z \in \complexPartition{i}$ it holds that $\gibbsPr{\graph}{x, z} = 0$.
	Thus, we get for each $y \in \complexGroundset$ with $y \neq x$ the transition probability
	\[
		\transitionMatrix[\skeletonWalk][x, y] 
		= \frac{\complexWeight[\{x, y\}]}{\sum_{\substack{z \in \complexGroundset: \\ z \neq x}} \complexWeight[\{x, z\}]}
		= \frac{\gibbsPr{\graph}{x, y}}{(\numberOfCliques - 1) \gibbsPr{\graph}{x}}
		= \frac{1}{\numberOfCliques - 1} \gibbsPr{\graph}{y}[x] .
	\] 
	Note that this especially implies $\transitionMatrix[\skeletonWalk][x, y] = 0$ if $x, y \in \complexPartition{i}$ for some $i \in [\numberOfCliques]$.
	
	Let $\mathbold{D}$ be the matrix with $\mathbold{D}(x, x) = \frac{1}{\numberOfCliques} \gibbsPr{\graph}{x}$ for each $x \in \complexGroundset$ and $0$ everywhere else, and let~$\mathbold{d}$ be the vector of its diagonal, that is, for all $x \in \complexGroundset$, we have $\mathbold{d}(x) = \mathbold{D}(x, x)$.
	Note that~$\skeletonWalk$ satisfies the detailed-balance equations with~$\mathbold{d}$, that is, it is reversible with respect to~$\mathbold{d}$.
    Thus, $\mathbold{A} = \mathbold{D}^{1/2} \transitionMatrix[\skeletonWalk] \mathbold{D}^{-1/2}$ is symmetric.
	This implies that for each eigenvector $\mathbold{z}$ of $\mathbold{A}$ with eigenvalue $\eigenvalue$ there is a left eigenvector $\transposed{\mathbold{z}} \mathbold{D}^{1/2} = \transposed{\left(\mathbold{D}^{1/2} \mathbold{z}\right)}$ and a right eigenvector $\mathbold{D}^{-1/2} \mathbold{z}$ of $\transitionMatrix[\skeletonWalk]$ for the same eigenvalue~$\eigenvalue$.
	Thus, if $\mathbold{z}' = \mathbold{D}^{-1/2} \mathbold{z}$ is such a right eigenvector of $\transitionMatrix[\skeletonWalk]$, then $\mathbold{D} \mathbold{z}' = \mathbold{D} \mathbold{D}^{-1/2} \mathbold{z} = \mathbold{D}^{1/2} \mathbold{z}$ yields a (transposed) left eigenvector for the same eigenvalue.
	
	We investigate the eigenvalues of $\transitionMatrix[\skeletonWalk]$ more carefully.
	Consider the column vector $\mathbold{1}$ with $\mathbold{1}(x) = 1$ for all $x \in \complexGroundset$.
	Note that $\mathbold{1}$ is a right eigenvector of $\transitionMatrix[\skeletonWalk]$ for eigenvalue $1$, which is the maximum eigenvalue, since~$\transitionMatrix[\skeletonWalk]$ is a transition matrix.
	We denote by $\mathbold{\pi} = \transposed{\left( \mathbold{D} \mathbold{1} \right)}$ the corresponding left eigenvector with $\mathbold{\pi}(x) = \frac{1}{\numberOfCliques} \gibbsPr{\graph}{x}$ for all $x \in \complexGroundset$.
	Further, we define a set of column vectors $\{\mathbold{1}_i\}_{i \in [\numberOfCliques]}$ and a set of row vectors $\{\mathbold{\pi}_i\}_{i \in [\numberOfCliques]}$ such that for each $i \in [\numberOfCliques]$ and each $x \in \complexGroundset$ it holds that
	\[
		\mathbold{1}_i (x) = \begin{cases}
			1 & \text{ if $x \in \complexPartition{i}$,} \\
			0 & \text{ otherwise}
		\end{cases}
		\text{\hspace*{1 em} and \hspace*{1 em} }
		\mathbold{\pi}_i (x) = \begin{cases}
		\frac{1}{\numberOfCliques} \gibbsPr{\graph}{x} & \text{ if $x \in \complexPartition{i}$,} \\
		0 & \text{ otherwise.}
		\end{cases}
	\] 
	
	Note that for each $i, j \in [\numberOfCliques]$, $i \neq j$, and each $x \in \complexPartition{i}$ it holds that $\sum_{y \in \complexPartition{j}} \transitionMatrix[\skeletonWalk][x, y] = \frac{1}{\numberOfCliques - 1}$ and $\sum_{y \in \complexPartition{i}} \transitionMatrix[\skeletonWalk][x, y] = 0$.
	Thus, for all $i \in [\numberOfCliques]$, we have $\transitionMatrix[\skeletonWalk] \mathbold{1}_{i} = \frac{1}{\numberOfCliques - 1} (\mathbold{1} - \mathbold{1}_{i})$.
	It follows that
	\[
		\transitionMatrix[\skeletonWalk] \left( \frac{1}{\numberOfCliques} \mathbold{1} - \mathbold{1}_i \right)
		= \frac{1}{\numberOfCliques} \mathbold{1} - \frac{1}{\numberOfCliques - 1} \left( \mathbold{1} - \mathbold{1}_i \right)
		= - \frac{1}{\numberOfCliques (\numberOfCliques - 1)} \mathbold{1} + \frac{1}{\numberOfCliques - 1} \mathbold{1}_i
		= - \frac{1}{\numberOfCliques - 1} \left( \frac{1}{\numberOfCliques} \mathbold{1} - \mathbold{1}_i \right) ,
	\]  
	which shows that for each $i \in [\numberOfCliques]$ the vector $\frac{1}{\numberOfCliques} \mathbold{1} - \mathbold{1}_i$ is a right eigenvector with eigenvalue $- \frac{1}{\numberOfCliques - 1}$.
	Similarly, the vector $\transposed{\left( \mathbold{D} \left( \frac{1}{\numberOfCliques} \mathbold{1} - \mathbold{1}_i \right) \right)} = \frac{1}{\numberOfCliques} \mathbold{\pi} - \mathbold{\pi}_i$ is a left eigenvector for this eigenvalue.
	
	We use these vectors to construct an eigenbasis of $\transitionMatrix[\skeletonWalk]$.
	Let $i \in [\numberOfCliques]$, and consider the set
	\[
		S = \{\mathbold{1}\} \cup \bigcup_{\substack{j \in [\numberOfCliques]: \\j \neq i}} \left\{\frac{1}{\numberOfCliques} \mathbold{1} - \mathbold{1}_{j}\right\}.
	\]
	Note that $S$ is a set of~$\numberOfCliques$ linearly independent right eigenvectors of $\transitionMatrix[\skeletonWalk]$.
	By the relation between eigenvectors of $\transitionMatrix[\skeletonWalk]$ and $\mathbold{A}$, we construct a set $S_{\mathbold{A}} = \set{\mathbold{D}^{1/2} \mathbold{z}}{\mathbold{z} \in S}$ of independent eigenvectors of $\mathbold{A}$.
	As~$\mathbold{A}$ is symmetric, such a set can always be extended to an eigenbasis~$\overline{S_{\mathbold{A}}}$ of $\mathbold{A}$, such that the vectors in $\overline{S_{\mathbold{A}}} \setminus S_{\mathbold{A}}$ are orthogonal to the vectors in $S_{\mathbold{A}}$.
	This gives us an eigenbasis $\overline{S} = \set{\mathbold{D}^{-1/2} \mathbold{z}}{\mathbold{z} \in \overline{S_{\mathbold{A}}}}$ of right eigenvectors of  $\transitionMatrix[\skeletonWalk]$.
	
	We proceed by relating the eigenvalues of $\transitionMatrix[\skeletonWalk]$ to the eigenvalues of $\cliqueInfluence{\graph}{\clique}$, using~$\overline{S}$. Note that both~$\transitionMatrix[\skeletonWalk]$ and~$\cliqueInfluence{\graph}{\clique}$ are $(\numberOfVertices + \numberOfCliques) \times (\numberOfVertices + \numberOfCliques)$ matrices.
	We first show (Claim~$1$) that all vectors of $S$ are in the kernel of $\cliqueInfluence{\graph}{\clique}$. Since $|S| = \numberOfCliques + 1$ and since the vectors of~$S$ are linearly independent, the kernel of~$\cliqueInfluence{\graph}{\clique}$ has a dimension of at least $\numberOfCliques + 1$. Thus,~$\cliqueInfluence{\graph}{\clique}$ has at least $\numberOfCliques + 1$ eigenvectors associated with the eigenvalue~$0$.
    Then (Claim~$2$) we show that all vectors of $\overline{S} \setminus S$, which are the remaining eigenvectors of~$\transitionMatrix[\skeletonWalk]$ in our consideration, are also right eigenvectors of $\cliqueInfluence{\graph}{\clique}$ but with eigenvalues scaled by $(\numberOfCliques - 1)$.
    Last (Claim~$3$), we conclude that \cref{eq:eigenvalueBound} holds.
	
	\textbf{Claim~$1$.}
    Let $\mathbold{z} \in S$, and let $\mathbold{0}$ denote the vector with $\mathbold{0}(x) = 0$ for all $x \in \complexGroundset$.
	If $\mathbold{z} = \mathbold{1}$, then $\cliqueInfluence{\graph}{\clique} \mathbold{z} = \mathbold{0}$ because for all $j \in [\numberOfCliques]$ and every $x \in \complexPartition{j}$ it holds that
	\[
		\cliqueInfluence{\graph}{\clique} \mathbold{1} (x) 
		= \sum_{\substack{k \in [\numberOfCliques]: \\k \neq j}} \sum_{y \in \complexPartition{k}} \left( \gibbsPr{\graph}{y}[x] - \gibbsPr{\graph}{y} \right)
		= \sum_{\substack{k \in [\numberOfCliques]: \\k \neq j}} \sum_{y \in \complexPartition{k}} \gibbsPr{\graph}{y}[x] - \sum_{\substack{k \in [\numberOfCliques]: \\k \neq j}} \sum_{y \in \complexPartition{k}} \gibbsPr{\graph}{y}
		= 0 .
	\]
    
	If $\mathbold{z} = \mathbold{1}_{j}$ for some $j \in [\numberOfCliques]$, again, we have $\cliqueInfluence{\graph}{\clique} \mathbold{z} = \mathbold{0}$ because for all $k \in [\numberOfCliques]$ and every $x \in \complexPartition{k}$ it holds that
	\[
		\cliqueInfluence{\graph}{\clique} \mathbold{1}_{j} (x) 
		= \sum_{\substack{l \in [\numberOfCliques]: \\l \neq k}} \sum_{y \in \complexPartition{l}} \cliqueInfluence{\graph}{\clique}(x, y) \indicator{k = j, l = j}
		= 0 .
	\] 
	
    \textbf{Claim~$2$.}    
    We first show that all vectors $\mathbold{z} \in \overline{S} \setminus S$ are orthogonal to $\frac{1}{\numberOfCliques} \mathbold{\pi} - \mathbold{\pi}_{j}$ for all $j \in [\numberOfCliques]$.
    Let $\mathbold{z} \in \overline{S} \setminus S$.
    First, note that 
    \[
        \mathbold{\pi} \mathbold{z} 
        = \transposed{\left( \mathbold{D} \mathbold{1} \right)} \mathbold{z}
        = \transposed{\mathbold{1} } \mathbold{D} \mathbold{z}
        = \transposed{\left( \mathbold{D}^{1/2} \mathbold{1} \right)} \mathbold{D}^{1/2} \mathbold{z}
        = 0 ,
    \]
    where the last equality is due to $\mathbold{D}^{1/2} \mathbold{1} \in S_{\mathbold{A}}$ and $\mathbold{D}^{1/2} \mathbold{z} \in \overline{S_{\mathbold{A}}} \setminus S_{\mathbold{A}}$.
    Similarly, we obtain for each $j \in [\numberOfCliques]$ with $j \neq i$ that
    \[
        \left( \frac{1}{\numberOfCliques} \mathbold{\pi} - \mathbold{\pi}_{j} \right) \mathbold{z} 
        = \transposed{\left( \mathbold{D}^{1/2} \left( \frac{1}{\numberOfCliques} \mathbold{1} - \mathbold{1}_{j} \right) \right)} \mathbold{D}^{1/2} \mathbold{z}
        = 0	. 
    \]
    Finally, note that $\frac{1}{\numberOfCliques} \mathbold{\pi} - \mathbold{\pi}_{i}$ can be obtain as a linear combination from $\mathbold{\pi}$ and the vectors $\frac{1}{\numberOfCliques} \mathbold{\pi} - \mathbold{\pi}_{j}$ for $j \neq i$, implying that it is orthogonal to $\mathbold{z}$ as well.
    
    Assume that~$\mathbold{z}$ has eigenvalue $\eigenvalue$.
    We define the matrix $\mathbold{\Pi} = \mathbold{1}\mathbold{\pi}$ and the matrices $\mathbold{\Pi}_{j} = \mathbold{1}_{j} \mathbold{\pi}_{j}$ for $j \in [\numberOfCliques]$, and note that
    \[
        \cliqueInfluence{\graph}{\clique} = (\numberOfCliques - 1) \transitionMatrix[\skeletonWalk] - \numberOfCliques \left( \mathbold{\Pi} - \sum_{j \in [\numberOfCliques]} \mathbold{\Pi}_{j} \right) .
    \]
    
	Since~$\mathbold{z}$ is orthogonal to all vectors $\frac{1}{\numberOfCliques} \mathbold{\pi} - \mathbold{\pi}_{j}$ for $j \in [\numberOfCliques]$, we have for every $k \in [\numberOfCliques]$ and $x \in \complexPartition{k}$ that
	\[
		\left( \left( \mathbold{\Pi} - \sum_{j \in [\numberOfCliques]} \mathbold{\Pi}_{j} \right) \mathbold{z} \right) (x)
		= \sum_{y \in \complexGroundset} \sum_{j \in [\numberOfCliques]} \left( \left( \frac{1}{\numberOfCliques} \mathbold{\Pi} (x, y) - \mathbold{\Pi}_{j} (x, y) \right) \mathbold{z} \right) (y) 
		= \sum_{j \in [\numberOfCliques]} \left( \frac{1}{\numberOfCliques} \mathbold{\pi} - \mathbold{\pi}_{j} \right) \mathbold{z}
		= 0 .
	\]
	This implies that $\left( \mathbold{\Pi} - \sum_{j \in [\numberOfCliques]} \mathbold{\Pi}_{j} \right) \mathbold{z} = \mathbold{0}$, and thus it holds that
	\[
		\cliqueInfluence{\graph}{\clique} \mathbold{z} = (\numberOfCliques - 1) \transitionMatrix[\skeletonWalk] \mathbold{z} = (\numberOfCliques - 1) \eigenvalue \mathbold{z}.
	\]
	
    \textbf{Claim~$3$.}
    Recall that $\eigenvalueOf{\transitionMatrix[\skeletonWalk]}[1] = 1$ and that all eigenvectors of~$\transitionMatrix[\skeletonWalk]$ from $S \setminus \{\mathbold{1}\}$ have a negative eigenvalue.
    We make a case distinction with respect to the sign of $\eigenvalueOf{\transitionMatrix[\skeletonWalk]}[2] = \beta$.
    If $\beta > 0$, then there is an eigenvector $\mathbold{z} \in \overline{S} \setminus S$ corresponding to~$\beta$, as the eigenvalues of vectors from $S \setminus \{\mathbold{1}\}$ are negative.
    By Claim~$2$, there is an eigenvalue~$\beta'$ of~$\cliqueInfluence{\graph}{\clique}$ such that $\beta = \beta'/(\numberOfCliques - 1)$.
    Since $\beta' \leq \eigenvalueOf{\cliqueInfluence{\graph}{\clique}}[1]$, \cref{eq:eigenvalueBound} holds.
    
    If $\beta \leq 0$, then \cref{eq:eigenvalueBound} follows immediately, as the kernel of~$\cliqueInfluence{\graph}{\clique}$ is nontrivial and, thus, $\eigenvalueOf{\cliqueInfluence{\graph}{\clique}}[1] \geq 0$.
    This concludes the proof.
\end{proof}

\subsection{Bounding clique influence}
We prove an upper bound for $\eigenvalueOf{\cliqueInfluence{\graph'}{\clique}}[1]$ for all induced subgraphs~$\graph'$ of~$\graph$ and every disjoint clique cover $\clique$ of $\graph'$, given that $(\graph, \weights)$ satisfy \Cref{condition:influence_bound}.
\cliqueInfluenceBound*

In the proof of \Cref{lemma:clique_influence_bound}, we apply the following lemma, based on a result in \cite{Chen2020rapid}.
\begin{lemma}
	\label{lemma:spectral_radius}
	Let $n \in \N$, let $A \in \C^{n \times n}$, and let $\spectralRadius{A}$ denote the spectral radius of~$A$.
	Assume that there is a $\xi \in \R$ and a $p\colon [n] \to \R_{>0}$ such that for all $i \in [n]$ it holds that
	$
	\sum_{j \in [n]} \absolute{A(i, j)} p(j) \le \xi p(i)
	$.
	Then $\spectralRadius{A} \le \xi$.
\end{lemma}

Note that, by \Cref{lemma:spectral_radius}, \Cref{condition:influence_bound} implies $\eigenvalueOf{\pairwiseInfluence{\subgraph{\graph}{S}}}[1] \le C$ for all $S \subseteq \vertices$.
  
We show that \Cref{condition:influence_bound} implies the existence of a~$\xi$ from \Cref{lemma:spectral_radius} such that for all induced subgraphs $\graph'$ of $\graph$ and every disjoint clique cover $\clique$ of $\graph'$ there is a function $p$ that satisfies the conditions of \Cref{lemma:spectral_radius} for $\cliqueInfluence{\graph'}{\clique}$.
To this end, we use the following lemmas.

\begin{lemma}
	\label{lemma:influence_on_clique}
	Let $(\graph, \weights)$ be an instance of the multivariate hard-core model with clique cover~$\clique$ of size $\numberOfCliques$.
	Further, let $(\complex, \complexWeight)$ be the corresponding simplicial-complex representation with ground set $\complexGroundset$ and partition $\{\complexPartition{i}\}_{i \in [\numberOfCliques]}$.
	For all $i, j \in [\numberOfCliques]$ and $x \in \complexPartition{i}$ it holds that 
	\[
	\cliqueInfluence{\graph}{\clique}[x, \complexEmpty{j}] = - \sum_{v \in \clique[j]} \cliqueInfluence{\graph}{\clique}[x, \complexElement{v}] .
	\qedhere
	\]
\end{lemma}

\begin{proof}
	By definition,
	\[
		\cliqueInfluence{\graph}{\clique}[x, \complexEmpty{j}] 
		= \gibbsPr{\graph}{\bigcap\nolimits_{w \in \clique[j]} \notInSet{w}}[x] - \gibbsPr{\graph}{\bigcap\nolimits_{w \in \clique[j]} \notInSet{w}}
		= - \left( \gibbsPr{\graph}{\bigcup\nolimits_{w \in \clique[j]} \inSet{w}}[x] - \gibbsPr{\graph}{\bigcup\nolimits_{w \in \clique[j]} \inSet{w}} \right) .
	\]
	Note that for any pair of vertices from the same clique $w_1, w_2 \in \clique[j]$ with $w_1 \neq w_2$ the events that $w_1$ is in an independent set and that $w_2$ is in an independent set are disjoint.
	Thus, we obtain
	\[
		- \left( \gibbsPr{\graph}{\bigcup\nolimits_{w \in \clique[j]} \inSet{w}}[x] - \gibbsPr{\graph}{\bigcup\nolimits_{w \in \clique[j]} \inSet{w}} \right)
		= - \sum_{w \in \clique[j]} \left( \gibbsPr{\graph}{\inSet{w}}[x] - \gibbsPr{\graph}{\inSet{w}} \right)
		= - \sum_{w \in \clique[j]} \cliqueInfluence{\graph}{\clique}[x, w] .
		\qedhere  
	\]
\end{proof}

\begin{lemma}
	\label{lemma:clique_to_pairwise_1}
	Let $(\graph, \weights)$ be an instance of the multivariate hard-core model with clique cover~$\clique$ of size $\numberOfCliques$.
	Further, let $(\complex, \complexWeight)$ be the corresponding simplicial-complex representation with ground set $\complexGroundset$ and partition $\{\complexPartition{i}\}_{i \in [\numberOfCliques]}$.
	For all $i, j \in [\numberOfCliques]$ with $i \neq j$ and all $v \in \clique[i], w \in \clique[j]$ it holds that
	\[
		\cliqueInfluence{\graph}{\clique}[\complexElement{v}, \complexElement{w}] = \gibbsPr{\graph}{\notInSet{v}} \pairwiseInfluence{\graph}[][v, w] .
		\qedhere
	\]
\end{lemma}

\begin{proof}
	By the law of total probability,
	\begin{align*}
		\cliqueInfluence{\graph}{\clique}[\complexElement{v}, \complexElement{w}] 
		&= \gibbsPr{\graph}{\inSet{w}}[\inSet{v}] - \gibbsPr{\graph}{\inSet{w}} \\
		&= \gibbsPr{\graph}{\inSet{w}}[\inSet{v}] - \gibbsPr{\graph}{\inSet{w}}[\inSet{v}] \gibbsPr{\graph}{\inSet{v}} - \gibbsPr{\graph}{\inSet{w}}[\notInSet{v}] \gibbsPr{\graph}{\notInSet{v}} \\
		&= \gibbsPr{\graph}{\notInSet{v}} \left( \gibbsPr{\graph}{\inSet{w}}[\inSet{v}] - \gibbsPr{\graph}{\inSet{w}}[\notInSet{v}] \right) \\
		&= \gibbsPr{\graph}{\notInSet{v}} \pairwiseInfluence{\graph}[][v, w] .
		\qedhere
	\end{align*}
\end{proof}

\begin{lemma}
	\label{lemma:clique_to_pairwise_2}
	Let $(\graph, \weights)$ be an instance of the multivariate hard-core model with clique cover~$\clique$ of size $\numberOfCliques$.
	Further, let $(\complex, \complexWeight)$ be the corresponding simplicial-complex representation with ground set $\complexGroundset$ and partition $\{\complexPartition{i}\}_{i \in [\numberOfCliques]}$.
	For all $i, j \in [\numberOfCliques]$ with $i \neq j$ and every $w \in \clique[j]$ it holds that
	\[
	\cliqueInfluence{\graph}{\clique}[\complexEmpty{i}, \complexElement{w}]
		= \sum_{v \in \clique[i]} \gibbsPr{\graph}{v} \pairwiseInfluence{\graph}[\zeroSpinConfig[\clique[i] \setminus \{v\}]][v, w]
		= \sum_{v \in \clique[i]} \gibbsPr{\graph}{v} \pairwiseInfluence{\graph_v}[][v, w] ,
	\]
	where $\graph_v = \subgraph{\graph}{\vertices \setminus (\clique[i] \setminus \{v\})}$.
\end{lemma}

\begin{proof}
	Let $\overline{\complexEmpty{i}}$ denote the complementary event to $\complexEmpty{i}$, meaning that some vertex $u \in \clique[i]$ is an independent set drawn from the Gibbs distribution.
	By the law of total probability,
	\begin{align*}
		\cliqueInfluence{\graph}{\clique}[\complexEmpty{i}, \complexElement{w}] 
		&= \gibbsPr{\graph}{\inSet{w}}[\complexEmpty{i}] - \gibbsPr{\graph}{\inSet{w}} \\
		&= \gibbsPr{\graph}{\inSet{w}}[\complexEmpty{i}] - \gibbsPr{\graph}{\inSet{w}}[\complexEmpty{i}] \gibbsPr{\graph}{\complexEmpty{i}} - \gibbsPr{\graph}{\inSet{w}}[\overline{\complexEmpty{i}}] \gibbsPr{\graph}{\overline{\complexEmpty{i}}} \\
		&= \gibbsPr{\graph}{\overline{\complexEmpty{i}}} \left( \gibbsPr{\graph}{\inSet{w}}[\complexEmpty{i}] - \gibbsPr{\graph}{\inSet{w}}[\overline{\complexEmpty{i}}] \right) \\
		&= \gibbsPr{\graph}{\bigcup\nolimits_{u \in \clique[i]} \inSet{u}} \left( \gibbsPr{\graph}{\inSet{w}}[\bigcap\nolimits_{u \in \clique[i]} \notInSet{u}] - \gibbsPr{\graph}{\inSet{w}}[\bigcup\nolimits_{u \in \clique[i]} \inSet{u}] \right) .
	\end{align*}
	Because the events that two distinct vertices from the same clique are in an independent set are disjoint, we get 
	\begin{align*}
		&\gibbsPr{\graph}{\inSet{w}}[\bigcup\nolimits_{u \in \clique[i]} \inSet{u}] 
		= \sum_{v \in \clique[i]} \gibbsPr{\graph}{\inSet{w}}[\inSet{v}] \frac{\gibbsPr{\graph}{\inSet{v}}}{\gibbsPr{\graph}{\bigcup_{u \in \clique[i]} \inSet{u}}} \textrm{ and} \\
		\gibbsPr{\graph}{\inSet{w}}[\bigcap\nolimits_{u \in \clique[i]} \notInSet{u}] &= \gibbsPr{\graph}{\inSet{w}}[\bigcap\nolimits_{u \in \clique[i]} \notInSet{u}] \frac{\gibbsPr{\graph}{\bigcup_{v \in \clique[i]} \inSet{v}}}{\gibbsPr{\graph}{\bigcup_{u \in \clique[i]} \inSet{u}}} = \sum_{v \in \clique[i]} \gibbsPr{\graph}{\inSet{w}}[\bigcap\nolimits_{u \in \clique[i]} \notInSet{u}] \frac{\gibbsPr{\graph}{\inSet{v}}}{\gibbsPr{\graph}{\bigcup_{u \in \clique[i]} \inSet{u}}} .
	\end{align*}
	Thus, we obtain
	\begin{align*}
		&\gibbsPr{\graph}{\bigcup\nolimits_{u \in \clique[i]} \inSet{u}} 
		\left( \gibbsPr{\graph}{\inSet{w}}[\bigcap\nolimits_{u \in \clique[i]} \notInSet{u}] - \gibbsPr{\graph}{\inSet{w}}[\bigcup\nolimits_{u \in \clique[i]} \inSet{u}] \right)\\
		&= \sum_{v \in \clique[i]} 
		\gibbsPr{\graph}{\inSet{v}} \left( \gibbsPr{\graph}{\inSet{w}}[\bigcap\nolimits_{u \in \clique[i]} \notInSet{u}] - \gibbsPr{\graph}{\inSet{w}}[\inSet{v}] \right) .  
	\end{align*}
	Note that for each $v \in \clique[i]$ it holds that
	\[
		\gibbsPr{\graph}{\inSet{w}}[\bigcap\nolimits_{u \in \clique[i]} \notInSet{u}] = \gibbsPr{\graph}{\inSet{w}}[\notInSet{v}, \bigcap\nolimits_{u \in \clique[i] \setminus \{v\}} \notInSet{u}] ,
	\]
	Further, because $v$ being in the independent set implies that no other vertex $u \in \clique[i]$ can be in the independent set too, it also holds that
	\[
		\gibbsPr{\graph}{\inSet{w}}[\inSet{v}] = \gibbsPr{\graph}{\inSet{w}}[\inSet{v}, \bigcap\nolimits_{u \in \clique[i] \setminus \{v\}} \notInSet{u}] .
	\] 
	Consequently, we conclude that
	\[
		\cliqueInfluence{\graph}{\clique}[\complexEmpty{i}, \complexElement{w}] 
		= \sum_{v \in \clique[i]} \gibbsPr{\graph}{\inSet{v}} \pairwiseInfluence{\graph}[\zeroSpinConfig[\clique[i] \setminus \{v\}]][v, w]
		= \sum_{v \in \clique[i]} \gibbsPr{\graph}{v} \pairwiseInfluence{\graph_v}[][v, w] .
		\qedhere
	\]
\end{proof}

\begin{lemma}
	\label{lemma:probability_bound}
	Let $(\graph, \weights)$ be an instance of the multivariate hard-core model, and let $v \in \vertices$ and $w \in \neighbors{v}[\graph]$.
	Then $\gibbsPr{\graph}{w} \le - \pairwiseInfluence{\graph}[][v, w] = \absolute{\pairwiseInfluence{\graph}[][v, w]}$.
\end{lemma}

\begin{proof}
	Since $w \in \neighbors{v}[\graph]$, it holds that $\pairwiseInfluence{\graph}[][v, w] = - \gibbsPr{\graph}{w}[\notInSet{v}]$.
	We conclude by noting that
	\[
		\gibbsPr{\graph}{w}[\notInSet{v}] 
		= \weight[w] \frac{\partitionFunction[\subgraph{\graph}{\vertices \setminus \neighborsClosed{w}[\graph]}]}{\partitionFunction[\subgraph{\graph}{\vertices \setminus \{v\}}]}
		\ge \weight[w] \frac{\partitionFunction[\subgraph{\graph}{\vertices \setminus \neighborsClosed{w}[\graph]}]}{\partitionFunction[\graph]}
		= \gibbsPr{\graph}{w} .
		\qedhere
	\]
\end{proof}

We now prove the main lemma of this subsection.

\begin{proof}[Proof of \Cref{lemma:clique_influence_bound}]
	To simplify notation, set $\graph' = \subgraph{\graph}{S}$ and $\numberOfCliques = \size{\clique}$.
	Let $(\complex, \complexWeight)$ be the simplicial-complex representation of $(\graph', \subgraph{\weights}{S})$ with clique cover $\clique$ and let $\complexGroundset$ be the corresponding ground set of $(\complex, \complexWeight)$ with partition $\{\complexPartition{i}\}_{i \in [\numberOfCliques]}$.
	
	As we aim to prove our claim using \Cref{lemma:spectral_radius}, we need to construct a function $p\colon \complexGroundset \to \R_{>0}$ such that for all $x \in \complexGroundset$ it holds that
	\[
		\sum_{y \in \complexGroundset} \absolute{\cliqueInfluence{\graph'}{\clique}[x, y]} p(y) \le (2 + C) C p(x) .
	\]
	To this end, we set $p(\complexElement{v}) = q(v)$ for all $v \in S$ and $p(\complexEmpty{i}) = \sum_{v \in \clique[i]} \gibbsPr{\graph'}{v} q(v)$ for all $i \in [\numberOfCliques]$.
	By \Cref{lemma:probability_bound} we have for all $i \in [\numberOfCliques]$ and $w \in \clique[i]$ that
	\[
		\sum_{v \in \clique[i]} \gibbsPr{\graph'}{v} q(v) 
		\le \gibbsPr{\graph'}{w} q(w) + \sum_{v \in \neighbors{w}[\graph']} \absolute{\pairwiseInfluence{\graph'}[][w, v]} q(v) ,
	\]
	which, by \Cref{condition:influence_bound}, implies 
	\begin{align}
		\label{eq:clique_influence_bound:1}
		p(\complexEmpty{i}) < (1 + C) q(w) .
	\end{align}
	
	Without loss of generality, assume $x \in \complexPartition{i}$ for some $i \in [\numberOfCliques]$. 
	Recall that by definition $\cliqueInfluence{\graph'}{\clique}[x, y] = 0$ for all $y \in \complexPartition{i}$.
	By \Cref{lemma:influence_on_clique}, we obtain
	\begin{align*}
	\sum_{y \in \complexGroundset} \absolute{\cliqueInfluence{\graph'}{\clique}[x, y]} p(y)
		&= \sum_{\substack{j \in [\numberOfCliques]: \\j \neq i}} \left(\absolute{\cliqueInfluence{\graph'}{\clique}[x, \complexEmpty{j}]} p(\complexEmpty{j}) 
			+ \sum_{w \in \clique[j]} \absolute{\cliqueInfluence{\graph'}{\clique}[x, \complexElement{w}]} p(\complexElement{w}) \right) \\
		&\le \sum_{\substack{j \in [\numberOfCliques]: \\j \neq i}} \sum_{w \in \clique[j]} \absolute{\cliqueInfluence{\graph'}{\clique}[x, \complexElement{w}]} \left( p(\complexEmpty{j}) + p(\complexElement{w}) \right) .
	\end{align*}
	Further, by our choice of $p$ and by \cref{eq:clique_influence_bound:1}, we obtain 
	\[
		\sum_{\substack{j \in [\numberOfCliques]: \\j \neq i}} \sum_{w \in \clique[j]} \absolute{\cliqueInfluence{\graph'}{\clique}[x, \complexElement{w}]} \left( p(\complexEmpty{j}) + p(\complexElement{w}) \right)
		<  (2 + C) \sum_{\substack{j \in [\numberOfCliques]: \\j \neq i}} \sum_{w \in \clique[j]} \absolute{\cliqueInfluence{\graph'}{\clique}[x, \complexElement{w}]} q(w) .
	\]
	
	We proceed with a case distinction based on $x$.
	Assume that $x = \complexElement{v}$ for some $v \in \clique[i]$.
	By \Cref{lemma:clique_to_pairwise_1}, we have
	\begin{align*}
		(2 + C) \sum_{\substack{j \in [\numberOfCliques]: \\j \neq i}} \sum_{w \in \clique[j]} \absolute{\cliqueInfluence{\graph'}{\clique}[\complexElement{v}, \complexElement{w}]} q(w)
		&= (2 + C) \gibbsPr{\graph'}{\notInSet{v}} \sum_{\substack{j \in [\numberOfCliques]: \\j \neq i}} \sum_{w \in \clique[j]} \absolute{\pairwiseInfluence{\graph'}[][v, w]} q(w) \\
		&\le (2 + C) \sum_{\substack{j \in [\numberOfCliques]: \\j \neq i}} \sum_{w \in \clique[j]} \absolute{\pairwiseInfluence{\graph'}[][v, w]} q(w) .
	\end{align*}
	Using that the cliques are disjoint and applying \Cref{condition:influence_bound}, we get
	\[
		(2 + C) \sum_{\substack{j \in [\numberOfCliques]: \\j \neq i}} \sum_{w \in \clique[j]} \absolute{\pairwiseInfluence{\graph'}[][v, w]} q(w)
		\le (2 + C) \sum_{w \in S} \absolute{\pairwiseInfluence{\graph'}[][v, w]} q(w)
		\le (2 + C) C q(v)
		= (2 + C) C p(\complexElement{v}) .
	\]
	
	Now, assume that $x = \complexEmpty{i}$. 
	By \Cref{lemma:clique_to_pairwise_2}, we have
	\begin{align*}
		(2 + C) \sum_{\substack{j \in [\numberOfCliques]: \\j \neq i}} \sum_{w \in \clique[j]} \absolute{\cliqueInfluence{\graph'}{\clique}[\complexEmpty{i}, \complexElement{w}]} q(w)
		&= (2 + C) \sum_{\substack{j \in [\numberOfCliques]: \\j \neq i}} \sum_{w \in \clique[j]} \absolute{\sum_{v \in \clique[i]} \gibbsPr{\graph'}{v} \pairwiseInfluence{\graph_v'}[][v, w]} q(w) \\
		&\le (2 + C) \sum_{v \in \clique[i]} \gibbsPr{\graph'}{v} \sum_{\substack{j \in [\numberOfCliques]: \\j \neq i}} \sum_{w \in \clique[j]} \absolute{\pairwiseInfluence{\graph_v'}[][v, w]} q(w) ,
	\end{align*}
	where $\graph_{v}' = \subgraph{\graph'}{S \setminus (\clique[i] \setminus \{v\})} = \subgraph{\graph}{S \setminus (\clique[i] \setminus \{v\})}$.
	As the cliques are disjoint and $\graph_v'$ is a subgraph of $\graph$, we apply \Cref{condition:influence_bound} to obtain
	\begin{align*}
		(2 + C) \sum_{v \in \clique[i]} \gibbsPr{\graph'}{v} \sum_{\substack{j \in [\numberOfCliques]: \\j \neq i}} \sum_{w \in \clique[j]} \absolute{\pairwiseInfluence{\graph_v'}[][v, w]} q(w)
		&= (2 + C) \sum_{v \in \clique[i]} \gibbsPr{\graph'}{v} \sum_{w \in S \setminus (\clique[i] \setminus \{v\})} \absolute{\pairwiseInfluence{\graph_v'}[][v, w]} q(w) \\
		&\le (2 + C) \sum_{v \in \clique[i]} \gibbsPr{\graph'}{v} C q(v) \\	
		&= (2 + C) C  \sum_{v \in \clique[i]} \gibbsPr{\graph'}{v} q(v) \\
		&= (2 + C) C p(\complexEmpty{i}),
	\end{align*}
	which concludes the proof.
\end{proof}

\subsection{Canonical paths in skeleton walks} \label{subsec:canonical_paths}
The previous section shows that we can bound local expansion of the simplicial-complex representation of a disjoint clique cover based on pairwise influence between vertices.
The resulting bound on the second largest eigenvalue of the skeleton walks might in some cases be larger than~$1$, which makes \Cref{lemma:twoStepMixing} inapplicable.
Thus, we introduce a more crude bound on this second eigenvalue by applying the canonical-path method to the skeleton walk.
Although the resulting bound is worse in most cases, it is guaranteed to be less than $1$, which is sufficient to cover the cases where using pairwise influence fails.

We start by giving a short overview on the canonical-path method.
Let $\markov$ be a Markov chain that is reversible with respect to its stationary distribution $\stationary[\markov]$.
Let $\mcEdges{\markov} = \{(x, y) \in \stateSpace[\markov]^2 \mid x \neq y, \transitionMatrix[\markov][x, y] > 0\}$ denote the edges of the Markov chain excluding self loops, and let $\mcAllEdges{\markov} = \{(x, y) \in \stateSpace[\markov]^2 \mid \transitionMatrix[\markov][x, y] > 0\}$ be the set of edges including self loops.
For each $(x, y) \in \mcAllEdges{\markov}$, we set $\transitionWeight{\markov}[x, y] = \stationary[\markov][x] \transitionMatrix[\markov][x, y]$.
The idea of the canonical-path method is to construct a path $\canonicalPath = (x_0 = x, x_1, \dots, x_l = y)$ for every $x, y \in \stateSpace[\markov]$ with $x \neq y$ using the edges in $\mcEdges{\markov}$ (i.e., $(x_{i-1}, x_i) \in \mcEdges{\markov}$ for all $i \in [l]$).
We denote by $\mcEdges{\canonicalPath[x y]}$ the set of edges that are used by the path $\canonicalPath[x y]$ and by $\length{\canonicalPath[x y]}$ the length of a path.
Further, we call a set of paths $\canonicalPaths = \{\canonicalPath[x y] \mid x, y \in \stateSpace[\markov], x \neq y\}$ canonical if and only if it contains exactly one path for each $x, y \in \stateSpace[\markov]$ with $x \neq y$, and its \emph{congestion} is defined to be
\[
	\congestion[\canonicalPaths] = \max_{(w, z) \in \mcEdges{\markov}} \frac{1}{\transitionWeight{\markov}[w, z]} \sum_{\substack{x, y \in \stateSpace[\markov]: \\ (w, z) \in \mcEdges{\canonicalPath[x y]}}} \length{\canonicalPath[x y]} \stationary[\markov][x] \stationary[\markov][y] . 
\]
\begin{theorem}[{\cite[Theorem $5$]{sinclair1992improved}}]
	\label{thm:canonical_paths_method}
	For any reversible Markov chain $\markov$ and every set of canonical paths $\canonicalPaths$ for $\markov$ it holds that
	\[
		\eigenvalueOf{\transitionMatrix[\markov]}[2] \le 1 - \frac{1}{\congestion[\canonicalPaths]}. \qedhere
	\]
\end{theorem}

By applying the canonical-path method to the skeleton walk, we obtain the following lemma.
\begin{lemma}
	\label{lemma:canonical_paths}
	Let $(\graph, \weights)$ be an instance of the multivariate hard-core model with disjoint clique cover $\clique$ of size $\numberOfCliques$, and let $\maxPartition = \max_{i \in [\numberOfCliques]} \{ \partitionFunction[\subgraph{\graph}{\clique[i]}] \}$.
	Further, let $(\complex, \complexWeight)$ be the resulting simplicial-complex representation, and let $\skeletonWalk = \skeletonWalk[\complex, \complexWeight]$ be the skeleton walk on $(\complex, \complexWeight)$.
	Then $\eigenvalueOf{\transitionMatrix[\skeletonWalk]}[2] \le 1 - \frac{1}{12 \maxPartition^2}$.
\end{lemma}

\begin{proof}
	As discussed in the proof of \Cref{lemma:eigenvalues_skeleton}, $\transitionMatrix[\skeletonWalk]$ is reversible with respect to its stationary distribution $\stationary[\skeletonWalk]$, where $\stationary[\skeletonWalk][x] = \frac{1}{\numberOfCliques}\gibbsPr{\graph}{x}$.
	Thus, $\transitionWeight{\skeletonWalk}[x, y] = \frac{1}{\numberOfCliques (\numberOfCliques - 1)}\gibbsPr{\graph}{x, y}$.
	
	We start by constructing the paths $\canonicalPaths = \{\canonicalPath[x y] \mid x, y \in \complexGroundset, x \neq y\}$.
	To this end, let~$p$ be a fixed-point-free permutation of $[\numberOfCliques]$.
	Our construction goes as follows:
	\begin{itemize}
		\item $\canonicalPath[\complexEmpty{i} \complexEmpty{j}] = (\complexEmpty{i}, \complexEmpty{j})$ for $i \neq j$,
		\item $\canonicalPath[\complexElement{v} \complexEmpty{i}] = (\complexElement{v}, \complexEmpty{i})$, $\canonicalPath[\complexEmpty{i} \complexElement{v}] = (\complexEmpty{i}, \complexElement{v})$ for all $v \in \vertices$ with $v \notin \clique[i]$,
		\item $\canonicalPath[\complexElement{v} \complexElement{w}] = (\complexElement{v}, \complexEmpty{i}, \complexEmpty{j}, \complexElement{w})$ for $v \in \clique[j], w \in \clique[i]$ with $i \neq j$,
		\item $\canonicalPath[x y] = (x, \complexEmpty{p(i)}, y)$ for $x, y \in \complexPartition{i}$.
	\end{itemize}
	Let $\mcEdges{\canonicalPaths} = \bigcup_{\canonicalPath \in \canonicalPaths} \mcEdges{\canonicalPath}$ and note that for all $x \neq y$ we have $\length{\canonicalPath[x y]} \le 3$.
	It suffices to upper-bound
	\[
		\congestion[\canonicalPaths] \le 3 \max_{e \in \mcEdges{\canonicalPaths}} \frac{1}{\transitionWeight{\skeletonWalk}[e]} \sum_{\substack{x, y \in \complexGroundset: \\ e \in \mcEdges{\canonicalPath[x y]}}} \stationary[\skeletonWalk][x] \stationary[\skeletonWalk][y] .
	\]
	
	We derive such an upper bound by partitioning $\mcEdges{\canonicalPaths}$ into the $3$ following types of edges:
	\begin{itemize}
		\item $A = \{(\complexElement{v}, \complexEmpty{i}) \in \mcEdges{\canonicalPaths} \mid v \notin \clique[i]\}$,
		\item $B = \{(\complexEmpty{i}, \complexElement{v}) \in \mcEdges{\canonicalPaths} \mid v \notin \clique[i]\}$,
		\item $C = \{(\complexEmpty{i}, \complexEmpty{j}) \in \mcEdges{\canonicalPaths} \mid i \neq j\}$.
	\end{itemize}
	We make a case distinction with respect to these types.
	
	\textbf{Case~$A$.}
    Let $(\complexElement{v}, \complexEmpty{i}) \in A$ and without loss of generality assume $v \in \clique[j]$ for $j \neq i$.
	If $p(j) \neq i$, then $(\complexElement{v}, \complexEmpty{i})$ is only used by paths that start at $\complexElement{v}$ and go to any element in $\complexPartition{i}$ (including $\complexEmpty{i}$).
	Further, if $p(j) = i$, then it is also used by paths from $\complexElement{v}$ to any $y \in \complexPartition{j}$ with $y \neq \complexElement{v}$.
	Thus, we obtain 
	\begin{align*}
		\frac{1}{\transitionWeight{\skeletonWalk}[\complexElement{v}, \complexEmpty{i}]} \sum_{\substack{x, y \in \complexGroundset: \\ (\complexElement{v}, \complexEmpty{i}) \in \mcEdges{\canonicalPath[x y]}}} \stationary[\skeletonWalk][x] \stationary[\skeletonWalk][y]
		& \le \frac{\numberOfCliques (\numberOfCliques - 1)}{\numberOfCliques^2} \frac{1}{\gibbsPr{\graph}{\complexEmpty{i}}[v]} \left(\sum_{y \in \complexPartition{i}} \gibbsPr{\graph}{y} + \sum_{\substack{y \in \complexPartition{j}: \\y \neq \complexElement{v}}}  \gibbsPr{\graph}{y} \right) \\
		& \le \frac{2}{\gibbsPr{\graph}{\complexEmpty{i}}[v]} ,
	\end{align*}
	where the second inequality comes from the fact that we have $\sum_{y \in \complexPartition{k}} \gibbsPr{\graph}{y} = \gibbsPr{\graph}{\bigcup_{y \in \complexPartition{k}} y} = 1$ for all $k \in [\numberOfCliques]$.
	By the submultiplicativity of~\partitionFunction,
	\[
		\gibbsPr{\graph}{\complexEmpty{i}}[v] 
		= \frac{\partitionFunction[\subgraph{\graph}{\vertices \setminus (\neighborsClosed{v} \cup \clique[i])}]}{\partitionFunction[\subgraph{\graph}{\vertices \setminus \neighborsClosed{v}}]}
		\ge \frac{\partitionFunction[\subgraph{\graph}{\vertices \setminus (\neighborsClosed{v} \cup \clique[i])}]}{\partitionFunction[\subgraph{\graph}{\vertices \setminus (\neighborsClosed{v} \cup \clique[i])}] \partitionFunction[\subgraph{\graph}{\clique[i] \setminus \neighborsClosed{v}}]} 
		\ge \frac{1}{\partitionFunction[\subgraph{\graph}{\clique[i]}]} .
	\]
	We obtain the bound
	\begin{gather} \label{eq:canonical_paths:A}
		\frac{1}{\transitionWeight{\skeletonWalk}[\complexElement{v}, \complexEmpty{i}]} \sum_{\substack{x, y \in \complexGroundset: \\ (\complexElement{v}, \complexEmpty{i}) \in \mcEdges{\canonicalPath[x y]}}} \stationary[\skeletonWalk][x] \stationary[\skeletonWalk][y]
		\le 2 \partitionFunction[\subgraph{\graph}{\clique[i]}] .
	\end{gather}
	
	\textbf{Case~$B$.}
    For $(\complexEmpty{i}, \complexElement{v}) \in B$, by symmetry, this case is analogous to case~$A$.
    Thus, we get
	\begin{gather} \label{eq:canonical_paths:B}
		\frac{1}{\transitionWeight{\skeletonWalk}[\complexEmpty{i}, \complexElement{v}]} \sum_{\substack{x, y \in \complexGroundset: \\ (\complexEmpty{i}, \complexElement{v}) \in \mcEdges{\canonicalPath[x y]}}} \stationary[\skeletonWalk][x] \stationary[\skeletonWalk][y]
		\le 2 \partitionFunction[\subgraph{\graph}{\clique[i]}] .
	\end{gather}
	
	\textbf{Case~$C$.}
    Finally, consider $(\complexEmpty{i}, \complexEmpty{j}) \in C$. 
	If $p(j) \neq i$ and $p(i) \neq j$, then this edge is only used by paths from $\complexElement{v}$ to $\complexElement{w}$ for any pair $v \in \clique[j], w \in \clique[i]$ and for the direct transition from $\complexEmpty{i}$ to $\complexEmpty{j}$.
	If $p(j) = i$, then it is also used by paths from any $\complexElement{v}$ for $v \in \clique[j]$ to $\complexEmpty{j}$.
    Symmetrically, if $p(i) = j$, then it is also used by paths from $\complexEmpty{i}$ to $\complexElement{v}$ for $v \in \clique[i]$.
	Thus, we obtain
	\begin{align*}
		\frac{1}{\transitionWeight{\skeletonWalk}[\complexEmpty{i}, \complexEmpty{j}]} & \sum_{\substack{x, y \in \complexGroundset: \\ (\complexEmpty{i}, \complexEmpty{j}) \in \mcEdges{\canonicalPath[x y]}}} \stationary[\skeletonWalk][x] \stationary[\skeletonWalk][y] \\
		& \le \frac{\numberOfCliques (\numberOfCliques - 1)}{\numberOfCliques^2} \frac{1}{\gibbsPr{\graph}{\complexEmpty{i}, \complexEmpty{j}}} \Bigg( \sum_{\substack{v \in \clique[j]: \\ w \in \clique[i]}} \gibbsPr{\graph}{v} \gibbsPr{\graph}{w} + \gibbsPr{\graph}{\complexEmpty{i}} \gibbsPr{\graph}{\complexEmpty{j}}\\
        &\hspace*{4 cm} + \gibbsPr{\graph}{\complexEmpty{j}} \sum_{v \in \clique[j]} \gibbsPr{\graph}{v} + \gibbsPr{\graph}{\complexEmpty{i}} \sum_{v \in \clique[i]} \gibbsPr{\graph}{v} \Bigg) \\
		& \le \frac{4}{\gibbsPr{\graph}{\complexEmpty{i}, \complexEmpty{j}}} .
	\end{align*}
	Now, observe that
	\begin{align*}
		\gibbsPr{\graph}{\complexEmpty{i}, \complexEmpty{j}} 
		& = \frac{\partitionFunction[\subgraph{\graph}{\vertices \setminus (\clique[i] \cup \clique[j])}]}{\partitionFunction[\graph]}
		\ge \frac{\partitionFunction[\subgraph{\graph}{\vertices \setminus (\clique[i] \cup \clique[j])}]}{\partitionFunction[\subgraph{\graph}{\vertices \setminus (\clique[i] \cup \clique[j])}] \partitionFunction[\subgraph{\graph}{\clique[i]}] \partitionFunction[\subgraph{\graph}{\clique[j]}]}\\
        & = \frac{1}{\partitionFunction[\subgraph{\graph}{\clique[i]}] \partitionFunction[\subgraph{\graph}{\clique[j]}]} .
	\end{align*}
	Thus,
	\begin{gather} \label{eq:canonical_paths:C}
		\frac{1}{\transitionWeight{\skeletonWalk}[\complexEmpty{i}, \complexEmpty{j}]} \sum_{\substack{x, y \in \complexGroundset \\\text{s.t. } (\complexEmpty{i}, \complexEmpty{j}) \in \mcEdges{\canonicalPath[x y]}}} \stationary[\skeletonWalk][x] \stationary[\skeletonWalk][y]
		\le 3 \partitionFunction[\subgraph{\graph}{\clique[i]}] \partitionFunction[\subgraph{\graph}{\clique[j]}] .
	\end{gather}
	
	Combining \cref{eq:canonical_paths:A,eq:canonical_paths:B,eq:canonical_paths:C} we get $\congestion[\canonicalPaths] \le 12 \maxPartition^2$.
	By \Cref{thm:canonical_paths_method} this implies $\eigenvalueOf{\transitionMatrix[\skeletonWalk]}[2] \le 1 - \frac{1}{12 \maxPartition^2}$, which concludes the proof.
\end{proof}

Note that \Cref{lemma:clique_influence_bound,lemma:eigenvalues_skeleton,lemma:canonical_paths} only consider the skeleton walk on the complex $(\complex, \complexWeight)$. 
However, in order to bound the local expansion, we need to investigate the skeleton walk on all links $(\complex[][\complexFace], \complexWeight[][\complexFace])$ for every face $\complexFace \in \complex[k]$ with $0 \le k \le \numberOfCliques - 2$.
To achieve this, we map the link for any such face to the simplicial complex representation of a smaller instance, such that we can apply \Cref{lemma:twoStepMixing}.
To this end, we introduce the following lemma.
\begin{lemma}
	\label{lemma:complex_scale_invariance}
	Let $\complex$ be a pure $d$-dimensional simplicial complex for $d \ge 2$, let $\complexWeight$ and $\complexWeight'$ be two weight functions for $\complex$, and let $\skeletonWalk = \skeletonWalk[\complex, \complexWeight]$ and $\skeletonWalk' = \skeletonWalk[\complex, \complexWeight']$.
	 Further, if there is an $r \in \R_{>0}$ such that for all maximum faces $\complexFace \in \complex[d]$ we have $\complexWeight'(\complexFace) = r \complexWeight[\complexFace]$, then $\skeletonWalk = \skeletonWalk'$ and, in particular, $\eigenvalueOf{\transitionMatrix[\skeletonWalk]}[2] = \eigenvalueOf{\transitionMatrix[\skeletonWalk']}[2]$. 
\end{lemma}

\begin{proof}
We prove this statement by showing equality of the state spaces and transition probabilites.
	The fact that $\stateSpace[\skeletonWalk] = \stateSpace[\skeletonWalk']$ follows directly from the fact that both walks are on the 1-skeleton of the same complex.
	Now, let $\complexFace' \in \complex$ be any face of $\complex$. 
	Note that
	\[
	\complexWeight'(\complexFace') 
	= \sum_{\substack{\complexFace \in \complex[d]: \\ \complexFace' \subseteq \complexFace}} \complexWeight'(\complexFace)
	= r \sum_{\substack{\complexFace \in \complex[d]: \\ \complexFace' \subseteq \complexFace}} \complexWeight[\complexFace]
	= r \complexWeight[\complexFace'] .
	\]
	Thus, the weights of all faces differ by the same factor $r$.
	Let $\{x\}, \{y\} \in \complex$ with $x \neq y$.
	If $\{x, y\} \notin \complex$, then $\transitionMatrix[\skeletonWalk][x, y] = 0 = \transitionMatrix[\skeletonWalk'][x, y]$.
	Otherwise, if $\{x, y\} \in \complex$, then
	\[
	\transitionMatrix[\skeletonWalk'][x, y] 
	= \frac{\complexWeight'(\{x, y\})}{\sum\limits_{\substack{\{z\} \in \complex: \\ \{x, z\} \in \complex}} \complexWeight'(\{x, z\})}
	= \frac{r \complexWeight[\{x, y\}]}{r \sum\limits_{\substack{\{z\} \in \complex: \\ \{x, z\} \in \complex}} \complexWeight[\{x, z\}]}
	= \transitionMatrix[\skeletonWalk][x, y] .
	\]
	As both $\skeletonWalk$ and $\skeletonWalk'$ have self loop probabilities of $0$, it follows $\transitionMatrix[\skeletonWalk] = \transitionMatrix[\skeletonWalk']$.
	This implies $\eigenvalueOf{\transitionMatrix[\skeletonWalk]}[2] = \eigenvalueOf{\transitionMatrix[\skeletonWalk']}[2]$.
\end{proof}

\subsection{Block dynamics for non-disjoint clique covers}
So far, we investigated the two-step walk on the simplicial complex representation resulting from a disjoint clique cover.
In this section, we discuss how this can be used to obtain information about the eigenvalues of the block dynamics based on arbitrary clique covers.
We use a Markov chain comparison theorem based on \cite{diaconis1993comparison}.
Similarly to the canonical-paths method, it is based on constructing paths between certain pairs of vertices in the state space.
Recall the notation introduced at the beginning of \Cref{subsec:canonical_paths}. 
\begin{theorem}[{\cite[Theorem $2.1$]{diaconis1993comparison}}]
	\label{thm:comparison}
	Let $\markov, \markov'$ be two reversible Markov chains on a common state space $\stateSpace[\markov] = \stateSpace = \stateSpace[\markov']$.
	For each $(x, y) \in \mcEdges{\markov}$, fix some path $\canonicalPath[x y] = (x_0 = x, x_1, ..., x_l=y)$ in $\mcAllEdges{\markov'}$ (i.e., $(x_{i-1}, x_i) \in \mcAllEdges{\markov'}$ for all $i \in [l]$) and let $\canonicalPaths = \{\canonicalPath[x y] \mid (x, y) \in \mcEdges{\markov}\}$ be the resulting set of paths.
	Further, let $\statRatio \in \R_{>0}$ be such that $\stationary[\markov][x] \ge \statRatio \stationary[\markov'][x]$ for all $x \in \stateSpace$ and set
	\[
		\flowRatio{\canonicalPaths} 
		= \max_{(w, z) \in \mcAllEdges{\markov'}} \frac{1}{\transitionWeight{\markov'}[w, z]} \sum_{\substack{x, y \in \stateSpace: \\ (z, w) \in \mcEdges{\canonicalPath[x y]}}} \length{\canonicalPath[x y]} \transitionWeight{\markov}[x, y] . 
	\]
	For all $2 \le i \le \size{\stateSpace}-1$ it now holds that
	\[
		\eigenvalueOf{\transitionMatrix[\markov']}[i] \le 1 - \frac{\statRatio}{\flowRatio{\canonicalPaths}} \left(1 - \eigenvalueOf{\transitionMatrix[\markov]}[i]\right) .
		\qedhere
	\]
\end{theorem}

Although the two-step walk and block dynamics do not act directly on the same state space, we know by \Cref{obs:dimension} that there is a one-to-one correspondence between their state spaces.
Thus, we can assume them to have the same state space up to relabeling of states.
For every $\independentSet \in \independentSets$, we write $\faceFromIndep{\independentSet} \in \complex[\numberOfCliques]$ for the corresponding maximum face in the simplicial complex and, for every maximum face $\complexFace \in \complex[\numberOfCliques]$, we let $\indepFromFace{\complexFace} \in \independentSets$ denote the respective independent set.

\begin{lemma}
	\label{lemma:non_disjoint_comparison}
	Let $(\graph, \weights)$ be an instance of the multivariate hard-core model with a clique cover~$\clique$ of size $\numberOfCliques$, and let $\maxPartition = \max_{i \in [\numberOfCliques]} \{\partitionFunction[\subgraph{\graph}{\clique[i]}]\}$.
	Further, let $\blockDynamics = \blockDynamics[\graph, \weights, \clique]$.
	Then there is a disjoint clique cover $K$ of size $\numberOfCliques$ such that
	\[
		\max_{i \in [\numberOfCliques]} \{\partitionFunction[\subgraph{\graph}{K_i}]\} \le \maxPartition,
	\] 
	and, for $(\complex, \complexWeight)$ being the simplicial complex representation resulting from $K$ and $\twoStep = \twoStep[\complex, \complexWeight]$, it holds that
	\begin{equation}
        \label{eq:eigenvalue_comparison_block_dynamics_simplicial_complex}
		\eigenvalueOf{\transitionMatrix[\blockDynamics]}[2] \le 1 - \frac{1}{\maxPartition} \left(1 - \eigenvalueOf{\transitionMatrix[\twoStep]}[2]\right) .
		\qedhere
	\end{equation}
\end{lemma}

\begin{proof}
	We start by constructing the a disjoint clique cover $K$ from $\clique$.
	To this end, we set $K_1 = \clique[1]$ and $K_i = \clique[i] \setminus K_{i - 1}$ for $2 \le i \le \numberOfCliques$.
	Note that $K_i \cap K_j = \emptyset$ for all $i \neq j$ and $\bigcup_{i \in [\numberOfCliques]} K_i = \bigcup_{i \in [\numberOfCliques]} \clique[i] = \vertices$.
	Thus, $K = \{K_i\}_{i \in [\numberOfCliques]}$ is a disjoint clique cover.
	Further, for all $i \in [\numberOfCliques]$ it holds that $K_i \subseteq \clique[i]$, which implies $\partitionFunction[\subgraph{\graph}{K_i}] \le \partitionFunction[\subgraph{\graph}{\clique[i]}]$ and thus $\max_{i \in [\numberOfCliques]} \{\partitionFunction[\subgraph{\graph}{K_i}]\} \le \maxPartition$.
	
	We now prove \cref{eq:eigenvalue_comparison_block_dynamics_simplicial_complex} by applying \Cref{thm:comparison}.
	First note that we know by \Cref{observation:twoStep_transitions} that $\stationary[\twoStep][\complexFace] = \gibbsDistributionFunction{\indepFromFace{\complexFace}} = \stationary[\blockDynamics][\indepFromFace{\complexFace}]$, meaning that we can set $\statRatio = 1$.
	It remains to construct a set of paths $\canonicalPaths$ and upper bound $\flowRatio{\canonicalPaths}$.
	Note that for every $(\complexFace, \complexFace') \in \mcEdges{\twoStep}$ there exists some $i \in [\numberOfCliques]$ such that $\indepFromFace{\complexFace} \symDiff \indepFromFace{\complexFace'} \subseteq K_i \subseteq \clique[i]$, implying that there is a corresponding edge $(\indepFromFace{\complexFace}, \indepFromFace{\complexFace'}) \in \mcAllEdges{\blockDynamics}$.
	We set the path $\canonicalPath[\complexFace \complexFace']$ between $\indepFromFace{\complexFace}$ and $\indepFromFace{\complexFace'}$ to be this edge and $\canonicalPaths$ to be the set of all such paths.
	
	Note that we actually only use edges in $\mcEdges{\blockDynamics}$, which means that we ignore self loops while bounding $\flowRatio{\canonicalPaths}$.
	Further, for all $(\complexFace, \complexFace') \in \mcEdges{\twoStep}$ it holds that $\length{\canonicalPath[\complexFace \complexFace']} = 1$.
	We obtain for every $(\independentSet, \independentSet') \in \mcEdges{\blockDynamics}$ that
	\[
		\frac{1}{\transitionWeight{\blockDynamics}[\independentSet, \independentSet']} \sum_{\substack{\complexFace, \complexFace' \in \complex[\numberOfCliques]\colon \\ (\independentSet, \independentSet') \in \mcEdges{\canonicalPath[\complexFace \complexFace']}}} \length{\canonicalPath[\complexFace \complexFace']} \transitionWeight{\twoStep}[\complexFace, \complexFace']
		= \begin{cases}
			\frac{\transitionWeight{\twoStep}[\faceFromIndep{\independentSet}, \faceFromIndep{\independentSet'}]}{\transitionWeight{\blockDynamics}[\independentSet, \independentSet']} & \text{ if } (\faceFromIndep{\independentSet}, \faceFromIndep{\independentSet'}) \in \mcEdges{\twoStep}, \\
			0 & \text{ otherwise. }
		\end{cases}
	\]
	
	We proceed by doing a case distinction based on $(\independentSet, \independentSet') \in \mcEdges{\blockDynamics}$, assuming that $(\faceFromIndep{\independentSet}, \faceFromIndep{\independentSet'}) \in \mcEdges{\twoStep}$.
	Consider a transition $(\independentSet, \independentSet') \in \mcEdges{\blockDynamics}$ with $\independentSet = \independentSet' \cup \{v\}$ for some $v \notin \independentSet$ (i.e., removing $v$).
	Without loss of generality, assume $v \in K_i$ for some $i \in [\numberOfCliques]$.
	First, note that
	\[
		\transitionWeight{\blockDynamics}[\independentSet, \independentSet'] 
		= \gibbsDistributionFunction{\independentSet} \frac{1}{\numberOfCliques} \sum_{\substack{j \in [\numberOfCliques]: \\ v \in \clique[j]}} \gibbsDistributionFunction{\emptyset}[][\clique[j]][\spinConfig[\vertices \setminus \clique[j]][][\independentSet]]
		\ge \gibbsDistributionFunction{\independentSet} \frac{1}{\numberOfCliques} \frac{1}{\maxPartition} .
	\] 
	Further, by \Cref{observation:twoStep_transitions} it holds that 
	\[
		\transitionWeight{\twoStep}[\complexFace, \complexFace'] 
		= \gibbsDistributionFunction{\independentSet} \frac{1}{\numberOfCliques} \gibbsDistributionFunction{\emptyset}[][K_i][\spinConfig[\vertices \setminus K_i][][\independentSet]]
		\le \gibbsDistributionFunction{\independentSet} \frac{1}{\numberOfCliques} .
	\]
	Thus, we obtain
	\[
		\frac{\transitionWeight{\twoStep}[\faceFromIndep{\independentSet}, \faceFromIndep{\independentSet'}]}{\transitionWeight{\blockDynamics}[\independentSet, \independentSet']} \le \maxPartition .
	\]
	
	Now, consider $(\independentSet, \independentSet') \in \mcEdges{\blockDynamics}$ with $\independentSet' = \independentSet \cup \{v\}$ for some $v \notin \independentSet$ (i.e., adding $v$) and assume $v \in K_i$ for some $i \in [\numberOfCliques]$.
	We observe that
	\[
		\transitionWeight{\blockDynamics}[\independentSet, \independentSet'] 
		= \gibbsDistributionFunction{\independentSet} \frac{1}{\numberOfCliques} \sum_{\substack{j \in [\numberOfCliques]: \\ v \in \clique[j]}} \gibbsDistributionFunction{\{v\}}[][\clique[j]][\spinConfig[\vertices \setminus \clique[j]][][\independentSet]]
		\ge \gibbsDistributionFunction{\independentSet} \frac{1}{\numberOfCliques} \frac{1}{\maxPartition} \weight[v].
	\]
	Similarly, we have 
	\[
		\transitionWeight{\twoStep}[\complexFace, \complexFace'] 
		= \gibbsDistributionFunction{\independentSet} \frac{1}{\numberOfCliques} \gibbsDistributionFunction{\{v\}}[][K_i][\spinConfig[\vertices \setminus K_i][][\independentSet]]
		\le \gibbsDistributionFunction{\independentSet} \frac{1}{\numberOfCliques} \weight[v] .
	\]
	Again, we get 
	\[
		\frac{\transitionWeight{\twoStep}[\faceFromIndep{\independentSet}, \faceFromIndep{\independentSet'}]}{\transitionWeight{\blockDynamics}[\independentSet, \independentSet']} \le \maxPartition .
	\]
	
	Finally, consider $(\independentSet, \independentSet') \in \mcEdges{\blockDynamics}$ with $\independentSet' = (\independentSet \setminus \{w\}) \cup \{v\}$ for $v \neq w$ (i.e., remove $w$ and add $v$).
	Note that, as $(\independentSet, \independentSet') \in \mcEdges{\blockDynamics}$, there has to be some $j \in [\numberOfCliques]$ such that $v, w \in \clique[j]$.
	We get that  
	\[
		\transitionWeight{\blockDynamics}[\independentSet, \independentSet'] 
		= \gibbsDistributionFunction{\independentSet} \frac{1}{\numberOfCliques} \sum_{\substack{j \in [\numberOfCliques]: \\ v, w \in \clique[j]}} \gibbsDistributionFunction{\{v\}}[][\clique[j]][\spinConfig[\vertices \setminus \clique[j]][][\independentSet]]
		\ge \gibbsDistributionFunction{\independentSet} \frac{1}{\numberOfCliques} \frac{1}{\maxPartition} \weight[v].
	\]
	Further, as $(\faceFromIndep{\independentSet}, \faceFromIndep{\independentSet'}) \in \mcEdges{\twoStep}$, there is exactly one $i \in [\numberOfCliques]$ such that $v, w \in K_i$.
	This yields
	\[
		\transitionWeight{\twoStep}[\complexFace, \complexFace'] 
		= \gibbsDistributionFunction{\independentSet} \frac{1}{\numberOfCliques} \gibbsDistributionFunction{\{v\}}[][K_i][\spinConfig[\vertices \setminus K_i][][\independentSet]]
		\le \gibbsDistributionFunction{\independentSet} \frac{1}{\numberOfCliques} \weight[v].
	\]
	Thus, also in this case we have
	\[
		\frac{\transitionWeight{\twoStep}[\faceFromIndep{\independentSet}, \faceFromIndep{\independentSet'}]}{\transitionWeight{\blockDynamics}[\independentSet, \independentSet']} \le \maxPartition .
	\]
	
	As $\blockDynamics$ changes at most one clique in $\clique$ at each step, the above case distinction covers all $(\independentSet, \independentSet') \in \mcEdges{\blockDynamics}$.
	We obtain $\flowRatio{\canonicalPaths} \le \maxPartition$, which proves the claim.
\end{proof}

\subsection{Bounding the mixing time} 
We now have everything to state and prove our main theorem on the mixing time of block dynamics.

\mixingBlockDynamics*

\begin{proof}
	By \cref{eq:mixing_spectral_gap} it is sufficient to lower bound the spectral gap of $\transitionMatrix[\blockDynamics]$ by $\frac{1}{\poly{\maxPartition} }$ and $\frac{1}{\poly{\numberOfCliques}}$ to prove our claim.
	Further, transforming the chain into a lazy version only results in constant overhead in the mixing time.
	Thus, we focus on lower-bounding $1 - \eigenvalueOf{\transitionMatrix[\blockDynamics]}[2]$, which is equivalent to upper-bounding $\eigenvalueOf{\transitionMatrix[\blockDynamics]}[2]$.
	
	Let $K$ be the disjoint clique cover, constructed from $\clique$ as described in the proof of \Cref{lemma:non_disjoint_comparison}, and let $\maxPartition^{(K)} = \max_{i \in [\numberOfCliques]} \{\partitionFunction[\subgraph{\graph}{K_i}]\}$. 
	Further, let $(\complex, \complexWeight)$ be the simplicial-complex representation based on~$K$ with groundset~$\complexGroundset$ and partitions $\{\complexPartition{i}\}_{i \in [\numberOfCliques]}$, and let $\twoStep = \twoStep[\complex, \complexWeight]$ denote the two-step walk on $(\complex, \complexWeight)$.
	By \Cref{lemma:non_disjoint_comparison},
	\begin{gather} \label{eq:main:comparison}
		\eigenvalueOf{\transitionMatrix[\blockDynamics]}[2] \le 1 - \frac{1}{\maxPartition} \left(1 - \eigenvalueOf{\transitionMatrix[\twoStep]}[2]\right) .
	\end{gather}
	Thus, it is sufficient for us to upper bound $\eigenvalueOf{\transitionMatrix[\twoStep]}[2]$.
	
	We aim to apply \Cref{lemma:twoStepMixing}, which involves upper-bounding local expansion of the simplicial-complex representation.
	Let~$C$ be the constant for which $(\graph, \weights)$ satisfies \Cref{condition:influence_bound}.
	We proceed by proving that $(\complex, \complexWeight)$ is a local $(\alpha_0, \dots, \alpha_{\numberOfCliques-2})$-expander, where
	\begin{gather} \label{eq:main:local_expansion}
		\alpha_k \le \min \left\{ 1 - \frac{1}{12 \maxPartition^2}, \frac{(2 + C)C}{\numberOfCliques - k - 1} \right\} \text{ for } 0 \le k \le \numberOfCliques - 2 .
	\end{gather}
	We start by arguing both bounds for the case $k = 0$.
    Then we generalize our arguments for the cases $k \in [\numberOfCliques - 2]$.
	
	\textbf{Case $k = 0$.}
    Let $\skeletonWalk = \skeletonWalk[\complex, \complexWeight]$ be the skeleton walk on $(\complex, \complexWeight)$.
	By definition, we have $\alpha_0 = \eigenvalueOf{\transitionMatrix[\skeletonWalk]}[2]$.
	Note that the first bound $\alpha_0 \le 1 - \frac{1}{12 \maxPartition^2}$ follows directly from \Cref{lemma:canonical_paths} and the fact that $\maxPartition^{(K)} \le \maxPartition$ by \Cref{lemma:non_disjoint_comparison}.
	To prove the second bound, we apply \Cref{lemma:eigenvalues_skeleton}, which gives us
	\[
		\alpha_0 \leq \frac{\eigenvalueOf{\cliqueInfluence{\graph}{K}}[1]}{\numberOfCliques - 1} .
	\]
	By \Cref{lemma:clique_influence_bound}, we conclude that $\alpha_0 \le \frac{(2 + C)C}{\numberOfCliques - 1}$.
	
	\textbf{Case $k \in [\numberOfCliques - 2]$.}
	By definition, we have to show for all $\complexFace \in \complex[k]$ that the skeleton walk $\skeletonWalk_{\complexFace} = \skeletonWalk[\complex[][\complexFace], \complexWeight[][\complexFace]]$ on the link $(\complex[][\complexFace], \complexWeight[][\complexFace])$ satisfies
	\[
		\eigenvalueOf{\transitionMatrix[\skeletonWalk_{\complexFace}]}[2] \le \min \left\{ 1 - \frac{1}{12 \maxPartition^2}, \frac{(2 + C)C}{\numberOfCliques - k - 1} \right\} .
	\]
	We construct a subset of vertices $S \subseteq \vertices$ such that
	\[
		S = \bigcup_{x \in \complexFace} \begin{cases}
			\neighborsClosed{v} & \text{ if $x = \complexElement{v}$ for some $v \in \vertices$,} \\
			\clique[i] & \text{ if $x = \complexEmpty{i}$ for some $i \in [\numberOfCliques]$.}
		\end{cases}
	\]
	Let $\graph' = \subgraph{\graph}{\vertices \setminus S}$ be the subgraph induced by $\vertices \setminus S$ and let $\weights' = \subgraph{\weights}{\vertices \setminus S}$ be the corresponding vertex weights.
	Further, let $K' = \set{K_i \setminus S}{ i \in [\numberOfCliques] \land \complexFace \cap \complexPartition{i} = \emptyset}$ and note that $K'$ is a disjoint clique cover of the multivariate hard-core instance $(\graph', \weights')$.
	Let $(\complex', \complexWeight')$ be the corresponding simplicial-complex representation and let $\skeletonWalk' = \skeletonWalk[\complex', \complexWeight']$ be the skeleton walk on $(\complex', \complexWeight')$.
	Note that $\complex' = \complex[][\complexFace]$ and that for every maximum face $\complexFace' \in \complex[\numberOfCliques - \size{\complexFace}][\complexFace]$ it holds that
	\begin{align*}
		\complexWeight[\complexFace'][\complexFace] 
		& = \complexWeight[\complexFace' \cup \complexFace] 
		\\
		& = \frac{1}{\partitionFunction[\graph, \weights]} \left(\prod_{w \in \vertices \setminus S} \weight[w] \ind{\complexElement{w} \in \complexFace'}\right) \left(\prod_{v \in S} \weight[v] \ind{\complexElement{v} \in \complexFace}\right) \\
		& = \frac{1}{\partitionFunction[\graph', \weights']} \left(\prod_{w \in \vertices \setminus S} \weight'_{w} \ind{\complexElement{w} \in \complexFace'}\right) \frac{\partitionFunction[\graph', \weights']}{\partitionFunction[\graph, \weights]} \left(\prod_{v \in S} \weight[v] \ind{\complexElement{v} \in \complexFace}\right) \\
		& = \complexWeight'(\complexFace') \partitionFunction[\graph', \weights'] \complexWeight[\complexFace],
	\end{align*}
	where $\partitionFunction[\graph', \weights'] \complexWeight[\complexFace] > 0$.
	Thus, by \Cref{lemma:complex_scale_invariance}, we obtain $\eigenvalueOf{\transitionMatrix[\skeletonWalk_{\complexFace}]}[2] = \eigenvalueOf{\transitionMatrix[\skeletonWalk']}[2]$.
	
    To upper-bound $\eigenvalueOf{\transitionMatrix[\skeletonWalk']}[2]$, observe that
	\begin{itemize}
		\item $\size{K'} = \numberOfCliques - \size{\complexFace} = \numberOfCliques - k$,
		\item $\max_{K_i' \in K'} \{\partitionFunction[\subgraph{\graph}{K_i'}]\} \le \maxPartition^{(K)} \le \maxPartition$, and
		\item $\graph'$ is and induced subgraph of $\graph$. 
	\end{itemize}
	Thus, analogous to the case $k=0$, applying \Cref{lemma:canonical_paths} yields
	\[
		\alpha_k \le 1 - \frac{1}{12 \max_{K_i' \in K'} \{\partitionFunction[\subgraph{\graph}{K_i'}]\}^2 } \le  1 - \frac{1}{12 \maxPartition^2 }.
	\]
	Together, \Cref{lemma:eigenvalues_skeleton} and \Cref{lemma:clique_influence_bound} result in
	\[
		\alpha_k \le \frac{(2+C)C}{\size{K'} - 1} = \frac{(2+C)C}{\numberOfCliques - k - 1} .
	\]
	
	From \cref{eq:main:local_expansion} and \Cref{lemma:twoStepMixing}, we conclude
	\begin{align*}
		\eigenvalueOf{\transitionMatrix[\twoStep]}[2] 
		& \le 1 - \frac{1}{\numberOfCliques} \prod_{0 \le k \le \numberOfCliques - 2} \left( 1 - \min \left\{ 1 - \frac{1}{12 \maxPartition^2}, \frac{(2+C)C}{\numberOfCliques - k - 1} \right\} \right) \\
		& = 1 - \frac{1}{\numberOfCliques} \prod_{0 \le k \le \numberOfCliques - 2} \max \left\{ \frac{1}{12 \maxPartition^2}, 1 - \frac{(2+C)C}{\numberOfCliques - k - 1} \right\} .
	\end{align*}
	Let $k_0 = m - 2 (2+C)C - 1$ and observe that for $k \le k_0$ it holds that 
	\[
		1 - \frac{(2+C)C}{\numberOfCliques - k - 1} \ge \frac{1}{2} > \frac{1}{12} \ge \frac{1}{12 \maxPartition^2} .
	\]
	Thus, we have
	\begin{align*}
		\frac{1}{\numberOfCliques} \prod_{0 \le k \le \numberOfCliques - 2} \max \left\{ \frac{1}{12 \maxPartition^2}, 1 - \frac{(2+C)C}{\numberOfCliques - k - 1} \right\}
		& \ge \frac{1}{\numberOfCliques} \left( \prod_{0 \le k \le k_0} \left( 1 - \frac{(2+C)C}{\numberOfCliques - k - 1} \right) \right) \left( \prod_{k_0 < k \le \numberOfCliques - 2} \frac{1}{12 \maxPartition^2} \right) \\
		& = \frac{1}{\numberOfCliques} \left( \frac{1}{12 \maxPartition^2} \right)^{m - 2 - k_0} \prod_{0 \le k \le k_0} \left( 1 - \frac{(2+C)C}{\numberOfCliques - k - 1} \right) \\
		& = \frac{1}{\numberOfCliques} \left( \frac{1}{\sqrt{12} \maxPartition} \right)^{4 (2+C)C - 2} \prod_{0 \le k \le k_0} \left( 1 - \frac{(2+C)C}{\numberOfCliques - k - 1} \right) .
	\end{align*}
	Further, because $\ln(1 - x) \ge \frac{- x}{1 - x}$ for $x < 1$, we have
	\begin{align*}
		\ln \left(\prod_{0 \le k \le k_0} \left( 1 - \frac{(2+C)C}{\numberOfCliques - k - 1} \right) \right) 
		& = \sum_{0 \le k \le k_0} \ln \left( 1 - \frac{(2+C)C}{\numberOfCliques - k - 1} \right) \\
		& \ge \sum_{0 \le k \le k_0} \frac{- (2+C)C / (\numberOfCliques - k - 1)}{1 - (2+C)C / (m - k - 1)} \\
		& = - (2+C)C \sum_{0 \le k \le k_0} \frac{1}{\numberOfCliques - (2+C)C - k - 1} \\
		& =  - (2+C)C \sum_{(2+C)C \le j \le \numberOfCliques - (2+C)C - 1} \frac{1}{j} \\
		& \ge - (2+C)C \ln (\numberOfCliques) .
	\end{align*}
	We obtain
	\begin{gather} \label{eq:main:twoStep}
		\eigenvalueOf{\transitionMatrix[\twoStep]}[2]
		\le 1 - \frac{1}{\numberOfCliques} \left( \frac{1}{\sqrt{12} \maxPartition} \right)^{4(2+C)C - 2} \eulerE^{- (2+C)C \ln(\numberOfCliques)}
		= 1 - \left(\frac{1}{\numberOfCliques}\right)^{(2+C)C + 1} \left( \frac{1}{\sqrt{12} \maxPartition} \right)^{4(2+C)C - 2}.
	\end{gather}
	
	By combining \cref{eq:main:comparison,eq:main:twoStep}, we get
	\[
		1 - \eigenvalueOf{\transitionMatrix[\blockDynamics]}[2] \ge \frac{1}{\bigO{\maxPartition^{4(2+C)C - 1} \numberOfCliques^{(2+C)C + 1}}}.
	\] 
	As $C$ is assumed to be a constant, this implies the desired mixing time.
\end{proof}


\section{Univariate model: mixing up to uniqueness}
We consider the univariate hard-core model, often just referred to as hard-core model, in which all vertices $v \in \vertices$ have the same weight $\weight[v] = \weight$ for some $\weight \in \R_{>0}$.
We denote an instance of this model by $(\graph, \weight)$.

We define $\criticalWeight{\degree} = \frac{(\degree - 1)^{\degree - 1}}{(\degree - 2)^{\degree}}$ to be the \emph{critical weight} of the hard-core model.
As we discussed in the introduction, $\criticalWeight{\degree}$ is the threshold for correlation decay on general graphs and a tight upper bound for rapid mixing of Glauber dynamics.

We show that the univariate model $(\graph, \weight)$ satisfies \Cref{condition:influence_bound} for all $\weight \le \criticalWeight{\degree}$.
To do so, we use the following recently established result.
\begin{lemma}[\cite{Chen2020rapid}]
	\label{lemma:univariate_chen}
	Let $(\graph, \weight)$ be an instance of the univariate hard-core model and assume that the maximum degree of $\graph$ is bounded by $\degree$.
	If there is a constant $\delta > 0$ such that $\weight \le (1 - \delta) \criticalWeight{\degree}$, then there is a constant $C \in \bigO{\frac{1}{\delta}}$ such that for all $S \subseteq \vertices$ it holds that $\norm{\pairwiseInfluence{\subgraph{\graph}{S}}}[\infty] \le C$.
\end{lemma}

This implies the following result immediately.
\begin{lemma}
	\label{lemma:univariate}
	Let $(\graph, \weight)$ be an instance of the univariate hard-core model and assume that the maximum degree of $\graph$ is bounded by $\degree$.
	If there is a constant $\delta > 0$ such that $\weight \le (1 - \delta) \criticalWeight{\degree}$, then $(\graph, \weight)$ satisfies \Cref{condition:influence_bound} for a constant $C \in \bigO{\frac{1}{\delta}}$.
\end{lemma}

\begin{proof}
	By \Cref{lemma:univariate_chen}, there is a $C \in \bigO{\frac{1}{\delta}}$ such that, for all $S \in \vertices$ and $r \in S$, it holds that
	\[
		\sum_{v \in S} \absolute{\pairwiseInfluence{\subgraph{\graph}{S}}[][r, v]} \le \norm{\pairwiseInfluence{\subgraph{\graph}{S}}}[\infty] \le C .
	\]
	Thus, \Cref{condition:influence_bound} is satisfied for the same constant $C$ and $q(v) = 1$ for all $v \in \vertices$.
\end{proof}

The following claim is a direct consequence of \Cref{thm:main} and \Cref{lemma:univariate}.

\begin{corollary}
	\label{cor:univariate}
	Let $(\graph, \weight)$ be an instance of the univariate hard-core model and assume that the maximum degree of $\graph$ is bounded by $\degree$.
	Let $\clique$ be a clique cover for~$\graph$ of size $\numberOfCliques$, and let $\maxPartition = \max_{i \in [\numberOfCliques]} \{\partitionFunction[\subgraph{\graph}{\clique[i]}]\}$.
	If there is a constant $\delta > 0$ such that $\weight \le (1-\delta) \criticalWeight{\degree}$, then the mixing time of the block dynamics $\blockDynamics[\graph, \weight, \clique]$, starting from $\emptyset \in \independentSets[\graph]$, is bounded by
	\[
		\mix{\blockDynamics}{\err}[\emptyset] \le \numberOfCliques^{\bigO{1/\delta^2}} \maxPartition^{\bigO{1/\delta^2}} \ln\left( \frac{1}{\err} \right) .
		\qedhere
	\] 
\end{corollary}


\section{Multivariate model: comparison to \pmc}\label{sec:multivariate}
In this section, we relate \Cref{condition:influence_bound} to a strict version of the \pmc, first introduced in \cite{friedrich2020polymer}.
\begin{definition}[\Spmc]
	An instance of the multivariate hard-core model $(\graph, \weights)$ satisfies the \spmc for a function $\spmcFunction\colon \vertices \to \R_{>0}$ and a constant $\spmcConstant \in (0, 1)$ if and only if for all $v \in \vertices$ it holds that
	\[
		\sum_{w \in \neighbors{v}} \frac{\weight[w]}{1 + \weight[w]} \spmcFunction[w] \le (1 - \spmcConstant) \spmcFunction[v] .
		\qedhere
	\]
\end{definition}

We show that the \spmc is sufficient to imply \Cref{condition:influence_bound}.
This yields a mixing-time bound for block dynamics based on the \spmc and also bounds the eigenvalues of the pairwise influence matrix, which might be of independent interest.
To obtain our result, we translate the original instance $(\graph, \weights)$ to the self-avoiding-walk tree and apply a recursive proof on this tree.

\subsubsection*{Influence in self-avoiding-walk trees} 
For any instance of the multivariate hard-core model $(\graph , \weights)$ and any vertex $\sawRoot \in \vertices$, let $\sawTree[\graph, \sawRoot]$ denote the tree of self-avoiding walks as defined in \cite{Weitz2006Counting}, which is constructed as follows.
Assume there is a total order of vertices in $v_1, \dots, v_{\numberOfVertices}$, where $\numberOfVertices = \size{\vertices}$.
A \emph{self-avoiding walk} of length $l \ge 2$ is a simple path $v_{i_1}, v_{i_2}, \dots, v_{i_l}$ in $\graph$.
Further, a \emph{closed self-avoiding walk} $v_{i_1}, \dots, v_{i_{l-1}}, v_{i_j}$ of length $l \ge 3$ consists of a self-avoiding walk $v_{i_1}, \dots, v_{i_{l-1}}$ of length $l-1$ and an appended vertex $v_{i_j}$ such that $j \in [l-2]$ and $(v_{i_{l-1}}, v_{i_j}) \in \edges$.
That is, the edge $(v_{i_{l-1}}, v_{i_j})$ closes a cycle $v_{i_j}, v_{i_{j+1}}, \dots, v_{i_{l-1}}$. 
The graph $\sawTree[\graph, \sawRoot]$ consists of all closed self-avoiding walks with $v_{i_1} = \sawRoot$, and all self-avoiding walks with $v_{i_1} = \sawRoot$ and $v_{i_l}$ having degree $1$ in $\graph$.
Note that any vertex $v \in \vertices$ with $v \neq \sawRoot$ might have multiple copies in $\sawTree[\graph, \sawRoot]$.

For any root $\sawRoot \in \vertices$, the multivariate hard-core model $(\graph, \weights)$ is translated to a multivariate hard-core model on $\sawTree[\graph, \sawRoot]$ as follows.
Let $v_{i_1}, \dots, v_{i_l - 1}, v_{i_j}$ be a closed self-avoiding walk.
We fix~$v_{i_j}$ always to be in the independent set (fix spin to $1$) if $i_{j+1} > i_{l-1}$, and we fix it always to be excluded from the independent (fix spin to $0$) otherwise.
We call such vertices \emph{fixed copies}.
For each $v \in \vertices$, let $\sawCopies{v}[\sawTree[\graph, \sawRoot]]$ denote the set of all \emph{unfixed copies of $v$} in $\sawTree[\graph, \sawRoot]$.
We write $\sawCopies{v}$ if the tree $\sawTree[\graph, \sawRoot]$ is clear from the context.
In the multivariate hard-core model on $\sawTree[\graph, \sawRoot]$, each such copy $\sawCopy{v} \in \sawCopies{v}[\sawTree[\graph, \sawRoot]]$ has weight $\weight[\sawCopy{v}] = \weight[v]$.

This way of translating $(\graph, \weights)$ to the tree of self-avoiding walks for some root $\sawRoot \in \vertices$ was shown to have a variety of useful properties.
One of them is that pairwise influences are preserved in the following sense.
\begin{lemma}[{\cite[Lemma $8$]{Chen2020rapid}}]
	\label{lemma:saw_influence}
	Let $(\graph, \weights)$ be an instance of the multivariate hard-core model.
	For all $\sawRoot, v \in \vertices$ and $T = \sawTree[\graph, \sawRoot]$ it holds that
	\[
		\pairwiseInfluence{\graph}[][\sawRoot, v] = \sum_{\sawCopy{v} \in \sawCopies{v}[T]} \pairwiseInfluence{T}[][\sawRoot, \sawCopy{v}] .
		\qedhere
	\] 
\end{lemma}

\Cref{lemma:saw_influence} states that it suffices to discuss the pairwise influence on the self-avoiding walk tree instead of the original graph.
This allows us to use the following multiplicative property for pairwise influence along paths in tree graphs.
\begin{lemma}[{\cite[Lemma B.2]{ALOG20}}]
	\label{lemma:tree_influence}
	Let $T=(\vertices, \edges)$ be a tree and $(T, \weights)$ be a multivariate hard-core model on $T$.
	Further, let $v, w \in \vertices$ be a pair of distinct, non-adjacent vertices, and let $u \in \vertices$ with $u \neq v$ and $u \neq w$ be any vertex on the unique path between $v$ and $w$.
	Then
	\[
		\pairwiseInfluence{T}[][v, w] = \pairwiseInfluence{T}[][v, u] \pairwiseInfluence{T}[][u, w] . 
		\qedhere 
	\]
\end{lemma}

\subsubsection*{Bounding pairwise influence via the \spmc}
We start by proving that the influence of the root on a certain layer in the self-avoiding-walk tree exhibits the following exponential decay in terms of depth. For a tree $T$ and integer $k$ let $\sawLayer{T}{k} \subseteq \sawCopy{V}$ denote the set of vertices in $T$ at layer $k \in \N$.
\begin{lemma}
	\label{lemma:influence_decay}
	Let $(\graph, \weights)$ be a multivariate hard-core model, and let $\sawRoot \in \vertices$.
	Furthermore, let $T = \sawTree[\sawRoot, \graph]$ and let $\sawCopy{V} = \bigcup_{v \in \vertices} \sawCopies{v}[T]$.
	Assume that $(\graph, \weights)$ satisfies the \spmc for a function $\spmcFunction$ and a constant~$\spmcConstant$, and define the function $\sawFunction\colon \sawCopy{V} \to \R_{>0}$ with $\sawFunction[\sawCopy{v}] = \spmcFunction[v]$ for all $\sawCopy{v} \in \sawCopies{v}[T]$ and $v \in \vertices$.
	Then for all $k \in \N_{>0}$ it holds that
	\[
		\sum_{w \in \sawLayer{T}{k}} \absolute{\pairwiseInfluence{T}[][r, w]} \sawFunction[w] \le (1 - \spmcConstant)^k \spmcFunction[r] .
		\qedhere 
	\]
\end{lemma}
\begin{proof}
	Note that if $(\graph, \weights)$ satisfies the \spmc for a function $\spmcFunction$ and a constant $\spmcConstant$, then the corresponding multivariate hard-core instance on $T$ satisfies the \spmc for $\sawFunction$ and the same constant $\spmcConstant$.
	Based on that, we prove our claim by induction on $k$.
	
	\textbf{Base case: $k = 1$.}
	Note that $\sawLayer{T}{1} = \neighbors{r}[T]$.
	Further, we have for each $w \in \neighbors{r}[T]$ by definition
	\[
		\absolute{\pairwiseInfluence{T}[][r, w]} = \absolute{\gibbsPr{T}{\inSet{w}}[\inSet{r}] - \gibbsPr{T}{\inSet{w}}[\notInSet{r}]} \le \frac{\weight[w]}{1 + \weight[w]} .
	\]
	Thus, by the \spmc, we obtain
	\[
		\sum_{w \in \sawLayer{T}{1}} \absolute{\pairwiseInfluence{T}[][r, w]} \sawFunction[w] 
		\le \sum_{w \in \sawLayer{T}{1}} \frac{\weight[w]}{1 + \weight[w]} \sawFunction[w]
		\le (1 - \spmcConstant) \sawFunction[r]
		= (1 - \spmcConstant) \spmcFunction[r],
	\]
	which proves the case $k = 1$.
	
	\textbf{Induction step: $k > 1$.}
    Assume that the statement holds for $k - 1$.
	For every $u \in \sawLayer{T}{k - 1}$, let $T_u$ denote the subtree rooted at~$u$, and let $\sawLayer{T_u}{l}$ denote the vertices at layer $l \in \N$ in $T_u$.
	Note that the sets $\sawLayer{T_u}{1}$ for $u \in \sawLayer{T}{k - 1}$ are a partition of $\sawLayer{T}{k}$.
	By \Cref{lemma:tree_influence}, we get
	\begin{align*}
		\sum_{w \in \sawLayer{T}{k}} \absolute{\pairwiseInfluence{T}[][r, w]} \sawFunction[w]
		&= \sum_{u \in \sawLayer{T}{k-1}} \sum_{w \in \sawLayer{T_u}{1}} \absolute{\pairwiseInfluence{T}[][r, w]} \sawFunction[w] \\
		&= \sum_{u \in \sawLayer{T}{k-1}} \absolute{\pairwiseInfluence{T}[][r, u]} \sum_{w \in \sawLayer{T_u}{1}} \absolute{\pairwiseInfluence{T}[][u, w]} \sawFunction[w] .
	\end{align*}
	Further, for every $u \in \sawLayer{T}{k - 1}$ it holds that $\sawLayer{T_u}{1} \subset \neighbors{u}[T]$, and for all $w \in \sawLayer{T_u}{1}$ we have
	\[
		\absolute{\pairwiseInfluence{T}[][u, w]} = \absolute{\gibbsPr{T}{\inSet{w}}[\inSet{u}] - \gibbsPr{T}{\inSet{w}}[\notInSet{u}]} \le \frac{\weight[w]}{1 + \weight[w]} .
	\]
	Thus, by the \spmc, we get
	\begin{align*}
		\sum_{u \in \sawLayer{T}{k-1}} \absolute{\pairwiseInfluence{T}[][r, u]} \sum_{w \in \sawLayer{T_u}{1}} \absolute{\pairwiseInfluence{T}[][u, w]} \sawFunction[w] 
		&< \sum_{u \in \sawLayer{T}{k-1}} \absolute{\pairwiseInfluence{T}[][r, u]} \sum_{w \in \neighbors{u}[T]} \frac{\weight[w]}{1 + \weight[w]} \sawFunction[w] \\
		&\le (1 - \spmcConstant) \sum_{u \in \sawLayer{T}{k-1}} \absolute{\pairwiseInfluence{T}[][r, u]} \sawFunction[u] .
	\end{align*}
	By the induction hypothesis, we obtain
	\[
		(1 - \spmcConstant) \sum_{u \in \sawLayer{T}{k-1}} \absolute{\pairwiseInfluence{T}[][r, u]} \sawFunction[u]
		\le (1 - \spmcConstant) (1 - \spmcConstant)^{k-1} \spmcFunction[r]
		= (1 - \spmcConstant)^{k} \spmcFunction[r] ,
	\]
	which concludes the proof.
\end{proof}

Now, we use this layer-wise decay in the self-avoiding-walk tree to prove that \Cref{condition:influence_bound} is satisfied.
\pairwiseInfluenceSpmc*
\begin{proof}
	Note that if $(\graph, \weights)$ satisfies the \spmc, the same holds for the instance $(\subgraph{\graph}{S}, \subgraph{\weights}{S})$ for the same function $\spmcFunction$ and constant $\spmcConstant$.
	
	Assume $\subgraph{\graph}{S}$ is connected and let $\graph' = \subgraph{\graph}{S}$.
	Further, let $T = \sawTree[\sawRoot, \graph']$ and $\sawCopy{S} = \bigcup_{v \in S} \sawCopies{v}[T]$, and define the function $\sawFunction\colon \sawCopy{S} \to \R_{> 0}$ as in \Cref{lemma:influence_decay}.
	Recall that, by definition, $\pairwiseInfluence{\graph'}[][\sawRoot, \sawRoot] = 0$.
	By \Cref{lemma:saw_influence}, we get
	\begin{align*}
		\sum_{v \in S} \absolute{\pairwiseInfluence{\graph'}[][\sawRoot, v]} \spmcFunction[v]
		&= \sum_{v \in S \setminus \{\sawRoot\}} \absolute{\pairwiseInfluence{\graph'}[][\sawRoot, v]} \spmcFunction[v] \\
		&= \sum_{v \in S \setminus \{\sawRoot\}} \absolute{\sum_{\sawCopy{v} \in \sawCopies{v}[T]} \pairwiseInfluence{T}[][\sawRoot, \sawCopy{v}]} \spmcFunction[v] \\
		&\le \sum_{v \in S \setminus \{\sawRoot\}} \sum_{\sawCopy{v} \in \sawCopies{v}[T]} \absolute{\pairwiseInfluence{T}[][\sawRoot, \sawCopy{v}]} \sawFunction[\sawCopy{v}].
	\end{align*}
	Note that the sets $\sawCopies{v}[T]$ for $v \in S \setminus \{\sawRoot\}$ are a partition of $\sawCopy{S} \setminus \{\sawRoot\}$.
	Recall that $\sawLayer{T}{k} \subset \sawCopy{S}$ denotes the vertices in $T$ at layer $k \in \N$, and observe that the sets $\sawLayer{T}{k}$ for $k \in \N_{>0}$ are a partition of $\sawCopy{S} \setminus \{\sawRoot\}$ as well.
	Thus, we have 
	\[
		\sum_{v \in S \setminus \{\sawRoot\}} \sum_{\sawCopy{v} \in \sawCopies{v}[T]} \absolute{\pairwiseInfluence{T}[][\sawRoot, \sawCopy{v}]} \sawFunction[\sawCopy{v}]
		\le \sum_{k \in \N_{>0}} \sum_{w \in \sawLayer{T}{k}} \absolute{\pairwiseInfluence{T}[][\sawRoot, w]} \sawFunction[w] .
	\]
	By \Cref{lemma:influence_decay}, we obtain the desired bound:
	\[
		\sum_{k \in \N_{>0}} \sum_{w \in \sawLayer{T}{k}} \absolute{\pairwiseInfluence{T}[][\sawRoot, w]} \sawFunction[w]
		\le \spmcFunction[\sawRoot] \sum_{k \in \N_{>0}} (1 - \spmcConstant)^k
		= \frac{1}{\spmcConstant} \spmcFunction[\sawRoot] .
	\]
	
	Now, assume $\subgraph{\graph}{S}$ is not connected.
	Set $\graph'$ to be the largest connected component in $\subgraph{\graph}{S}$ that contains $\sawRoot$ and let $S' \subset S$ be the set of vertices in $\graph'$.
	The claim follows from applying the proof above to $\graph'$ with vertex set $S'$ and by observing that $\pairwiseInfluence{\subgraph{\graph}{S}}[][\sawRoot, v] = 0$ for all $v \in S \setminus S'$.
\end{proof}

\Cref{lemma:pairwise_influence_spmc} immediately implies with \Cref{thm:main} the following result for the mixing time of block dynamics under \spmc.
\begin{corollary}
	Let $(\graph, \weights)$ be an instance of the multivariate hard-core model.
	Let $\clique$ be a given clique cover for~$\graph$ of size~$\numberOfCliques$, and let $\maxPartition = \max_{i \in [\numberOfCliques]} \{\partitionFunction[\subgraph{\graph}{\clique[i]}]\}$.
	If $(\graph, \weight)$ satisfies the \spmc for a function $\spmcFunction$ and a constant $\spmcConstant$, then the mixing time of the block dynamics $\blockDynamics = \blockDynamics[\graph, \weight, \clique]$, starting from $\emptyset \in \independentSets[\graph]$, is bounded by
	\[
		\mix{\blockDynamics}{\err}[\emptyset] \le \numberOfCliques^{\bigO{1/\alpha^2}} \maxPartition^{\bigO{1/\alpha^2}} \ln\left( \frac{1}{\err} \right) .
	\qedhere
	\]
\end{corollary}

Finally, \Cref{lemma:pairwise_influence_spmc} together with \Cref{lemma:spectral_radius} imply the following result.

\begin{corollary}
	Let $(\graph, \weights)$ be an instance of the multivariate hard-core model that satisfies the \spmc for a function $\spmcFunction$ and a constant $\spmcConstant$.
	For every $S \subseteq \vertices$ it holds that $\eigenvalueOf{\pairwiseInfluence{\subgraph{\graph}{S}}}[1] \le \frac{1}{\spmcConstant}$.
\end{corollary}


\section{Comparing clique dynamics to block dynamics} \label{sec:cique_dynamics_comparison}
Using \Cref{thm:comparison}, we show that a bound on the mixing time of the clique dynamics for any clique cover is obtained based on a bound for the mixing time of the block dynamics for the same clique cover.
\begin{lemma}
	\label{lemma:comparison_clique_dynamics}
	Let $(\graph, \weights)$ be an instance of the multivariate hard-core model, let $\clique$ be a clique cover of $\graph$ with size $\numberOfCliques$, and let $\maxPartition = \max_{i \in [\numberOfCliques]} \{\partitionFunction[\clique[i]]\}$.
	Let $\blockDynamics = \blockDynamics[\graph, \weights, \clique]$ denote the block dynamics and $\cliqueDynamics = \cliqueDynamics[\graph, \weights, \clique]$ the clique dynamics based on $\clique$.
	Then
	\[
		\eigenvalueOf{\transitionMatrix[\cliqueDynamics]}[2] \le 1 - \frac{1}{2 \maxPartition} \big(1 - \eigenvalueOf{\transitionMatrix[\blockDynamics]}[2]\big).
		\qedhere
	\]
\end{lemma}

\begin{proof}
	We aim to apply \Cref{thm:comparison}.
	Note that, by $\stationary[\cliqueDynamics] = \gibbsDistribution = \stationary[\blockDynamics]$, we can choose $\statRatio = 1$.
	We proceed by constructing a set of paths $\canonicalPaths$ and by upper-bounding $\flowRatio{\canonicalPaths}$.
	For every $(\independentSet_s, \independentSet_t) \in \mcEdges{\blockDynamics}$, we choose the path $\canonicalPath[\independentSet_s \independentSet_t]$ via edges in $\mcAllEdges{\cliqueDynamics}$ as follows:
	\begin{itemize}
		\item $\canonicalPath[\independentSet_s \independentSet_t] = (\independentSet_s, \independentSet_t)$ if $\independentSet_s \symDiff \independentSet_t = \{v\}$ for some $v \in \vertices$, and
		\item $\canonicalPath[\independentSet_s \independentSet_t] = (\independentSet_s, \independentSet_s \cap \independentSet_t, \independentSet_t)$ if $\independentSet_s \symDiff \independentSet_t = \{v, w\}$ for $v \neq w$.  		
	\end{itemize}
	
	We set $\canonicalPaths$ to be the set of all such paths.
	Note that we actually only use edges in $\mcEdges{\cliqueDynamics}$, which means that we can ignore self loops while bounding $\flowRatio{\canonicalPaths}$.
	We proceed by doing a case distinction on the transition $(\independentSet, \independentSet') \in \mcEdges{\cliqueDynamics}$.
	
	\textbf{Case 1.}
    Consider $(\independentSet, \independentSet') \in \mcEdges{\cliqueDynamics}$ with $\independentSet = \independentSet' \cup \{v\}$ for a $v \notin \independentSet$ (i.e., removing $v$).
	First, note that
	\begin{align}
		\notag
		\transitionWeight{\cliqueDynamics}[\independentSet, \independentSet']
		&= \gibbsDistributionFunction{\independentSet} \frac{1}{\numberOfCliques} \sum_{\substack{i \in [\numberOfCliques]: \\ v \in \clique[i]}} \gibbsDistributionFunction{\emptyset}[][\clique[i]][\zeroSpinConfig[\vertices \setminus \clique[i]]] \\
		\notag
		&= \gibbsDistributionFunction{\independentSet} \frac{1}{\numberOfCliques} \sum_{\substack{i \in [\numberOfCliques]: \\ v \in \clique[i]}} \frac{1}{\partitionFunction[\subgraph{\graph}{\clique[i]}]} \\
		\label{eq:comparison_clique_dynamics:case1:1}
		&\ge \gibbsDistributionFunction{\independentSet} \frac{1}{\numberOfCliques} \frac{1}{\maxPartition} \size{\set{i \in [\numberOfCliques]}{ v \in \clique[i]}} .
	\end{align}
	Further, for all $(\independentSet_s, \independentSet_t) \in \mcEdges{\blockDynamics}$ such that $(\independentSet, \independentSet') \in \mcEdges{\canonicalPath[\independentSet_s \independentSet_t]}$ it holds that $\independentSet_s = \independentSet$, and either $\independentSet_t = \independentSet'$ or $\independentSet_t = \independentSet' \cup \{w\}$ for some $w \in \vertices$ that shares a clique with $v$.
	If $\independentSet_t = \independentSet'$, then we have
	\begin{align}
		\notag
		\transitionWeight{\blockDynamics}[\independentSet, \independentSet'] 
		&= \gibbsDistributionFunction{\independentSet} \frac{1}{\numberOfCliques} \sum_{\substack{i \in [\numberOfCliques]: \\ v \in \clique[i]}} \gibbsDistributionFunction{\emptyset}[][\clique[i]][\spinConfig[\vertices \setminus \clique[i]][][\independentSet]] \\
		\notag
		&= \gibbsDistributionFunction{\independentSet} \frac{1}{\numberOfCliques} \sum_{\substack{i \in [\numberOfCliques]: \\ v \in \clique[i]}} \dfrac{\gibbsDistributionFunction{\independentSet'}}{ \gibbsDistributionFunction{\independentSet'} + \sum\limits_{\substack{u \in \clique[i]:\\ \independentSet' \cup \{u\} \in \independentSets}} \gibbsDistributionFunction{\independentSet' \cup \{u\}}} \\
		\label{eq:comparison_clique_dynamics:case1:2}
		&= \gibbsDistributionFunction{\independentSet} \frac{1}{\numberOfCliques} \sum_{\substack{i \in [\numberOfCliques]: \\ v \in \clique[i]}} \dfrac{1}{ 1 + \sum\limits_{\substack{u \in \clique[i]:\\ \independentSet' \cup \{u\} \in \independentSets}} \weight[u]} .
	\end{align}
	Otherwise, if $\independentSet_t = \independentSet' \cup \{w\}$, we get
	\begin{align}
		\notag
		\transitionWeight{\blockDynamics}[\independentSet, \independentSet' \cup \{w\} ] 
		&= \gibbsDistributionFunction{\independentSet} \frac{1}{\numberOfCliques} \sum_{\substack{i \in [\numberOfCliques]: \\ v, w \in \clique[i]}} \gibbsDistributionFunction{\{w\}}[][\clique[i]][\spinConfig[\vertices \setminus \clique[i]][][\independentSet]] \\
		\notag
		&= \gibbsDistributionFunction{\independentSet} \frac{1}{\numberOfCliques} \sum_{\substack{i \in [\numberOfCliques]: \\ v, w \in \clique[i]}} \dfrac{\gibbsDistributionFunction{\independentSet' \cup \{w\}}}{ \gibbsDistributionFunction{\independentSet'} + \sum\limits_{\substack{u \in \clique[i]:\\ \independentSet' \cup \{u\} \in \independentSets}} \gibbsDistributionFunction{\independentSet' \cup \{u\}}} \\
		\label{eq:comparison_clique_dynamics:case1:3}
		&= \gibbsDistributionFunction{\independentSet} \frac{1}{\numberOfCliques} \sum_{\substack{i \in [\numberOfCliques]: \\ v, w \in \clique[i]}} \dfrac{\weight[w]}{ 1 + \sum\limits_{\substack{u \in \clique[i]:\\ \independentSet' \cup \{u\} \in \independentSets}} \weight[u]} .
	\end{align}
	Combining \cref{eq:comparison_clique_dynamics:case1:2,eq:comparison_clique_dynamics:case1:3} and observing that $\length{\canonicalPath[\independentSet_s \independentSet_t]} \le 2$ for all involved paths yields
	\begin{align*}
		\sum_{\substack{(\independentSet_s, \independentSet_t) \in \mcEdges{\blockDynamics}: \\(\independentSet, \independentSet') \in \canonicalPath[\independentSet_s \independentSet_t]}} \length{\canonicalPath[\independentSet_s \independentSet_t]} \transitionWeight{\blockDynamics}[\independentSet_s, \independentSet_t]
		&\le 2 \gibbsDistributionFunction{\independentSet} \frac{1}{\numberOfCliques} \sum_{\substack{i \in [\numberOfCliques]: \\ v \in \clique[i]}} 
		\left(\dfrac{1}{ 1 + \sum\limits_{\substack{u \in \clique[i]:\\ \independentSet' \cup \{u\} \in \independentSets}} \weight[u]} + \sum_{\substack{w \in \clique[i]:\\ \independentSet' \cup \{w\} \in \independentSets}} \dfrac{\weight[w]}{ 1 + \sum\limits_{\substack{u \in \clique[i]:\\ \independentSet' \cup \{u\} \in \independentSets}} \weight[u]}\right) \\
		&= 2 \gibbsDistributionFunction{\independentSet} \frac{1}{\numberOfCliques} \size{\set{i \in [\numberOfCliques]}{ v \in \clique[i]}} .
	\end{align*}
	Together with \cref{eq:comparison_clique_dynamics:case1:1}, we get
	\[
		\frac{1}{\transitionWeight{\cliqueDynamics}[\independentSet, \independentSet']} \sum_{\substack{(\independentSet_s, \independentSet_t) \in \mcEdges{\blockDynamics}: \\(\independentSet, \independentSet') \in \canonicalPath[\independentSet_s \independentSet_t]}} \length{\canonicalPath[\independentSet_s \independentSet_t]} \transitionWeight{\blockDynamics}[\independentSet_s, \independentSet_t]
		\le 2 \maxPartition.
	\]
	
    \textbf{Case 2.}
	Consider $(\independentSet, \independentSet') \in \mcEdges{\cliqueDynamics}$ with $\independentSet' = \independentSet \cup \{v\}$ for some $v \notin \independentSet$ (i.e., adding $v$).
	Note that, as~$\cliqueDynamics$ is reversible, it holds that $(\independentSet', \independentSet) \in \mcEdges{\cliqueDynamics}$.
	As $(\independentSet', \independentSet)$ is a transition that deletes~$v$ from~$\independentSet'$, resulting in $\independentSet$, it was already considered in the first case.
	Further, by construction of $\canonicalPaths$, it holds that $(\independentSet, \independentSet')$ is on a path $\canonicalPath[\independentSet_s \independentSet_t]$ for any $(\independentSet_s, \independentSet_t) \in \mcEdges{\blockDynamics}$ if and only if $(\independentSet', \independentSet)$ is on the path $\canonicalPath[\independentSet_t \independentSet_s]$ for $(\independentSet_t, \independentSet_s) \in \mcEdges{\blockDynamics}$.
	Due to reversibility of $\cliqueDynamics$ and $\blockDynamics$, we know that $\transitionWeight{\cliqueDynamics}$ and $\transitionWeight{\blockDynamics}$ are symmetric.
	Thus, we can conclude from our result for $(\independentSet', \independentSet)$ in the first case that
	\[
		\frac{1}{\transitionWeight{\cliqueDynamics}[\independentSet, \independentSet']} \sum_{\substack{(\independentSet_s, \independentSet_t) \in \mcEdges{\blockDynamics}: \\(\independentSet, \independentSet') \in \canonicalPath[\independentSet_s \independentSet_t]}} \length{\canonicalPath[\independentSet_s \independentSet_t]} \transitionWeight{\blockDynamics}[\independentSet_s, \independentSet_t]
		= \frac{1}{\transitionWeight{\cliqueDynamics}[\independentSet', \independentSet]} \sum_{\substack{(\independentSet_t, \independentSet_s) \in \mcEdges{\blockDynamics}: \\(\independentSet', \independentSet) \in \canonicalPath[\independentSet_t \independentSet_s]}} \length{\canonicalPath[\independentSet_t \independentSet_s]} \transitionWeight{\blockDynamics}[\independentSet_t, \independentSet_s]
		\le 2 \maxPartition .
	\]

	As $\cliqueDynamics$ changes at most one vertex at each step, the above case distinction is complete and we get $\flowRatio{\canonicalPaths} \le 2 \maxPartition$, which by \Cref{thm:comparison} concludes the proof.
\end{proof}

We immediately obtain the following corollary, which is central for our proof of \Cref{thm:hard_sphere_approx}.
\univariateCliqueDynamics*

\begin{proof}
	Let $\blockDynamics = \blockDynamics[\graph, \weight, \clique]$.
	Due to \Cref{lemma:comparison_clique_dynamics} we know that $\mix{\cliqueDynamics}{\err}[\emptyset] \le 2 \maxPartition \mix{\blockDynamics}{\err}[\emptyset]$.
	Bounding $\mix{\blockDynamics}{\err}[\emptyset]$ based on \Cref{cor:univariate} proves the claim.
\end{proof}


\section{The monoatomic hard-sphere model}\label{sec:hard_spheres}
We study the \emph{grand canonical ensemble} of the monoatomic hard-sphere model in a $d$-dimensional finite cubic region $\volume = [0, \sidelength)^d$ of Euclidean space with side length $\sidelength \in \R_{\geq 1}$.
We write $\vol{\volume} = \sidelength^d$ for the volume of $\volume$.
The hard-sphere model describes the distribution of identical particles, represented as $d$-dimensional balls in $\volume$. 
This distribution is governed by a fugacity parameter $\hsFugacity \in \R_{>0}$, describing the contribution of each particle to the chemical potential, and hard-core interactions between particles, meaning that no two particles are allowed to overlap.
For simplicity, it is common to assume particles to have volume $1$, meaning that their radius is $\radius = (1/\normSphere{d})^{1/d}$, where $\normSphere{d}$ denotes the volume of a unit sphere in $d$ dimensions.

A probabilistic interpretation of grand canonical ensemble is that the centers of particles are distributed according to a Poisson point process on $\volume$ with activity $\hsFugacity$, conditioned on the fact that particles are non-overlapping (i.e., each pair of distinct centers have distance at least $2r$).
Note that this implies that particles are indistinguishable, meaning that exchanging the positions of two particles results in exactly the same configuration of the system.
We aim for approximating the \emph{grand canonical partition function}, which can be seen as the normalizing constant of the resulting distribution of system states.
As a reminder, the partition function can formally be defined as
\[
	\hardSpherePrt[\volume, \hsFugacity] = 
	1 + \sum_{\numParticles \in \N_{> 0}}  
	\frac{\hsFugacity^\numParticles}{\numParticles !}
	\int_{\volume^\numParticles} \valid[x^{(1)}, \dots, x^{(\numParticles)}] \,\d \lebesgue{d \times \numParticles} ,
\]
where 
\[
	\valid[x^{(1)}, \dots, x^{(\numParticles)}] = \begin{cases}
		1 \emph{ if $\dist{x^{(i)}}{x^{(j)}} \ge 2r$ for all $i, j \in [\numParticles]$ with $i \neq j$} \\
		0 \emph{ otherwise} 
	\end{cases} 
\]
and $\lebesgue{d \times \numParticles}$ is the Lebesgue measure on $\R^{d \times \numParticles}$.

\subsection{Hard-core representation}\label{subsec:hard_spheres:discretization}
To apply our result for clique dynamics to the continuous hard-sphere model, we will approximate it by an instance of the hard-core model.
The main idea of this discretization is to restrict the centers of spheres to vertices in an integer grid, while scaling the fugacity $\hsFugacity$ and the radius $\radius$ appropriately.
The resulting discrete hard-sphere model can easily be transformed into a hard-core instance.
We proceed by formalizing the direct transformation from the continuous hard-sphere model instance to the discrete hard-core model.

Let $(\volume, \hsFugacity)$ be an instance of the continuous hard-sphere model with $\volume = [0, \sidelength)^d$.
Recall that we fixed the radius $\radius = (1/\normSphere{d})^{1/d}$.
Let $\grid[n] = \Z^d \cap [0, n)^d$ be a finite integer grid of side length $n \in \N_{>0}$.
For any $\resolution \in \R_{>0}$ such that $\resolution \sidelength  \in \N_{>0}$, the hard-core representation of $(\volume, \hsFugacity)$ with resolution $\resolution$ is a hard-core model $(\hsGraph{\resolution}, \hsWeight{\resolution})$ with $\hsGraph{\resolution} = (\hsVertices{\resolution}, \hsEdges{\resolution})$ and
\begin{itemize}
	\item there is a vertex $\hsVertex{x} \in \hsVertices{\resolution}$ for each grid point $x \in \grid[\resolution \sidelength]$,
	\item there is an edge $(\hsVertex{x}, \hsVertex{y}) \in \hsEdges{\resolution}$ for any pair of grid points $x, y \in \grid[\resolution \sidelength]$ with $x \neq y$ and $\dist{x}{y} \le 2  \resolution \radius$, and
	\item $\hsWeight{\resolution} = \resolution^{-d} \hsFugacity$.
\end{itemize}
Note that in the above definition $\dist{x}{y}$ denotes the Euclidean distance.

We will use the following convergence result for the partition function of the hard-core representation in terms of the resolution $\resolution$ to approximate the hard-sphere partition function.   
\hardsphereConvergence*

\begin{proof}
	Note that it suffices to bound the additive error $\absolute{\hardSpherePrt[\volume, \hsFugacity] - \partitionFunction[\hsGraph{\resolution}, \hsWeight{\resolution}]}$.
	Because $\hardSpherePrt[\volume, \hsFugacity] \ge 1$, this directly results in the desired multiplicative bound.
	
	In order to obtain an additive bound, we start by transforming $\partitionFunction[\hsGraph{\resolution}, \hsWeight{\resolution}]$ to a form that is more similar to the form of $\hardSpherePrt[\volume, \hsFugacity]$.
	
	Let $\volume = [0, \sidelength)^{d}$ and let $\rescale[\resolution]\colon \grid[\resolution \sidelength] \to \volume$ with $(x_1, \dots, x_d) \mapsto \rescale[\resolution][x] = \left( x_1/\resolution, \dots, x_d/\resolution \right)$.
	Note that, for all $x^{(i)}, x^{(j)} \in \grid[\resolution \sidelength]$ it holds that
	\[
		\dist{x^{(i)}}{x^{(j)}} \ge 2 \resolution \radius 
		~\leftrightarrow~
		\dist{\rescale[\resolution][x^{(i)}]}{\rescale[\resolution][x^{(j)}]} \ge 2 \radius .
	\]
	Thus, we see that
	\begin{align}
		\notag
		\partitionFunction[\hsGraph{\resolution}, \hsWeight{\resolution}]
		&= \sum_{\independentSet \in \independentSets[\hsGraph{\resolution}]} \hsWeight{\resolution}^{\size{\independentSet}} \\
		\notag
		&= 1 + \sum_{\numParticles \in \N_{>0}} \sum_{\substack{\independentSet \in \independentSets[\hsGraph{\resolution}]\\ \size{\independentSet} = \numParticles}} \hsWeight{\resolution}^\numParticles \\
		\notag
		&= 1 + \sum_{\numParticles \in \N_{> 0}} \frac{\hsWeight{\resolution}^\numParticles}{\numParticles!} \sum_{\substack{\left( x^{(1)}, \dots, x^{(\numParticles)} \right) \\\in \left( \grid[\resolution \sidelength] \right)^\numParticles}}  \valid[\rescale[\resolution][x^{(1)}], \dots, \rescale[\resolution][x^{(\numParticles)}]] \\
		\label{eq:discreteHardSpherePartitationFunction}
		&= 1 + \sum_{\numParticles \in \N_{> 0}} \frac{\hsFugacity^\numParticles}{\numParticles!} 
			\sum_{\substack{\left( x^{(1)}, \dots, x^{(\numParticles)} \right) \\\in \left( \grid[\resolution \sidelength] \right)^\numParticles}}  \left(\frac{1}{\resolution}\right)^{d \cdot \numParticles} \valid[\rescale[\resolution][x^{(1)}], \dots, \rescale[\resolution][x^{(\numParticles)}]] .
	\end{align}
	
	We continue by rewriting
	\[
		\sum_{\substack{\left( x^{(1)}, \dots, x^{(\numParticles)} \right) \\\in \left( \grid[\resolution \sidelength] \right)^\numParticles}}  \left(\frac{1}{\resolution}\right)^{d \cdot \numParticles} \valid[\rescale[\resolution][x^{(1)}], \dots, \rescale[\resolution][x^{(\numParticles)}]]	
	\]
	for any fixed $\numParticles \in \N_{> 0}$.
	Let $\rescale[\resolution][\grid[\resolution \sidelength]] \subseteq \volume$ denote the image of $\rescale[\resolution]$, and let $\remap[\resolution]\colon \volume \to \rescale[\resolution][\grid[\resolution \sidelength]]$ with
	\[
		(x_1, \dots, x_d) \mapsto \left( \frac{\lfloor \resolution x_1 \rfloor }{\resolution}, \dots, \frac{\lfloor \resolution x_d \rfloor }{\resolution} \right).
	\]
	Further, for all $\numParticles \in \N_{> 0}$ and all $\left( x^{(1)}, \dots, x^{(\numParticles)} \right) \in \left( \rescale[\resolution][\grid[\resolution \sidelength]] \right)^\numParticles$, let
	\begin{align*}
		W^{(\resolution)}_{x^{(1)}, \dots, x^{(\numParticles)}} 
		& = \set{ \left( y^{(1)}, \dots, y^{(\numParticles)} \right) \in \volume^\numParticles}{ \forall i \in [\numParticles]\colon \remap[\resolution][y^{(i)}] = x^{(i)} } \\ 
		& = \left( \invremap{\resolution}{x^{(1)}} \right) \times \dots \times \left( \invremap{\resolution}{x^{(\numParticles)}} \right) .
	\end{align*}
	Note that the sets $W^{(\resolution)}_{x^{(1)}, \dots, x^{(\numParticles)}}$ partition $\volume^{\numParticles}$ into $(d \times k)$-dimensional hypercubes of side length~$1/\resolution$.
	Thus, for all $\left( x^{(1)}, \dots, x^{(\numParticles)} \right) \in \left( \rescale[\resolution][\grid[\resolution \sidelength]] \right)^\numParticles$, it holds that
	\[
		\lebesgue{d \times \numParticles}[W^{(\resolution)}_{x^{(1)}, \dots, x^{(\numParticles)}}] = \left( \frac{1}{\resolution} \right)^{d \cdot k} .	
	\]
	By this and by the definition of a Lebesgue integral for elementary functions, we obtain
	\begin{align*}
		\sum_{\substack{\left( x^{(1)}, \dots, x^{(\numParticles)} \right) \\\in \left( \grid[\resolution \sidelength] \right)^\numParticles}} 
		&	\left( \frac{1}{\resolution} \right)^{d \cdot k} 
		\valid[\rescale[\resolution][x^{(1)}], \dots, \rescale[\resolution][x^{(\numParticles)}]] \\
		& = \sum_{\substack{\left( x^{(1)}, \dots, x^{(\numParticles)} \right) \\\in \left( \grid[\resolution \sidelength] \right)^\numParticles}} 
		\lebesgue{d \times \numParticles}[W^{(\resolution)}_{\rescale[\resolution][x^{(1)}], \dots, \rescale[\resolution][x^{(\numParticles)}]}] \cdot
		\valid[\rescale[\resolution][x^{(1)}], \dots, \rescale[\resolution][x^{(\numParticles)}]] \\
		& = \sum_{\substack{\left( x^{(1)}, \dots, x^{(\numParticles)} \right) \\\in \left( \rescale[\resolution][\grid[\resolution \sidelength]] \right)^\numParticles}} 
		\lebesgue{d \times \numParticles}[W^{(\resolution)}_{x^{(1)}, \dots, x^{(\numParticles)}}] \cdot
		\valid[x^{(1)}, \dots, x^{(\numParticles)}] \\
		& = \int_{\volume^{\numParticles}} \valid[\remap[\resolution][x^{(1)}], \dots, \remap[\resolution][x^{(\numParticles)}]] \,\d \lebesgue{d \times \numParticles} .
	\end{align*}
	
	Substituting this expression back into \cref{eq:discreteHardSpherePartitationFunction} yields
	\[
		\partitionFunction[\hsGraph{\resolution}, \hsWeight{\resolution}] 
		= 1 + \sum_{\numParticles \in \N_{> 0}} \frac{\hsFugacity^\numParticles}{\numParticles!}
			\int_{\volume^\numParticles} \valid[\remap[\resolution][x^{(1)}], \dots, \remap[\resolution][x^{(\numParticles)}]] \,\d \lebesgue{d \times \numParticles} .
	\]  
	
	We now express $\absolute{\hardSpherePrt[\volume, \hsFugacity] - \partitionFunction[\hsGraph{\resolution}, \hsWeight{\resolution}]}$ in terms of the absolute difference of the integrals for all $\numParticles \in \N_{> 0}$.
	It holds that
	\begin{gather*} 
		\absolute{\int_{\volume^\numParticles} \valid[x^{(1)}, \dots, x^{(\numParticles)}] \,\d \lebesgue{d \times \numParticles} - \int_{\volume^\numParticles} \valid[\remap[\resolution][x^{(1)}], \dots, \remap[\resolution][x^{(\numParticles)}]] \,\d \lebesgue{d \times \numParticles}} \\
		\le \int_{\volume^\numParticles} \absolute{\valid[x^{(1)}, \dots, x^{(\numParticles)}] - \valid[\remap[\resolution][x^{(1)}], \dots, \remap[\resolution][x^{(\numParticles)}]]} \,\d \lebesgue{d \times \numParticles} .
	\end{gather*}
	
	Let $N^{(\resolution)} \subseteq \volume^\numParticles$ be such that for all $\left( x^{(1)}, \dots, x^{(\numParticles)} \right) \in N^{(\resolution)}$ it holds that $\valid[x^{(1)}, \dots, x^{(\numParticles)}] \neq \valid[\remap[\resolution][x^{(1)}], \dots, \remap[\resolution][x^{(\numParticles)}]]$.
	As $\valid$ is an indicator function, it holds that
	\[
		\int_{\volume^\numParticles} \absolute{ \valid[x^{(1)}, \dots, x^{(\numParticles)}] - \valid[\remap[\resolution][x^{(1)}], \dots, \remap[\resolution][x^{(\numParticles)}]] } \,\d \lebesgue{d \times \numParticles}
		= \lebesgue{d \times \numParticles}[N^{(\resolution)}] .	
	\]
	
	We construct a superset of $N^{(\resolution)}$, for which we calculate the Lebesgue measure.
	First, note that $N^{(\resolution)} = \emptyset$ for $k=1$, as in this case $\valid[x^{(1)}] = \valid[\remap[\resolution][x^{(1)}]] = 1$ for all $x^{(1)} \in \volume$.
	Further, let $K = \big( \sidelength \sqrt{d}/(2 \radius) \big)^d$. 
	Note that, for all $\numParticles > K$, it holds that at least two particles have distance less than $2 \radius$, meaning that such a configuration has always overlapping particles and $N^{(\resolution)} = \emptyset$.
	We are left with considering $2 \le \numParticles \le K$.
	
	We observe that, for all $\left( x^{(1)}, \dots, x^{(\numParticles)} \right) \in \volume^\numParticles$ such that
	\[
		\valid[x^{(1)}, \dots, x^{(\numParticles)}] \neq \valid[\remap[\resolution][x^{(1)}], \dots, \remap[\resolution][x^{(\numParticles)}]],
	\]
	there is a pair of points $x^{(i)}, x^{(j)}$ for $i, j \in [\numParticles]$ such that $i \neq j$ and
	\begin{align*}
		&\dist{x^{(i)}}{x^{(j)}} < 2 \radius \le \dist{\remap[\resolution][x^{(i)}]}{\remap[\resolution][x^{(j)}]} \ \textrm{ or}\\
		&\dist{x^{(i)}}{x^{(j)}} \ge 2 \radius > \dist{\remap[\resolution][x^{(i)}]}{\remap[\resolution][x^{(j)}]} .
	\end{align*}
	As, for every point $x^{(i)} \in \volume$, it holds that
	\[
		\dist{x^{(i)}}{\remap[\resolution][x^{(i)}]} \le \frac{\sqrt{d}}{\resolution},
	\]
	there is a pair of points $x^{(i)}, x^{(j)}$ for $i, j \in [\numParticles]$ such that $i \neq j$ and
	\[
		\absolute{2 \radius - \dist{x^{(i)}}{x^{(j)}}} \le 2 \frac{\sqrt{d}}{\resolution} .
	\]
	For all $i, j \in [\numParticles]$ with $i \neq j$ let $S^{(\resolution)}_{i, j} \subseteq \volume^{\numParticles}$ be the set of points $\left( x^{(1)}, \dots, x^{(\numParticles)} \right) \in \volume^\numParticles$ such that this is the case.
	Then
	\[
		\lebesgue{d \times \numParticles}[N^{(\resolution)}] 
		\le \lebesgue{d \times \numParticles}[\bigcup_{1 \le i < j \le \numParticles} S^{(\resolution)}_{i, j}]	
		\le \sum_{1 \le i < j \le \numParticles} \lebesgue{d \times \numParticles}[S^{(\resolution)}_{i, j}] .
	\]
	
	By Fubini's theorem, noting that~$S^{(\resolution)}_{i, j}$ only depends on~$i$ and~$j$, we get
	\begin{align*}
		\lebesgue{d \times \numParticles}[S^{(\resolution)}_{i, j}] 
		& = \int_{\volume^\numParticles}  \indicator{\absolute{2 \radius - \dist{x^{(i)}}{x^{(j)}}} \le 2 \frac{\sqrt{d}}{\resolution}} \,\d \lebesgue{d \times \numParticles} \\
		& = \sidelength^{d (\numParticles - 2)} \int_{\volume^2}  \indicator{\absolute{2 \radius - \dist{x^{(i)}}{x^{(j)}}} \le 2 \frac{\sqrt{d}}{\resolution}}  \,\d \lebesgue{d \times 2} \\
		& \le \sidelength^{d (\numParticles - 1)} \cdot \left( \left( 2 \radius + 2 \frac{\sqrt{d}}{\resolution} \right)^d - \left( 2 \radius - 2 \frac{\sqrt{d}}{\resolution} \right)^d \right)\ ,
	\end{align*}
	where the last equality comes from the fact that $\radius$ was chosen as the radius of a ball of volume $1$ in $d$ dimensions.
	By the assumption $\resolution \ge 2 \sqrt{d}$ and the binomial theorem, we further bound
	\begin{align*}
		\left( 2 \radius + 2 \frac{\sqrt{d}}{\resolution} \right)^d - \left( 2 \radius - 2 \frac{\sqrt{d}}{\resolution} \right)^d &= \sum_{i = 0}^{d} 2\cdot\indicator{i \text{ is odd}} \binom{d}{i} \big(2 \radius\big)^{d - i} \left(2 \frac{\sqrt{d}}{\resolution}\right)^i\\
		&\hspace*{-2 em}= \frac{2\sqrt{d}}{\resolution}\sum_{i = 1}^{d} 2\cdot\indicator{i \text{ is odd}} \binom{d}{i} \big(2 \radius\big)^{d - i} \left(2 \frac{\sqrt{d}}{\resolution}\right)^{i - 1}\\
		&\leq \frac{2\sqrt{d}}{\resolution}\sum_{i = 1}^{d} 2\cdot\indicator{i \text{ is odd}} \binom{d}{i} \big(2 \radius\big)^{d - i} 1^{i - 1}\\
		&\le 2\frac{2 \sqrt{d}}{\resolution} \big( 2 \radius + 1 \big)^d.
	\end{align*}
	Using this bound for $\lebesgue{d \times \numParticles}[S^{(\resolution)}_{i, j}]$, we obtain
	\begin{align*}
		\lebesgue{d \times \numParticles}[N^{(\resolution)}] 
		\le \numParticles^2 \cdot 2 \cdot \sidelength^{d (\numParticles - 1)} \cdot \frac{2 \sqrt{d}}{\resolution} \cdot \left( 2 \radius + 1 \right)^d .		
	\end{align*}
	Thus, we get
	\begin{align*}
		\absolute{\hardSpherePrt[\volume, \hsFugacity] - \partitionFunction[\hsGraph{\resolution}, \hsWeight{\resolution}]} 
		& \le \sum_{\numParticles = 2}^{K} \frac{\hsFugacity^\numParticles}{\numParticles!} \lebesgue{d \times \numParticles}[N^{(\resolution)}] \\
		& \le \frac{1}{\resolution}\sum_{\numParticles = 2}^{K} \frac{\hsFugacity^\numParticles}{\numParticles!} \numParticles^2 \cdot 4 \sidelength^{d (\numParticles - 1)} \cdot \sqrt{d} \cdot \left( 2 \radius + 1 \right)^d . 
	\end{align*} 
	
	We simplify the bound further by
	\begin{align*}
		\frac{1}{\resolution} \sum_{\numParticles = 2}^{K} \frac{\hsFugacity^\numParticles}{\numParticles!} \numParticles^2 \cdot 4 \sidelength^{d (\numParticles - 1)} \cdot \sqrt{d} \cdot \left( 2 \radius + 1 \right)^d 
		&\le \frac{1}{\resolution} K^2 \cdot 4 \sidelength^{d (K - 1)} \cdot \sqrt{d} \cdot \left( 2 \radius + 1 \right)^d \sum_{\numParticles = 2}^{K} \frac{\hsFugacity^\numParticles}{\numParticles!} \\
		&\le \frac{1}{\resolution} K^2 \cdot 4 \sidelength^{d (K - 1)} \cdot \sqrt{d} \cdot \left( 2 \radius + 1 \right)^d \eulerE^{\hsFugacity}, 
	\end{align*}	
	where the last inequality follows from the Taylor expansion of $\eulerE^x$ at~$0$. 
	
	Overall, we bound
	\begin{align*}
	\absolute{\hardSpherePrt[\volume, \hsFugacity] - \partitionFunction[\hsGraph{\resolution}, \hsWeight{\resolution}]} 
	&\le \frac{1}{\resolution} K^2 \cdot 4 \sidelength^{d (K - 1)} \cdot \sqrt{d} \cdot \left( 2 \radius + 1 \right)^d \eulerE^{\hsFugacity}\\
	& \le \frac{1}{\resolution} \eulerE^{\bigTheta{K  d \ln (\sidelength) + \ln \left( \radius + 1 \right) + \eulerE^{\hsFugacity}}} .
	\end{align*}
	Observe that $\radius \in \bigO{1}$ and $\eulerE^{\hsFugacity} \in \bigO{1}$.
	Further, for $\radius = (1/\normSphere{d})^{1/d}$ it holds that $K \in \bigO{\sidelength^d}$.
	Thus we have  
	\[
		\absolute{\hardSpherePrt[\volume, \hsFugacity] - \partitionFunction[\hsGraph{\resolution}, \hsWeight{\resolution}]}
		\le \frac{1}{\resolution} \eulerE^{\bigTheta{\sidelength^d  \ln (\sidelength^d)}}
		= \frac{1}{\resolution} \eulerE^{\bigTheta{ \vol{\volume} \ln (\vol{\volume})}},
	\]
	which concludes the proof.
\end{proof}

\subsection{Approximation bound}
We aim for applying \Cref{cor:univarite_clique_dynamics} to the hard-core representation of the hard-sphere model.
In order to do so, we need a bound on the maximum degree $\hsDegree{\resolution}$ of the graph $\hsGraph{\resolution}$ for any sufficiently large resolution $\resolution$.
Let $\integerSphere{d}[s]$ denote the number of integer grid points in a $d$-dimensional sphere of radius $s$ centered at the origin.
Note that the number of neighbors of a vertex $\hsVertex{x} \in \hsVertices{\resolution}$ for any grid point $x \in \grid[\resolution \sidelength]$ is upper bounded by $\integerSphere{d}[2 \resolution r]$.
We use the following bound on $\integerSphere{d}$.
\begin{lemma}
	\label{lemma:integer_sphere}
	Let $\gaussCircleError \in (0, 1]$ and $s \in \R_{> 0}$. 
	For all $\resolution \ge \big( 2 \sqrt{d} \big)^d / (\gaussCircleError s)$ it holds that $\integerSphere{d}[\resolution s] \le (1 + \gaussCircleError) \cdot \normSphere{d} \cdot \left( \resolution s \right)^d$.   
\end{lemma}

\begin{proof}
	We start by considering a sphere of radius $\resolution s + \sqrt{d}$ at the origin.
	Note that this enlarged sphere contains for each grid point $(x_1, \dots, x_d)$ in the original sphere the cubic region $[x_1, x_1 + 1] \times \dots \times [x_d, x_d + 1]$ of volume $1$.
	Thus, the volume of the enlarged sphere is a trivial upper bound on the number of grid points in the original sphere. 
	
	Formally, we get
	\[
		\integerSphere{d}[\resolution s] 
		\le \normSphere{d} \cdot \left( \resolution s + \sqrt{d} \right)^d , 	
	\]
	which we rewrite as
	\[
		\normSphere{d} \cdot \left( \resolution s + \sqrt{d} \right)^d 
		= \normSphere{d} \cdot \left( \resolution s \right)^d + \normSphere{d} \cdot \sum_{i \in [d]} \binom{d}{i} \left( 	\resolution s \right)^{d-i} \sqrt{d}^i	.
	\]
	Further, note that for our choice of $\resolution$ it holds that $\resolution s \ge 1$.
	Thus, we get
	\[
		\normSphere{d} \cdot \left( \resolution s \right)^d + \normSphere{d} \cdot \sum_{i \in [d]} \binom{d}{i} \left( \resolution s \right)^{d-i} \sqrt{d}^i 
		\le 
		\normSphere{d} \cdot \left( \resolution s \right)^d + \normSphere{d} \cdot \left( \resolution s \right)^{d-1} \cdot 2^d \sqrt{d}^d 
		=
		\normSphere{d} \cdot \left( \resolution s \right)^d \cdot \left( 1 + \frac{1}{\resolution s} \left( 2 \sqrt{d} \right)^d \right) .
	\]
We conclude the proof by noting that $\big( 2 \sqrt{d} \big)^d / \resolution s \le \gaussCircleError$.	
\end{proof}

As we fixed $\radius = (1/\normSphere{d})^{1/d}$ we can immediately conclude that for every $\gaussCircleError \in (0, 1]$ there is some $\resolution_{\gaussCircleError} \in \bigTheta{1/\gaussCircleError}$ such that for all $\resolution \ge \resolution_{\gaussCircleError}$ it holds that
\[
	\hsDegree{\resolution} \le (1 + \gaussCircleError) \left(2 \resolution \right)^d .
\]

Finally, the following general lemma will help us to turn a sampling scheme for $\gibbsDistribution[\hsGraph{\resolution}, \hsWeight{\resolution}]$ into a randomized approximation of $\partitionFunction[\hsGraph{\resolution}, \hsWeight{\resolution}]$.
\begin{lemma}[{\cite[Lemma $13$]{friedrich2020polymer}}]
	\label{lemma:appx_partition_function}
	Let $(\graph, \weights)$ be an instance of the multivariate hard-core model and let $\clique$ be a clique cover of size $\numberOfCliques$ with $\maxPartition = \max_{i \in [\numberOfCliques]} \{\partitionFunction[\subgraph{\hsGraph{\resolution}}{\clique[i]}]\}$.
	Further, for $i \in [\numberOfCliques]$ let $\vertices_i = \vertices \setminus \bigcup_{j < i} \clique[j]$.
	For every $\err \in (0, 1]$ there are $s \in \bigTheta{\numberOfCliques \maxPartition / \err^2}$ and $\err_s \in \bigTheta{\err/(\numberOfCliques \maxPartition)}$ such that a randomized $\err$-approximation of $\partitionFunction[\graph, \weights]$ can be computed by drawing $s$ samples $\err_s$-approximately from $\gibbsDistribution[\subgraph{\graph}{\vertices_i}]$ for each $i \in [\numberOfCliques]$. 
\end{lemma}
Note that sampling from $\gibbsDistribution[\subgraph{\graph}{\vertices_i}]$ means sampling from $\gibbsDistribution[\graph]$ for $i = 0$ and ignoring all cliques $\{\clique_j\}_{j < i}$ for $i \ge 1$.

\hardsphereApproximation*
\begin{proof}
	Set $\gaussCircleError = \delta/2$ and $\err' = \err/3$.
	By combining \Cref{lemma:hard_sphere_convergence} and \Cref{lemma:integer_sphere}, we know that we can choose a resolution $\resolution \in \bigTheta{\eulerE^{\vol{\volume} \ln \vol{\volume}} / \left(\err' \gaussCircleError\right)} = \bigTheta{\eulerE^{\vol{\volume} \ln \vol{\volume}} / \left(\err \gaussCircleError\right)}$ such that
	\begin{gather}
		\label{eq:hard_sphere_approx:convergence}
		1 - \err' \le \frac{\hardSpherePrt[\volume, \hsFugacity]}{\partitionFunction[\hsGraph{\resolution}, \hsWeight{\resolution}]} \le 1 + \err' \text{ and }\\
		\label{eq:hard_sphere_approx:integer_sphere}
		\hsDegree{\resolution} \le (1 + \gaussCircleError) \left(2 \resolution \right)^d .
	\end{gather}
	Note that $(1 - \err')^2 \ge 1 - \err$ and $(1 + \err')^2 \le 1 + \err$.
	Thus, \cref{eq:hard_sphere_approx:convergence} implies that it is sufficient to $\err'$-approximate $\partitionFunction[\hsGraph{\resolution}, \hsWeight{\resolution}]$.
	We start by arguing that we can apply \Cref{cor:univarite_clique_dynamics} to $\partitionFunction[\hsGraph{\resolution}, \hsWeight{\resolution}]$.
	Then, we construct a clique cover and show that each step of the clique dynamics can be computed efficiently.
	Finally, we will use \Cref{lemma:appx_partition_function} to get the desired approximation.
	
	To apply \Cref{cor:univarite_clique_dynamics}, we need to show that $\hsWeight{\resolution} \le \left(1 - \delta'\right) \criticalWeight{\hsDegree{\resolution}}$ for some $\delta' \in (0, 1]$. 
	To this end, we choose $\delta' = \delta/2$.
	Due to \cref{eq:hard_sphere_approx:integer_sphere} we know that
	\[
		\criticalWeight{\hsDegree{\resolution}} 
		= \frac{(\hsDegree{\resolution} - 1)^{\hsDegree{\resolution} - 1}}{(\hsDegree{\resolution} - 2)^{\hsDegree{\resolution}}}
		\ge \frac{\left((1 + \gaussCircleError) \left(2 \resolution \right)^d - 1\right)^{(1 + \gaussCircleError) \left(2 \resolution \right)^d - 1}}{\left((1 + \gaussCircleError) \left(2 \resolution \right)^d - 2\right)^{(1 + \gaussCircleError) \left(2 \resolution \right)^d}} .
	\]
	Now, note that
	\[
		\hsFugacity 
		\le (1 - \delta) \frac{\eulerE}{2^d}
		\le \frac{1 - \frac{\delta}{2}}{1 + \frac{\delta}{2}} \frac{\eulerE}{2^d}
		= \frac{1 - \delta'}{1 + \gaussCircleError} \frac{\eulerE}{2^d}
		\le \left(1 - \delta'\right) \resolution^{d} \frac{\left((1 + \gaussCircleError) \left(2 \resolution \right)^d - 1\right)^{(1 + \gaussCircleError) \left(2 \resolution \right)^d - 1}}{\left((1 + \gaussCircleError) \left(2 \resolution \right)^d - 2\right)^{(1 + \gaussCircleError) \left(2 \resolution \right)^d}} ,
	\]
	where the last inequality comes from the fact that $x \frac{(x-1)^{x-1}}{(x-2)^{x}}$ converges to $\eulerE$ from above as $x \to \infty$.
	Dividing by $\resolution^d$ yields $\hsWeight{\resolution} = \resolution^{-d} \hsFugacity \le \left(1 - \delta'\right) \criticalWeight{\hsDegree{\resolution}}$.
	
	We now construct the clique cover that we are going to use.
	This is done by dividing the grid $\grid[\resolution \sidelength]$ into cubic regions of side length $a = \left\lfloor \frac{2 \resolution}{\sqrt{d}} \normSphere{d}^{-1/d} \right\rfloor$.  
	Formally, for a tuple $(i_{1}, \dots, i_{d}) \in \N^d$, let 
	\[
		\subgrid[i_{1}, \dots, i_{d}] = \set{(x_1, \dots, x_d) \in \grid }{\forall j \in [d]\colon i_j a \le x_j < (i_j + 1) a}.	
	\]
	Note that for every pair of grid points $x, y \in \subgrid[i_{1}, \dots, i_{d}]$ it holds that $\dist{x}{y} < 2 \resolution \normSphere{d}^{-1/d} = 2 \resolution \radius$.
	Thus, the set of vertices, corresponding to grid points in $\subgrid[i_{1}, \dots, i_{d}]$, form a clique in $\hsGraph{\resolution}$.
	We obtain a clique cover $\clique$ of size $\numberOfCliques = \size{\clique} \in \bigO{\left(\resolution \sidelength / a \right)^d} = \bigO{\vol{\volume}}$.
	Further, it holds that
	\[
		\maxPartition \le 1 + a^d \hsWeight{\resolution} = 1 + a^d \resolution^{-d} \hsFugacity \in \bigO{1}.
	\]
	By \Cref{cor:univarite_clique_dynamics}, the clique dynamics based on $\clique$ have mixing time polynomial in $\vol{\volume}^{1/\delta'^2}$, thus also polynomial in $\vol{\volume}^{1/\delta^2}$, and in $\ln\left( 1 / \err_{s}\right)$ for any sampling error $\err_{s} \in (0, 1]$.
	
	We proceed by arguing that we can compute each step efficiently.
	Note that we cannot construct the graph explicitly, as it would be far to large for our choice of resolution $\resolution$.
	However, by identifying each vertex by its corresponding grid point, deciding whether there is an edge between two vertices or if a vertex belongs to a certain clique can be done by comparing integers up to size $\bigO{\resolution \sidelength}$, which can be done in $\bigO{\ln\left(\resolution \sidelength\right)} = \bigO{\vol{\volume} \ln \vol{\volume}}$.
	Choosing a clique from the clique cover can be done by choosing $d$ integers up to size $\bigO{\sidelength}$.
	Finally, for a given clique $\clique[i]$, we can sample from $\gibbsDistributionFunction{~\cdot~}[\hsGraph{\resolution}, \hsWeight{\resolution}][\clique[i]][\zeroSpinConfig[\hsVertices{\resolution} \setminus \clique[i]]] = \gibbsDistribution[\subgraph{\hsGraph{\resolution}}{\clique[i]}, \hsWeight{\resolution}]$ by
	\begin{enumerate}[(1)]
		\item \label{enum:hardsphere_inner_sampling:pos}
		sample $x \in \subgrid[i]$ uniformly at random, where $\subgrid[i]$ is the region of the grid corresponding to $\clique[i]$, and
	 	\item \label{enum:hardsphere_inner_sampling:reject}
	 	return $\emptyset$ with probability $\frac{1}{\partitionFunction[\subgraph{\hsGraph{\resolution}}{\clique_i}]}$ and $\{\hsVertex{x}\}$ otherwise.
	\end{enumerate}
 	Note that \ref{enum:hardsphere_inner_sampling:pos} involves sampling $d$ integers up to size $\bigO{a}$. In addition, step \ref{enum:hardsphere_inner_sampling:reject} involves computing $\partitionFunction[\subgraph{\hsGraph{\resolution}}{\clique_i}] = 1 + \size{\clique[i]} \resolution^{-d} \hsFugacity$.
 	Thus, sampling from $\gibbsDistribution[\subgraph{\hsGraph{\resolution}}{\clique[i]}, \hsWeight{\resolution}]$ can be done in $\bigO{\ln (a)} = \bigO{\vol{\volume} \ln \vol{\volume}}$.
 	
 	We now know that we can sample $\err_s$-approximately from $\gibbsDistribution[\hsGraph{\resolution}, \hsWeight{\resolution}]$ in time polynomial in $\vol{\volume}^{1/\delta^2}$ and $\ln(1/\err_s)$.
 	Applying \Cref{lemma:appx_partition_function} proves the theorem.
\end{proof}

\printbibliography

\newpage

\end{document}